\documentclass[11pt,a4paper]{article}

\usepackage{Format_english}
\usepackage{Macros}
\usepackage{makeidx}

\makeindex
\begin{document}

\title{Patching and multiplicities of $p$-adic eigenforms}
\author{Eugen Hellmann\footnote{Mathematisches Institut, Universit\"at
    M\"unster, Einsteinstrasse 62, D-48149 M\"unster, Germany
    \emph{e-mail :} e.hellmann@uni-muenster.de} \and Valentin
  Hernandez\footnote{Laboratoire de Mathématiques d'Orsay,
    Universit\'e Paris-Saclay \emph{e-mail :}
    valentin.hernandez@math.cnrs.fr} \and Benjamin
  Schraen\footnote{Universit\'e Claude Bernard Lyon 1, CNRS, \'Ecole
    Centrale de Lyon, INSA Lyon, Université Jean Monnet, I CJ UMR5208,
    69622 Villeurbanne, France. \emph{e-mail :}
    schraen@math.univ-lyon1.fr}} \date{}

\maketitle

\begin{abstract}
  We prove the existence of non-classical $p$-adic automorphic
  eigenforms associated to a classical system of eigenvalues on
  definite unitary groups in $3$ variables. These eigenforms are
  associated to Galois representations which are crystalline but very
  critical at $p$. We use patching techniques related to the
  trianguline variety of local Galois representations and its local
  model. The new input is a comparison of the coherent sheaves
  appearing in the patching process with coherent sheaves on the
  Grothendieck--Springer version of the Steinberg variety given by a
  functor constructed by Bezrukavnikov.
\end{abstract}

\tableofcontents
\section{Introduction}

The aim of this paper is to unravel (and explain) a new phenomenon in the theory of $p$-adic automorphic forms. 
Given a reductive group $\underline{G}$ over a number field (overconvergent) $p$-adic automorphic forms are $p$-adic avatars of automorphic forms on $\underline{G}$. We usually refer to the latter as classical automorphic forms in order to distinguish them from their $p$-adic limits. 
Additional structures on spaces of automorphic forms, such as the Hecke-action, naturally extend to the $L$-vector spaces of overconvergent $p$-adic automorphic forms $S^\dagger(K^p), S^\dagger_\kappa(K^p)$, where the field of coefficients $L$ is a finite extension of $\mathbb{Q}_p$ and $K^p\subset G(\mathbb{A}^p)$ is a compact open subgroup (referred to as the \emph{tame level}) and $\kappa$ is a \emph{weight}. 
A central question about $p$-adic automorphic forms is to clarify whether a given overconvergent $p$-adic automorphic form (of algebraic weight) that is an eigenform for the Hecke action is a classical automorphic form. 
Often this question can be answered in terms of the Hecke eigenvalues. Coleman's \emph{small slope implies classical} result \cite{Coleman} and generalizations thereof (see e.g.~\cite{Kassaeigluing}, \cite{Chefougere}, \cite{BPS}) asserts that this question can be purely decided using the Hecke action at $p$ if the $p$-adic valuation of the Hecke eigenvalues at $p$ is small compared to the weight. 
Beyond the \emph{numerically non critical slope} it is known that this fails. However, one can ask the same question taking into account the full Hecke action (as opposed to the Hecke action at $p$). 

Assume that we are in a situation where we can construct the Galois representation $\rho_f=\rho_\chi$ attached to a $p$-adic eigenform $f$, respectively to the Hecke character $\chi$ giving the system of Hecke eigenvalues of $f$. 
Then the Hecke action away from $p$ encodes all the information about the $p$-adic Galois representation $\rho_f$, including the $p$-adic Hodge theoretic information at places dividing $p$ (though this is encoded in a rather indirect and mysterious way).
The naive generalization of the classicality question about overconvergent $p$-adic automorphic forms can hence be phrased as follows (though we phrase the question in a rather informal way):

\noindent \textbf{Question A:}
Let $f$ be an overconvergent $p$-adic eigenform of dominant algebraic weight such that the corresponding Galois representation $\rho_f$ is de Rham at places dividing $p$. Is it true that $f$ is a classical automorphic form?

We note that a softer version of this question is the following expectation that is implied by the Fontaine--Mazur conjecture. Again we state the expectation in a rather informal way -- it might fail without more precise assumptions on the group the level, etc.~(see e.g.~\cite[Conj. 5.1.1]{BHS3} for a precise formulation).

\noindent \textbf{Rough Expectation B:}
Let $S^\dagger_\kappa(K^p)[\chi]\subset S^\dagger_\kappa(K^p)$ be an eigensystem (for the action of the full Hecke algebra $\mathbb{T}$ generated by Hecke operators at $p$ and away from $p$) in the space $S^\dagger_\kappa(K^p)$ of overconvergent $p$-adic automorphic forms of weight $\kappa$ on $\underline{G}$. Assume that $\kappa$ is dominant algebraic and that the Galois representation $\rho_\chi$ associated to the Hecke character $\chi:\mathbb{T}\rightarrow L$ is de Rham at places dividing $p$. Then $S^\dagger_\kappa(K^p)[\chi]$ contains a classical automorphic form, i.e.~its subspace $S^{\rm cl}_\kappa(K^p)[\chi]$ of classical forms is non-zero. 

Question $A$ then can be rephrased as the question whether $S^{\rm cl}_\kappa[\chi](K^p)=S^\dagger_\kappa(K^p)[\chi]$ in Expectation B. 
It is known that Question $A$ does not have an affirmative answer in general. Ludwig \cite{MR3846054} and Johansson--Ludwig \cite{johansson2023endoscopy} have shown that there are counterexamples for ${\rm SL}_2$. The reason for these counterexamples however, is of global (endoscopic) nature and it remains a reasonable question to ask Question A for groups where these phenomena do not apply, e.g.~for definite unitary groups. 

Expectation B has been verified for $\rm{GL}_2$ (this is basically \cite{Kisinoverconvergent}), and generalizations of Kisins' result were proven by  Bellaïche and his coauthors
(\cite{BCLisseEisenstein},\cite{BellaicheCritInvent} and \cite{BellaicheDimitrov}). For definite unitary groups, and under Taylor--Wiles assumptions, these results were vastly generalized in \cite{BHS2}, \cite{BHS3}. 
We point out that in the cases treated in \cite{BHS2} the results imply that $S^{\rm cl}_\kappa(K^p)[\chi]=S^\dagger_\kappa(K^p)[\chi]$, while the more general case in \cite{BHS3} only allows to construct some classical form in the eigensystem (though no counterexample to Question A is constructed in loc.~cit.).
The reason for this difference is due to a phenomenon in the geometry of eigenvarieties (i.e.~rigid analytic spaces parametrizing the systems of Hecke eigenvalues in the space of overconvergent $p$-adic automorphic forms of finite slope), respectively in the geometry of their local Galois-theoretic counterparts (the so-called  trianguline variety of \cite{BHS1}). In the case treated in \cite{BHS2} the trianguline variety is smooth at the Galois representations in question (and hence the eigenvariety is local complete intersection). In general the trianguline variety is not smooth, and as a consequence one can construct non-smooth points on the corresponding eigenvarieties, see \cite[Thm. 5.4.2]{BHS3}. It is this failure of smoothness that prevents \cite{BHS3} from identifying $S^{\rm cl}_\kappa(K^p)[\chi]$ and $S^\dagger_\kappa(K^p)[\chi]$.

In this paper we prove that the answer to Question A is \emph{no} for definite unitary groups in three variables (see Theorem \ref{thm:main} below for a more precise formulation). 
\begin{theor}
There exists a unitary group in three variables $U$, a tame level $K^p$, a dominant algebraic weight $\kappa$ and a Hecke character $\chi:\mathbb{T}\rightarrow L$ that occurs in the space $S^\dagger_{\kappa}(K^p)_{\rm fs}$ of overconvergent automorphic forms of finite slope and weight $\kappa$ such that the eigenspace $S^\dagger_\kappa(K^p)[\chi]$ contains classical as well as non-classical eigenforms. 
\end{theor}
The construction of this example also clarifies the role of the singularities of the trianguline variety $X_{\rm tri}$. The precise results we prove suggest that the answer to Question A is no, whenever the dualizing sheaf $\omega_{X_{\rm tri}}$ is not locally free at the point defined by $\rho$ (and the refinement associated to $\chi$), i.e.~whenever $X_{\rm tri}$ is non-Gorenstein at this point (we refer to Theorem \ref{theo:Verma} below for the precise link with $\omega_{X_{\rm tri}}$). In the three dimensional case, this results in a precise comparison of the dimensions of the eigenspaces $S^{\rm cl}_\kappa(K^p)[\chi]\subset S^\dagger_\kappa(K^p)[\chi]$.

We point out that, in contrast to \cite{MR3846054} and \cite{johansson2023endoscopy} this is a purely local \emph{$p$-adic} phenomenon. Moreover, the theorem implies that the usual invariants (i.e.~the Hecke action, respectively the $p$-adic Hodge theoretic information of the associated Galois representation) can not  distinguished between classical and non-classical forms.
We like to refer to the non-classical forms in such eigensystems as \emph{undercover} automorphic forms.

The main result, and in particular the occurrence of the dualizing sheaf $\omega_{X_{\rm tri}}$ therein, is inspired by the categorical point of view in the $p$-adic Langlands program, see \cite{EGH}. The space of overconvergent $p$-adic automorphic forms of finite slope $S^\dagger(K^p)_{\rm fs}$ can be viewed as the topological dual of the global sections of a coherent sheaf (that we simply refer to as the \emph{sheaf of $p$-adic automorphic forms}) on the rigid analytic generic fiber of the universal deformation space of Galois representations (more precisely, on the product of this space with the space of continuous characters of a maximal torus $\underline{T}(\QQ_p)\subset \underline{G}(\mathbb{Q}_p)$ at $p$). The support of this sheaf is, by definition, the corresponding eigenvariety. The local-global-compatibility conjectures \cite[Conj. 9.6.8 and Conj. 9.6.16]{EGH} give a precise description of this sheaf in terms of the geometry of moduli stacks of $(\varphi,\Gamma)$-modules (that are closely related to the trianguline variety). 
More precisely, the categorical approach to the $p$-adic Langlands program asks for a functor from certain (locally analytic) representations of $\underline{G}(\mathbb{Q}_p)$ to sheaves on stacks of $(\varphi,\Gamma)$-modules, and the sheaf of $p$-adic automorphic forms is the globalization of the evaluation of this functor on a specific representation.
One of the punchlines of \cite{EGH} (see section 1.6 therein for a more detailed discussion) is that avatars of the envisioned functor have been around in number theory during the past decades in the context of the Taylor--Wiles patching method, in particular \emph{patching functors} as used for example in \cite{EGS} (or also in \cite[5.]{BHS3})
A crucial point in the proof of the main theorem is the identification of such a patching functor with an explicit local functor, see Theorem \ref{thm:explsheaves} below. This partially confirms expectations in the categorical picture, see \cite[Expectation 6.2.27]{EGH}.

Note that the multiplicity result in Theorem \ref{thm:main} has some striking consequence for the $p$-adic Langlands Program for $\GL_3(\QQ_p)$. It implies that the locally analytic representation of $\GL_3(\QQ_p)$ on the Hecke eigenspace of overconvergent $p$-adic modular forms over $\underline{G}$ corresponding to a Galois representation $\rho$ as in Theorem \ref{thm:main} contains locally analytic vectors which are \emph{not} in the socle of the representation (see Remark \ref{rema:loc_alg_not_in_socle}). After finishing this works, the authors learned that Ding also proved examples of this penomena for generic Galois representations (see \cite{Ding_loc_alg}).

We now describe our results in more detail. Let $F$ be a totally real
number field and let $E/F$ be a CM (imaginary) quadratic extension in
which every place $v | p$ in $F$ splits in $E$. Let $U$ be a unitary
group (over $\QQ$) in $n$ variables for the quadratic extension $E/F$
which is compact at infinity. By the hypothesis on $p$ the group
$U_{\QQ_p}$ is a product of general linear groups over finite
extensions of $\QQ_p$ and we denote $\underline T$ a maximal torus of
$U_{\QQ_p}$. We also fix a finite extension $L/\QQ_p$ which is big
enough to split $E$. Let $\cO_L\subset L$ be its ring of integers, $\pi_L$ a uniformizer and
$k_L$ its residue field.

For any continuous
character $\delta : \underline T(\QQ_p) \fleche
L^\times$, we can define a weight $\kappa$ (which is given by the derivative of $\delta$ at $1$) and a character of the Atkin--Lehner ring $\mathcal{A}(p)$ (the ring of Hecke-operators at $p$, see Definition \ref{def:Hecke_action}) that we still denote by $\delta$. 
We will assume that $\delta_{|T^0}$ is algebraic where $T^0 \subset
\underline T(\QQ_p)$ is the maximal compact subgroup.
Let $K^p\subset U(\mathbb{A}^p)$ be a tame level and let $S$ be a finite set, containing places above $p$, away from which
$K^p$ is hyperspecial. We write $\mathbb{T}^S$ for the unramified Hecke algebra at places not in $S$ and $\mathbb{T}=\mathbb{T}^S\otimes_{\ZZ} \mathcal{A}(p)$. 
Associated to these data we consider the spaces $S_\kappa^\dagger(K^p)$ and $S_\kappa^{\rm cl}(K^p)$, see Definition \ref{def:classical} for the precise definition, which come equipped with an action of $\mathbb{T}^S$ and $\mathcal{A}(p)$. 

Given a character $\chi^S:\mathbb{T}^S\rightarrow L$ let $\chi=\chi^S\otimes \delta$ and consider the eigenspaces $S_\kappa^\dagger(K^p)[\chi]$ and $S_\kappa^{\rm cl}(K^p)[\chi]$. We note that the classical subspace $S_\kappa^{\rm cl}(K^p)[\chi]$ is zero unless $\kappa$ is dominant algebraic.
To an eigenvector $f \in S_\kappa^\dag(K^p)[\chi]$ we can associate a
Galois representation $\rho =\rho_f=\rho_{\chi}: \Gal_E \coloneqq \Gal(\overline{E}/E)
\fleche \GL_n(\overline\QQ_p)$. 
For the precise form of the main result we introduce the following (strong) Taylor--Wiles
hypothesis. Let $\overline\rho : {\rm Gal}_E \fleche \GL_n(k_L)$ be the semisimplification of the
reduction modulo the maximal ideal of $\cO_L$ of $\rho$. We
assume that (see Hypothesis \ref{hyp:TWtext} in the text)
\begin{equation}\label{hyp:TW}
\left\{
\begin{array}{ll}
\bullet & p > 2, \\ \bullet & E/F \text{ is unramified and } \zeta_p
\notin E, \\ \bullet & U\text{ is quasi-split at all finite places of
}F, \\ \bullet & \text{if a place $v$ of $F$ is inert in $E$, then
  $K_v$ is hyperspecial},\\ \bullet & \overline\rho \text{ is
  absolutely irreducible and } \overline\rho(\Gal_{E(\zeta_p)}) \text{ is
  adequate}.
\end{array}\right.
\end{equation}

For simplicity of the exposition we assume now that that $p$ is totally split in $F$ (in the core of the paper we work in the general case). If the
representation $\rho$ is crystalline at $v | p$, it can be described
by its associated \emph{filtered isocrystal} which is a finite
dimensional $L$-vector space $D_{\cris}(\rho_v)$ endowed
with a linear automorphism $\varphi\in\GL(D_{\cris}(\rho_v))$ and a
complete flag $D^\bullet$, called the Hodge--Tate filtration (in our
case, this is a complete flag as $\rho_v$ has necessarily regular Hodge--Tate
weights). We say that $\rho_v$ is \emph{$\varphi$-generic} if the
ratio of two of its eigenvalues is not in $\set{1,p}$. In this case
the character $\delta$ determines an order of the eigenvalues of
$\varphi$ (that is called a \emph{refinement} of $\rho_v$) which in turn (using the fact that the $\varphi$-eigenvalues are pairwise distinct) defines another complete flag $\mathcal F_\bullet$ on
$D_{\cris}(\rho_v)$ which is $\varphi$-stable. We denote
$w_{\rho,\delta,v} \in \mathfrak S_n$ the relative position of the flags
$\mathcal F_\bullet$ and $D^\bullet$ in the flag variety of
$D_{\cris}(\rho_v)$. When $w_{\rho,\delta,v}= w_0$ is the longest element
of $\mathfrak S_n$, i.e.~when the two flags $D^\bullet$ and $\mathcal
F_\bullet$ are in generic position, we say that $f$ is
\emph{non-critical} at $v$. The ``most critical case'' is the case
where $w_{\rho,\delta,v} = 1$, i.e.~when the two flags coincides. In this case we say that $f$ is \emph{very critical} at $v$. 

\begin{theor}
\label{thm:main} Assume $n=3$. 
Let $\delta: \underline{T}(\mathbb{Q}_p)\rightarrow L^\times$ be a continuous character of weight $\kappa$ dominant algebraic. Let $\chi^S:\mathbb{T}^S\rightarrow L$ be a character and let $\chi=\chi^S\otimes\delta$. We assume that the eigenspace $S^\dagger_\kappa(K^p)[\chi]$ is non-zero and that for any $v|p$ the local Galois representation $\rho_v=\rho_\chi|_{\Gal_{E_v}} : \Gal_{E_v} \fleche \GL_3(\overline{\QQ_p})$ is crystalline with distinct Hodge-Tate weights and is $\varphi$-generic. Assume moreover that the Taylor--Wiles hypothesis (\ref{hyp:TW}) is satisfied. 
Let $r$ be the number of places $v|p$ in $F$ such that $w_{\rho_\chi,\delta,v}=1$. Then
\[ \dim S_\kappa^\dag(K^p)[\chi] = 2^r\dim S_\kappa^{\rm cl}(K^p)[\chi].\]
\end{theor}
We refer to Corollary \ref{cor:nonclassicalforms} for a more general statement where $p$ is not necessarily totally split in $F$.

Theorem \ref{thm:main} would be vacuous without proving the existence of characters $\chi$ and $\delta$ (and a group $U$ and a tame level $K^p$) such that the corresponding eigenspace $S_\kappa^{\rm cl}(K^p)[\chi]$ is non-zero and consists of very critical forms.
 As there exist only countably many classical automorphic
forms, but uncountably many flags it doesn't seem very easy to
construct an $f$ with $w_{\rho_f,\delta} =1$. 
This is Corollary \ref{cor:exofverycriticalforms}, the main result of section \ref{sect:critforms}, which uses global automorphic methods that are rather disjoint from the methods of the other parts of the paper. The Galois representation corresponding to the constructed Hecke character is induced from a degree 3 extension of $E$.

We finally discuss the relation of these results with patching functors and the categorical approach to a $p$-adic Langlands correspondence. 
Assume that $\delta=\delta_\lambda\delta_{\mathcal R}^{\sm}$ is the product of a dominant
algebraic character $\delta_\lambda$ and a smooth unramified character
$\delta_{\mathcal R}^{\sm}$ (which is in fact implied by the assumption that $\rho_v$ is crystalline). As the notation suggests,
the character $\delta_{\mathcal R}^{\sm}$ corresponds to the choice of a refinement $\mathcal R$ of $\rho_p \coloneqq (\rho_v)_{v | p}$. 
Let $\mathcal X_{\rho_p}=\Spec(R_{\rho_p})$
be the scheme associated to the universal deformation ring of
$\rho_p$. Using results of \cite{BHS3}, we can construct a subscheme
\[\mathcal X_{\rho_p,\mathcal R}^{\qtri}=\Spec(R_{\rho_p,\mathcal{R}}^{\qtri})\subset \mathcal X_{\rho_p}\] of ``quasi-trianguline''
deformations of $\rho_p$ associated to the refinement $\mathcal
R$. By loc.~cit.~this scheme has a local model modeled on the Steinberg
variety (or rather its ``Grothendieck--Springer'' variant) and its irreducible components $\mathcal X_{\rho,\mathcal R}^{\qtri,w}$ are labeled by the Weyl group $W$ of $\prod_{v | p} \GL_3$. It is known that these irreducible components are normal and Cohen--Macaulay.

Let's denote $\lambda=\delta|_{T^0}(=\delta_{\lambda | T^0})$, this is a dominant algebraic character. 
Using hypothesis (\ref{hyp:TW}) the Taylor--Wiles method, as extended to the setting of completed cohomology in \cite{CEGGPS},
can be used (\cite[5.]{BHS3}) to construct coherent sheaves $\mathcal M_{\infty}(L(\lambda))$ and
$\mathcal M_{\infty}(M(w\cdot\lambda))$ for $w\in W$
over $\mathcal X_{\infty,\mathcal \rho,\mathcal R}^{\qtri}=\Spec(R_{\rho_p,\mathcal R}^{\qtri}[[x_1,\dots,x_g]])$ for some $g\geq 0$, 
that ``patch'' the duals of the spaces of classical, respectively $p$-adic, automorphic forms. More precisely
\begin{align*}
\mathcal M_{\infty}(L(\lambda))\otimes k(\rho_p)&={\rm Hom}_L(S^{\rm cl}_\lambda(K^p)[\chi],L),\\
\mathcal M_{\infty}(M(w\cdot\lambda))\otimes k(\rho_p)&={\rm Hom}_L(S^\dagger_{w\cdot \lambda}(K^p)[\chi],L).
\end{align*}
These coherent sheaves are in a certain precise sense associated to the $U(\mathfrak{g})$-modules $L(\lambda)$ (the algebraic representation of highest weight $\lambda$) respectively the Verma modules $M(w\cdot\lambda)$, where $\mathfrak{g}$ is the Lie algebra of $U_L\cong \prod_{v|p}{\rm GL}_3$.  
The results of \cite{BHS3} show that the coherent
sheaves $\mathcal{M}_\infty(M(w\cdot \lambda))$ have generic rank (when nonzero) equal to
$\dim_L S_\lambda^{\rm cl}(K^p)[\chi]$. Denote $\mathcal X_{\infty,\rho,\mathcal R}^{\qtri,w} \coloneqq \mathcal X_{\infty,\rho,\mathcal R}^{\qtri} \times_{\mathcal X_{\rho,\mathcal R}^{\qtri}} \mathcal X_{\rho,\mathcal R}^{\qtri,w}$. The key to the
proof of Theorem \ref{thm:main} is the following result:
\begin{theor}\label{theo:Verma}
Under the assumptions of Theorem \ref{thm:main}, let $m=\dim_L S_\lambda^{\rm cl}[\chi]$. 
For any $w\in W$, there is an isomorphism
\[\mathcal M_\infty(M(w\cdot\lambda))\cong \omega_{\overline{\mathcal X}_{\infty,\rho,\mathcal R}^{\qtri,ww_0}}^{\oplus m}.\]
Here $\omega_{\overline{\mathcal X}_{\infty,\rho\mathcal R}^{\qtri,ww_0}}$ is the dualizing sheaf of a complete intersection
 $\overline{\mathcal X}_{\infty,\rho,\mathcal R}^{\qtri,ww_0}\subset \mathcal X_{\infty,\rho,\mathcal R}^{\qtri,ww_0}$.
\end{theor}

In order to prove Theorem \ref{theo:Verma}, we extend
$\mathcal M_\infty$ to a functor on the whole category $\mathcal O_\lambda$, the
block of the BGG category $\mathcal{O}$ containing
$L(\lambda)$. This is the \emph{patching functor} alluded to above. 
More precisely, assuming that $\rho_p$ is crystalline with
regular Hodge-Tate weights, and $\delta$ is $\varphi$-generic, we
construct an exact functor
\[\mathcal M_\infty : \mathcal O_\lambda \fleche \Coh(\mathcal
  X_{\infty,\rho,\mathcal R}^{\qtri}),\] such that, for every
$M \in \mathcal O_\lambda$ the sheaf $\mathcal M_\infty(M)$ is Cohen--Macaulay of
the expected dimension. 

In spirit of the categorical approach to the $p$-adic Langlands correspondence the functor $\mathcal{M}_\infty$ should be a ``local'' functor, that is (up to multiplicities coming from contributions at the places away from $p$) the functor $\mathcal{M}_\infty$ should be the pullback, denoted $\mathcal B_\infty$, of a functor
\[ \mathcal B_p : \mathcal O_\lambda \fleche \Coh(\mathcal
  X_{\rho_p,\mathcal R}^{\qtri}).\]
This functor $\mathcal{B}_p$ can be written down explicitly using the local model for $\mathcal
  X_{\rho_p,\mathcal R}^{\qtri}$ and a functor constructed by Bezrukavnikov \cite{BezTwo}, see \ref{sec: Bez} for details.  
Our main local result compares $\mathcal M_\infty$ and $\mathcal B_p$ (see Corollary \ref{cor:functors_isom} for
the general version):

\begin{theor}\label{thm:explsheaves} 
Under the assumptions of Theorem \ref{thm:main}, let $m=\dim_L S_\lambda^{\rm cl}[\chi]$. 
 Then there is an isomorphism of functors $\mathcal
  M_\infty\simeq\mathcal B_\infty^{\oplus m}$. 
  As a consequence, we have
\begin{enumerate}[1)]
\item  \label{case:1} for all $w \in W$, $\mathcal M_\infty(M(w \cdot
  \lambda)^\vee) \simeq \mathcal O_{\overline{\mathcal
      X}^{\qtri,ww_0}_{\infty,\rho,\mathcal R}}^{\oplus m}$ ;
\item \label{case:3} for all $w\in W$, $\mathcal M_\infty(M(w \cdot
  \lambda)) \simeq \omega_{\overline{\mathcal
      X}^{\qtri,ww_0}_{\infty,\rho,\mathcal R}}^{\oplus m}$ ;
\item  \label{case:4} for all $M \in \mathcal O$, we have $\mathcal M_\infty(M^\vee) \simeq \mathcal M_\infty(M)^\vee$ where 
$(\cdot)^\vee$ denote both the dual in $\mathcal O_\lambda$ and the
Serre dual in the category of coherent sheaves.
\end{enumerate}
\end{theor}
\begin{rema}
We can only prove Theorem \ref{thm:explsheaves} in the three dimensional case. However, we expect an isomorphism $\mathcal{M}_\infty\cong \mathcal{B}_\infty^{\oplus m}$ for higher dimensional definite unitary groups as well.
\end{rema}

 In fact $\mathcal{B}_p$ should factor through the category of locally analytic representations, and is expected to extend to a functor with values in coherent sheaves on the stack of all $(\varphi,\Gamma)$-modules (compare \cite[Conjecture 6.2.4 and Expectation 6.2.27]{EGH}). Theorem \ref{thm:explsheaves} should be viewed as some partial evidence for these expectations.

The key to proving Theorem
\ref{thm:explsheaves} is to extend the functor
$\mathcal M_\infty$ to a larger category $\mathcal O_\alg^\infty$ and
to a deformation $\widetilde\cO_\alg$ as introduced in
\cite{SoergelHC}, which we think of as a \emph{deformed version} of
$\mathcal O_\alg$. We would like to emphasize that we first prove
\ref{case:1} and we deduce the
isomorphism $\mathcal M_\infty\simeq\mathcal B_\infty^{\oplus m}$ from this in a second time . The proof of
\ref{case:1} is based on a d\'evissage whose has its origin in the
paper \cite{EGS}. We first prove the result in the case where
$\mathcal X_{\infty,\mathcal R}^{\qtri,w}$ is smooth and then proceed
inductively.  Note that the \emph{existence} of Bezrukavnikov's
functor $\mathcal B_\infty$ plays a key role in this induction.
The second main input into this induction is the computation of
$\mathcal M_\infty(M_I(w\cdot\lambda))$ where $M_I(w\cdot\lambda)$ is
a generalized Verma module (corresponding to some parabolic $P_I$).
These sheaves, that are related to sheaves of $p$-adic automorphic forms on the partial eigenvarieties constructed by Wu \cite{Wu2}, 
are supported on ``partially de Rham quasi-trianguline''
deformation spaces $\mathcal X_{\rho_p,\mathcal R}^{I-\qtri}$ which
have been studied by Breuil and Ding in \cite{BreuilDing}.

We finally note that the component
$\mathcal X^{\qtri,w_0}_{\infty,\rho,\mathcal R}$ is not Gorenstein and its
dualizing sheaf has a $2^r$-dimensional fiber at $\rho_p$, which is the reason for the factor $2^r$ in Theorem
\ref{thm:main}.

We now describe the content of the article. In section \ref{sec:cat_O} we introduce the category $\mathcal O_\alg$ and its deformed versions. Section \ref{sec:emert-jacq-funct} studies Emerton's Jacquet functor and gives the abstract framework to construct patching functors. In section \ref{sec:quasi-triang-local}, we recall the quasi-trianguline deformation spaces of \cite{BHS3}, their local models, and their parabolic version (\cite{BreuilDing,Wu2}). Section \ref{sec:global} recalls the definitions of the global objects like completed cohomology, overconvergent automorphic forms and their patched versions. Section \ref{sec:patching_functors} is devoted to the further study of the functor $\mathcal M_\infty$ and its factorization through $\mathcal X_{\infty,\rho,\mathcal R}^{\qtri}$, the (global) quasi-trianguline deformation space. In section \ref{sec:localcomp}, we study the supports of the sheaves $\mathcal M_\infty(M)$ for specific objects of $\mathcal O_\alg$ (and their deformed version), and we recall results on Bezrukavnikov's functor before deducing Theorem \ref{thm:explsheaves} (in the three dimensional case). Finally, in section \ref{sect:critforms} we explain how to explicitly construct very critical forms satisfying the assumptions in Theorem \ref{thm:main} for $n=3$.

\subsection*{Acknowledgements}

We are very grateful to Roman Bezrukavnikov, George Boxer, Christophe
Breuil, Matthew Emerton, Toby Gee, Arthur-Cesar Le Bras, Simon Riche and Olivier Schiffmann for
discussions related to this work. We also would like to thank the team
maintaining the software Macaulay2 \cite{M2} which has been very
useful to make preliminary computations related to this work.
E.H.~was supported by Germany’s Excellence Strategy EXC 2044-390685587 ``Mathematics M\"unster: Dynamics–Geometry–Structure'' and by the CRC 1442 ``Geometry: Deformations and Rigidity'' of the DFG. Part of this work was done during the Trimester Program 2023 on the arithmetic of the Langlands Program in HIM Bonn where the three authors were invited.

\subsection*{Notations}

Let $p$ be a prime number. When $K$ is a field, we write $\mathrm{Gal}_K = \Gal(K^{sep}/K)$ for its absolute Galois group.
We fix $L$ a finite extension of $\QQ_p$ which will be chosen sufficiently large in the text. 

\section{Variants of the BGG-category $\mathcal O$}
\label{sec:cat_O}

In this section, we fix $L$ to be a field of characteristic $0$. Let
$\underline{G}$ be a split reductive group over $L$. Let
$\underline{B}$ be a Borel subgroup, $\underline{T}$ a maximal split
torus of $\underline{G}$ contained in $\underline{B}$ and
$\underline{N}$ the radical of $\underline{B}$. We use the notation
$\frakg$, $\frakb$, $\mathfrak t$, $\mathfrak n$... for the Lie
algebras of $\underline{G}$, $\underline{B}$, $\underline{T}$,
$\underline{N}$... We denote by $X^*(\underline T)$ the finite free
abelian group $\Hom(\underline T,\mathbb{G}_{m,L})$ of characters of
$\underline T$. This abelian group can be identified with a $\ZZ$-lattice in $\mathfrak t^* := \Hom_L(\mathfrak t,L)$. For $\lambda\in X^*(\underline T)$, we also write
$\lambda$ for the character of $\mathfrak t$ induced by $\lambda$. Let
$\Phi$ be the set of roots of the pair $(\underline G,\underline T)$
and let $\Phi^+\subset\Phi$ be the subset of positive roots with
respect to $\underline B$ and $\Delta\subset\Phi^+$ the subset of
simple roots. As usual we write $\delta_G\in X^*(\underline T)\otimes_{\ZZ}\QQ$ for the
half sum of positive roots. Let $W$ be the Weyl group of
$(\underline G,\underline T)$. For $w\in W$, we write
$\lambda\mapsto w\cdot\lambda$ for the \emph{dot action} of $W$ on
$X^*(\underline T)$ (with respect to $\underline B$, that is
$w\cdot\lambda\coloneqq w(\lambda+\delta_G)-\delta_G$). We equip $W$ with the
Bruhat order corresponding to the choice of $\underline B$ and we
denote $w_0\in W$ the longest element for this order.

If $I\subset\Delta$ is a subset of simple roots, we denote by
$\Phi_I\subset\Phi$ the subset of roots which are sums of elements of
$I$ and $\underline P_I\supset\underline B$ be the standard parabolic
subgroup of $\underline G$ such that
$\frakp_I=\frakb+\sum_{\alpha\in\Phi_I}\frakg_\alpha$. Let
$\underline L_I$ be the standard Levi subgroup of $\underline P_I$ and
$\underline Z_I$ be the center of $\underline L_I$. We say that a
character $\lambda\in X^*(\underline T)$ is \emph{dominant with
  respect to $\underline P _I$} if $\scalar{\lambda,\alpha^\vee}\geq0$
for $\alpha\in I$ and we denote $X^*(\underline T)_I^+$ the set of
such characters. When $I=\Delta$, we have
$\underline P_\Delta=\underline G$ and we write
$X^*(\underline T)^+=X^*(\underline T)^+_{\Delta}$. We use the following
order relation on $X^*(\underline T)$, we say that $\lambda\geq\mu$ if
and only if $\lambda-\mu\in \sum_{\alpha\in\Phi^+}\NN\alpha$.

We write $W_I$ for the
Weyl group of the Levi $\underline L_I$ of $\underline P_I$; it is the subgroup of $W$ generated
by the simple reflexions $s_\alpha$ for $\alpha \in I$. Given
$w \in W$, we denote $w^{\rm min}$ (resp.~$w^{\rm max}$) the unique minimal
(resp.~maximal) element for the Bruhat order having the same class as $w$
in $W_I\backslash W$. This definition depends on $I$ (and on the fact
that the quotient is on the left) but we hope our notation will cause
no confusion. As usual, we write $w_0 \in W$ for the longest element in $W$.
Then we have $(ww_0)^{min} = w^{max}w_0$ and
$(ww_0)^{max} = w^{min}w_0$ for any $w\in W$ . Finally, we write ${}^IW$ for the set of minimal length representatives of $W_I\backslash W$ in $W$.

If $\mathfrak h$ is a Lie algebra we note $\mathfrak h^{\mathrm{ss}}$ its
derived Lie algebra.

\subsection{Recollections}
\label{sec:BGG}

For $I\subset\Delta$, we consider the full subcategory $\mathcal{O}^{I,\infty}$ of the category $U(\frakg)\text{-}{\rm mod}$ of 
$U(\frakg)$-modules that consists of all finitely generated $U(\frakg)$-modules $M$ such that
\begin{itemize}
\item for any $m\in M$, the $L$-vector space $U(\frakp_I)m$ is finite
  dimensional;
\item for any $h\in\mathfrak t$ and any $h$-stable finite dimensional
  $L$-vector subspace $V\subset M$, the characteristic polynomial of
  $h_{|V}$ is split in $L[X]$.
\end{itemize}
 This
is the category $\cO^{\frakp_I,\infty}$ in
\cite[\S3.1]{Agrawal_Strauch}.

For $\mu\in\Hom_L(\mathfrak t,L)$, we write $M^\mu\subset M$ for the
$L$-subspace of those $v\in M$ such that, for any $h\in\mathfrak t$,
$(h-\mu(h))^n\cdot v=0$ for some $n\geq1$. We have
\[ M=\bigoplus_{\mu\in\Hom_L(\mathfrak t,L)}M^\mu. \] We write
$\cO_\alg^{I,\infty}$ for the full subcategory of $\cO^{I,\infty}$ whose
objects $M$ satisfy $M^\mu=0$ for $\mu\notin X^*(\underline T)$.

Moreover, we write $\cO_{\alg}^I\subset\cO_{\alg}^{I,\infty}$ for the full
subcategory whose objects are direct sums of finitely generated
semisimple $U(\mathfrak l_I)$-modules (when seen as $U(\mathfrak l_I)$-modules). This coincides with the usual parabolic (algebraic) category $\mathcal O$, which is denoted
$\cO_\alg^{\frakp_I}$ in \cite{OSJH}). When $I=\emptyset$ we simply
use the notations $\cO_\alg^\infty$ and $\cO_\alg$ for
$\cO_{\alg}^{\emptyset,\infty}$ and $\cO_\alg^\emptyset$. Note that
$\cO_\alg^{I,\infty}\subset\cO_\alg^\infty$ for any $I\subset\Delta$. 
As these categories depend on the choices of $\mathfrak{g}$ and $\mathfrak{b}$ we write $\mathcal O^{\mathfrak
  g,\mathfrak b}$ (with additional decorations) instead of $\mathcal
O$, when the context is unclear.

These categories are stable by subobject and quotients in the category
of $U(\frakg)$-modules. Moreover the category $\cO_{\alg}^{I,\infty}$
is stable under extensions.

For any character $\lambda\in X^*(\underline T)_I^+$, we write
$L_I(\lambda)$ for the simple $U(\mathfrak l_I)$-module of highest weight
$\lambda$. This is a finite dimensional $L$-vector space and we define the \emph{generalized Verma module} of highest weight
$\lambda$ as \[M_I(\lambda) \coloneqq U(\mathfrak g) \otimes_{U(\frakp_I)}
L_I(\lambda).\] The generalized Verma module is an object of $\cO_{\alg}^I$ and has a unique
simple quotient $L(\lambda)$. When $I=\emptyset$, we simply write
$M(\lambda)=M_\emptyset(\lambda)$ and say that $M(\lambda)$ is a
\emph{Verma module}. We also denote by $P(\lambda)$ the projective cover of
the simple module $L(\lambda)$. If $\lambda$ is dominant with respect to $\underline B$, we call 
$P(w_0 \cdot \lambda)$ the \emph{antidominant 
projective} (with respect to $\lambda$).

\subsection{Nilpotent action of $U(\mathfrak t)$}
\label{sec:nilp-acti-umathfr}

Given $I\subset\Delta$ we denote by $\mathfrak{m}_I$ the augmentation
ideal of $U(\mathfrak{z}_I)$ and set
\begin{align*}
A_I&\coloneqq U(\mathfrak{z}_I)_{\mathfrak{m}_I}\\
A \coloneqq A_\emptyset &\coloneqq U(\mathfrak t)_{\mathfrak{m}}.
\end{align*} 
The canonical Lie algebra
decomposition
$\mathfrak{l}_I=\mathfrak{z}_I\oplus\mathfrak{l}_I^{\mathrm{ss}}$ defines a canonical morphism of Lie algebras
$p_I : \mathfrak{l}_I\twoheadrightarrow\mathfrak{z}_I$ which extends to a morphism
$U(\mathfrak l_I)\twoheadrightarrow U(\mathfrak z_I)$ of $L$-algebras also denoted by $p_I$. 
This morphism induces a surjective morphism $A\twoheadrightarrow A_I$ of $A_I$-algebras.

We show that the category $\mathcal{O}^{I,\infty}$ naturally embeds into the category $U(\frakg)_{A_I}\text{-}{\rm mod}$, where $U(\mathfrak{g})_{A_I}\coloneqq U(\frakg)\otimes_LA_I$.

Let $M$ be an object of the category $\cO^{I,\infty}$. Let $h\in \mathfrak{t}$. For $v\in M$ the element $h$ defines an $L$-linear endomorphism of the finite dimensional $L$-vector space $U(\mathfrak t)v$ and we write $h=D_{h,v}+N_{h,v}$ for its Jordan decomposition with semisimple part $D_{h,v}$ and nilpotent part $N_{h,v}$. As $M$ is locally $U(\mathfrak t)$-finite, uniqueness of the Jordan decomposition implies that these endomorphisms ``glue'' into
an endomorphism $D_h$ and a locally nilpotent endomorphism $N_h$ of $M$ such that $D_{h,v}$ respectively $N_{h,v}$ is
the restriction of $D_h$ respectively $N_h$ to $U(\mathfrak t)v$ for any $v\in M$.

\begin{lemma}\label{lemm:nilp_action}
  The endomorphism $N_h$ is $U(\frakg)$-equivariant.
\end{lemma}

\begin{proof}
  By construction $N_h$ and $D_h$ commute with the action of
  $\mathfrak t$ and stabilize each $M^\mu$. Let $\alpha\in\Phi$ and $x\in\frakg_\alpha$. For $v\in M^\mu$, we have
  $x\cdot v\in M^{\mu+\alpha}$ and $[h,x]=\alpha(h)x$ so that
  \[ D_hx\cdot v+N_hx\cdot v=xD_h\cdot v+xN_h\cdot v+\alpha(h) xv. \]
  By definition of $M^\mu$, we have $D_h\cdot v=\mu(h)v$ for any
  $v\in M^\mu$. This implies $D_hx\cdot v=(\mu(h)+\alpha(h))x\cdot v$ and
  $xD_h\cdot v=\mu(h)x\cdot v$. Therefore $N_hx\cdot v=xN_h\cdot
  v$. We conclude that $N_h$ commutes with the endomorphism of $M$
  induced by $x$. Therefore $N_h$ is $U(\frakg)$-equivariant.
\end{proof}

Given $M\in\mathcal{O}^{I,\infty}$ Lemma \ref{lemm:nilp_action} implies that we can define an $U(\mathfrak{t})$-module structure on $M$ by letting $h\in\mathfrak{t}\subset U(\mathfrak{t})$ act via $N_h$. As the action of each $h$ on $M$ is locally nilpotent, this action extends to an $A$-module structure.  
\begin{lemma}
Let $M$ be an object of $\mathcal{O}^{I,\infty}$, then the $A$-action on $M$ factors through $A_I$. Moreover, this $A_I$-module structure makes $\mathcal{O}^{I,\infty}$ into a full subcategory of $U(\mathfrak{g})_{A_I}\text{-}{\rm mod}$.
\end{lemma}
\begin{proof}
In order to prove that the $A$-action factors through $A\rightarrow A_I$ it is enough to prove that for $h\in \mathfrak{t}\cap \mathfrak{l}^{\rm ss}$ the endomorphism $N_h$ is zero. This is a direct consequence of the fact that $\mathfrak{l}^{\rm ss}$ is a semi-simple Lie algebra and that the $L$-vector space $U(\mathfrak{l}^{\rm ss})v$ is finite dimensional for any $v\in M$ (by definition of $\mathcal{O}^{I,\infty}$). As the $U(\mathfrak{g})$-action commutes with the $A$-action by Lemma \ref{lemm:nilp_action} the module $M$ is an $U(\mathfrak{g})_{A_I}$-module. Finally we note that, given $h\in\mathfrak{t}$, the construction of $N_h$ is functorial in $M$. 
\end{proof}
\begin{rema}\label{rema: second version of Mmu}
Let  $M\in\mathcal{O}^{I,\infty}$ and $\mu \in {\rm Hom}_L(\mathfrak{t},L)$ then the above construction implies that 
\[M^\mu=\{v\in M\mid hv=(\mu(h)v)+p_I(v))v\ \forall h\in\mathfrak{t}\}\]
\end{rema}

Let $M\in \mathcal{O}_\alg^{\infty}$. Lemma \ref{lemm:nilp_action} also implies that we can define another
structure of an $U(\frakg)$-module on $M$ where an element
$h\in\mathfrak t$ acts through the semisimple part $D_h$ and the
action of an element $x\in\frakg_\alpha$ for $\alpha\in\Phi$ is not
modified. We denote this $U(\mathfrak{g})$-module structure by $M^{\mathrm{ss}}$. Then $M^{\mathrm{ss}}$ is an object of $\cO_\alg$ and
\cite[Lemm.~3.2]{OSJH} implies that there is a unique structure of
algebraic $\underline{B}$-module on $M$ lifting the structure of
$U(\frakb)$-module on $M^{\mathrm{ss}}$. This $\underline{B}$-action is compatible with the original $U(\mathfrak{g})$-module structure on $M$ in the following sense:

\begin{lemma}\label{lemm:B_action}
  Let $M$ be an object of $\cO_\alg^\infty$ endowed with the
  $\underline B$-module structure defined above. Then 
   \[ b\cdot(X\cdot (b^{-1}\cdot v))=(\Ad(b)X)\cdot v\]
  for any $b\in\underline B(L)$,
  $X\in\frakg$ and $v\in M$.
 
\end{lemma}

\begin{proof}
  It is sufficient to prove the formula for $b\in\underline N(L)$ and
  for $b\in\underline T(L)$. If $b\in\underline N(L)$, then
  $b=\exp(n)$ for some $n\in\mathfrak n$. It follows that $\Ad(b)X$ is equal to
  the finite sum $\sum_{k\geq0}\tfrac{1}{k!}\ad(n)^kX$ and that the action
  of $b$ on $M$ is given by the series
  $\sum_{k\geq0}\tfrac{1}{k!}n^k$ (which is locally finite). Therefore we have,
  \begin{align*}
    b\cdot(X\cdot (b^{-1}\cdot v))&=\sum_{k\geq0,\ell\geq0}(-1)^\ell\frac{1}{k!\ell!}n^kXn^\ell\cdot v \\
                                  &=\sum_{m\geq0}\frac{1}{m!}\sum_{k+\ell=m}(-1)^{k-m}\binom{m}{k}n^kXn^\ell\cdot v \\
                                  &=\sum_{m\geq0}\frac{1}{m!}(\ad(n)^mX)\cdot v =\Ad(b)X\cdot v.
  \end{align*}
  If $b\in\underline T(L)$, then if $\alpha\in\Phi\cup\set{0}$ and
  $X\in\frakg_\alpha$, and if $v\in M^\mu$, we have
  \begin{align*}
    b\cdot(X\cdot(b^{-1}\cdot v))&=b\cdot(X\cdot(\mu(b^{-1})v))=(\mu+\alpha)(b)\mu(b^{-1})X\cdot v \\
                                 &=\alpha(b)X\cdot v=\Ad(b)X\cdot v
  \end{align*}
  as $\Ad(b)X=\alpha(b)X$.
\end{proof}

For later use, we note that we can resolve objects in $\mathcal{O}^{\infty}_\alg$ as follows: 
\begin{lemma}\label{lemm:resolution}
  Let $M$ be an object of $\mathcal{O}_\alg^\infty$. Then there exist
  finite dimensional $U(\frakb)$-modules $V_0$ and $V_1$ and an exact sequence of $U(\frakg)$-modules
  \begin{equation}
    \label{eq:resolution_M}
    U(\frakg)\otimes_{U(\frakb)}V_1\longrightarrow
    U(\frakg)\otimes_{U(\frakb)}V_0\longrightarrow M\longrightarrow0.
  \end{equation}
  Moreover, this exact sequence is $\underline{B}$-equivariant for the $\underline{B}$-actions (on each of the three terms) defined just before Lemma \ref{lemm:B_action}.
\end{lemma}

\begin{proof}
  The existence of a finite dimensional $U(\frakb)$-module $V_0$ and
  a surjective map
  $U(\frakg)\otimes_{U(\frakb)}V_0\twoheadrightarrow M$ is a
  consequence of Proposition \ref{prop:lift}. The existence of $V_1$
  and of the map
  $U(\frakg)\otimes_{U(\frakb)}V_1\rightarrow
  U(\frakg)\otimes_{U(\frakb)}V_0$ follows again from Proposition
  \ref{prop:lift} applied to the kernel of
  $U(\frakg)\otimes_{U(\frakb)}V_0\twoheadrightarrow M$. The
  $\underline{B}$-equivariance is a direct consequence of the definition of the
  algebraic action of $\underline{B}$-action on each
  term of the sequence (\ref{eq:resolution_M}).
\end{proof}

\subsection{Deformations of the category $\cO$}
\label{sec:deform-objects-categ}

Fix $I\subset \Delta$ and let $M$ be some $U(\mathfrak{g})_{A_I}$-module. For
$\mu\in X^*(\underline T)$, we define the $A_I$-submodule
\[ M^\mu\coloneqq\set{v\in M\mid \forall h\in\mathfrak{t},\, h\cdot
    v=(p_I(h)+\mu(h))v}.\]
We note that for $M\in\mathcal{O}^{I,\infty}$ this coincides with the generalized eigenspace for $\mu$ by Remark \ref{rema: second version of Mmu}.
Inspired by the construction of \cite[\S3.1]{SoergelHC}, we define
$\widetilde{\mathcal O}_{\alg}^I$ as the category of
$U(\mathfrak{g})_{A_I}$-modules $M$ such that
\begin{itemize}
\item $M$ is finitely generated over $U(\mathfrak{g})_{A_I}$ ;
\item $M=\bigoplus_{\mu\in X^*(\underline{T})}M^\mu$ and each
  $M^\mu$ is a finite free $A_I$-module ;
\item for any $m\in M$ the $A_I$-submodule $(U(\mathfrak{p}_I)\otimes_L A_I)m$ is
  finitely generated.
\end{itemize}

\begin{lemma}\label{lemm:deformed_to_catO}
  Let $M$ be an object of $\widetilde{\mathcal O}_{\alg}^I$. Then for
  any $n\geq0$, the $U(\frakg)$-module $M/\frakm_I^nM$ is an object of
  $\cO_{\alg}^{I,\infty}$ and $M/\frakm_I M$ is in $\cO_{\alg}^I$.
\end{lemma}

\begin{proof}
  This is a direct consequence of the definitions.
\end{proof}

For $\lambda\in X^*(\underline T)_I^+$ we define
the \emph{deformed generalized Verma module} of weight $\lambda$ as
\[ \widetilde{M}_I(\lambda)\coloneqq
  U(\mathfrak{g})\otimes_{U(\mathfrak{p}_I)}(L_I(\lambda)\otimes_L
  A_I)\] where $U(\mathfrak{p}_I)$ acts on $A_I$ via the composition
$U(\mathfrak{p}_I)\rightarrow U(\mathfrak{l}_I)\xrightarrow{p_I}
U(\mathfrak{z}_I)\rightarrow A_I$. The module $\widetilde{M}_I(V)$ is
an object of $\widetilde{\mathcal O}_{\alg}^I$ and we
have an isomorphism of $U(\frakg)_{A_I}$-modules
\[ \widetilde{M}_I(\lambda)\otimes_{A_I} A_I/\mathfrak m_I =
  M_I(\lambda).\]

\subsubsection{Duality} Recall that there exist an internal duality functor $M\mapsto M^\vee$
on the category $\cO_\alg$ (see \cite[\S3.2]{HumBGG}). We will define
an analogue on $\widetilde{\cO}_\alg$. Let $M$ be an object of the
category $\widetilde{\mathcal O}_{\alg}^I$. We define an action of
$U(\frakg)$ on $M^*\coloneqq\Hom_{A_I}(M,A_I)$ by
$x\cdot f(m)=f(\tau(x)m)$ where $\tau$ is the anti-involution of
$U(\frakg)$ defined in \cite[\S0.5]{HumBGG}. We then define $M^\vee$ to be the
sub-$U(\frakg)$-module of $M^*$ given by
\[ M^\vee\coloneqq\bigoplus_{\mu\in X^*(\underline T)}(M^*)^\mu. \]
\begin{lemma}\label{lemm:dual_deformed}
  If $M$ is an object of the category $\widetilde{\mathcal O}_\alg^I$,
  then so is $M^\vee$. Moreover there is a canonical isomorphism
  $M\xrightarrow{\sim}(M^\vee)^\vee$. Consequently
  $M^\vee/\frakm_IM^\vee\simeq(M/\frakm_IM)^\vee$ is in the category
  $\cO_{\alg}^{I}$.
\end{lemma}

\begin{proof}
  We have a canonical isomorphism of $A_I$-modules
  \[M^*\simeq\prod_{\mu\in X^*(\underline T)}\Hom_{A_I}(M^\mu,A_I)\] and
  we easily check that $(M^*)^\mu=\Hom_{A_I}(M^\mu,A_I)$ for
  $\mu\in X^*(\underline T)$. As any $M^\mu$ is a finite free
  $A_I$-module, so is $(M^*)^\mu=(M^\vee)^\mu$.

  Let
  $\mathfrak
  n_I^-\coloneqq\bigoplus_{\alpha\in-\Phi^+\backslash\Phi_I}\frakg_\alpha$ denote the nilpotent radical of the parabolic Lie subalgebra
  opposite to $\frakp_I$. Note that a $U(\frakg)_{A_I}$-module $M$
  such that $M=\bigoplus_\mu M^\mu$ with $M^\mu$ finite free over
  $A_I$ is in $\widetilde{\cO}_{\alg}^I$ if and only if we can write
  $M=U(\mathfrak n^-_I)\cdot\big(\bigoplus_{\mu\in S}M_\mu\big)$
  for some finite set $S\subset X^*(\underline T)$. By Lemma
  \ref{lemm:deformed_to_catO}, the object $M/\frakm_I M$ lies in the
  category $\cO_\alg^I$ and it follows from \cite[\S9.3]{HumBGG} that
  $(M/\frakm_I M)^\vee$ lies in $\cO_\alg^I$. This implies
  that there exists a finite set $S \subset X^*(\underline T)$ such
  that
 \[(M/\frakm_I M)^\vee=U(\mathfrak n_I^-)\cdot\big(\bigoplus_{\mu\in
      S}(M/\frakm_I M)^{\vee,\mu}\big).\] It follows that for any
  $\mu$ such that $(M/\frakm_I M)^{\vee,\mu}\neq0$, the map
  \[
    \bigoplus_{\substack{\nu\in\sum_{\alpha\in-\Phi\backslash\Phi_I}\NN\alpha\\\mu'\in
        S \\ \mu'+\nu=\mu}}(M/\frakm_I
    M)^{\vee,\mu'}\longrightarrow(M/\frakm_I M)^{\vee,\mu} \]
 given by the action of the corresponding element of $U(\mathfrak{n}_I^-)$ on each summand, is surjective. As $M^\mu$ is a finite free $A_I$-module and $A_I$ is
  a local ring, we deduce from Nakayama's Lemma that the map
    \[
      \bigoplus_{\substack{\nu\in\sum_{\alpha\in-\Phi\backslash\Phi_I}\NN\alpha\\\mu'\in
          S \\ \mu'+\nu=\mu}}M^{\vee,\mu'}\longrightarrow
      M^{\vee,\mu} \] is surjective and thus that
    $M^\vee=U(\mathfrak n_I^-)\cdot\left(\bigoplus_{\mu'\in
        S}M^{\vee,\mu'}\right)$. This implies that $M^\vee$ is a
    finitely generated $U(\frakg)_{A_I}$-module and we also deduce
    from this equality that $M^\vee$ is locally $U(\frakp_I)_{A_I}$-finite.
    
    In order to prove that $M\xrightarrow{\sim}(M^\vee)^\vee$ we note that the natural map $M \fleche (M^\vee)^*$ of $U(\mathfrak g)_{A_I}$-modules  factors through $(M^\vee)^\vee$ 
    and respects the weight decomposition. Moreover as $M^\mu$ is free over $A_I$ for all $\mu$, 
    the induced bi-duality $M^\mu \overset{\sim}{\fleche} (M^{\mu,*})^*$ morphism is an isomorphism.
\end{proof}

\subsubsection{Blocks} Let $Z(\mathfrak g)$ denote the center of $U(\frakg)$ and let
$\chi : Z(\mathfrak g)\rightarrow L$ be a character of $Z(\mathfrak g)$. Let
$\cO_\chi$ be the subcategory of objects $M$ of $\cO_\alg$ such that
$z-\chi(z)$ acts nilpotently on $M$ for any $z\in Z(\mathfrak g)$. For
$I\subset\Delta$, we denote by $\cO_\chi^I$ the full subcategory of objects of
$\cO_\alg^I$ which are also in $\cO_\chi$. We deduce from
\cite[Prop.~1.12]{HumBGG} that there is a decomposition into blocks 
\[\cO_\alg^I=\bigoplus_\chi\cO_\chi^I.\] 
We write $\widetilde{\cO}_\chi^I$ for
the subcategory of objects $M$ of $\widetilde{\cO}^I_{\alg}$ such that
$M/\frakm_I M$ lies in $\cO_\chi^I$, and similarly $\mathcal
O^{I,\infty}_\chi$.

\begin{rema}\label{rema:central_char}
  For $\lambda\in X^*(\underline T)$,
    let $\chi$ be the character $\chi_\lambda$ defined in
    \cite[\S1.7]{HumBGG}. Then \emph{loc.~cit.} implies that
    $\widetilde{M}_I(\lambda)$ is in $\widetilde{\cO}^I_{\chi_\lambda}$.
\end{rema}

\begin{lemma}\label{lemm:block_dec}
  We have decompositions
  $\widetilde{\cO}^I_\alg=\bigoplus_\chi\widetilde{\cO}_\chi^I$ and
  $\mathcal O^{I,\infty}_\alg = \bigoplus_\chi \mathcal
  O_\chi^{I,\infty}$.
\end{lemma}

\begin{proof}
  Let $M$ be an object of $\widetilde{\cO}^I_\alg$. For a character
  $\chi : Z(\mathfrak g)\rightarrow L$ and
  $\mu\in X^*(\underline T)$, let $M^{\mu,\chi}$ denote the subset of
  elements $x\in M^\mu$ such that $(z-\chi(z))^nx\rightarrow0$ for the
  $\frakm_I$-adic topology on the finite free $A_I$-module $M^\mu$. We
  easily check that
  $M^\chi\coloneqq\bigoplus_{\mu\in X^*(\underline T)}M^{\mu,\chi}$ is
  an $U(\frakg)_{A_I}$-submodule of $M$ which lies in
  $\widetilde{\mathcal O}^I_\chi$ and that
  $M=\bigoplus_{\chi}M^\chi$. The case of $\cO^{I,\infty}_\alg$ is
  similar.
\end{proof}

\begin{lemma}\label{lemm:Vermas_in_blocks}
  Let $\lambda_1,\lambda_2\in X^*(\underline T)$. Assume that $\widetilde M_I(\lambda_1)$ and
  $\widetilde M_I(\lambda_2)$ are in the same block $\widetilde{\cO}_\chi^I$ for a character $\chi:Z(\mathfrak{g})\rightarrow L$. Then there exists
  $w\in W$ such that $w\cdot\lambda_1=\lambda_2$.
\end{lemma}

\begin{proof}
  By Remark \ref{rema:central_char}, the
  claim follows from the same claim in the category $\cO_\chi^I$. As
  $M_I(\lambda_1)$ and $M_I(\lambda_2)$ are quotients of
  $M(\lambda_1)$ and $M(\lambda_2)$, this is a consequence of
  \cite[Thm.~1.10]{HumBGG}.
\end{proof}

When $\lambda$ is a character of $\mathfrak t$, we often write by abuse of notation
$\mathcal O_\lambda$ (resp. $\mathcal O^{I,\infty}_{\lambda},\widetilde{\mathcal O}^I_{\lambda}$) 
for the block $\mathcal O_{\chi_\lambda}$ (resp. $\mathcal O^{I,\infty}_{\chi_\lambda},
\widetilde{\mathcal O}^I_{\chi_\lambda}$) where $\chi_\lambda$ is the character of $Z(\mathfrak g)$ 
giving the action of the center on $M(\lambda)$  (see
  \cite[\S1.7]{HumBGG}). In particular, $\chi_\lambda = \chi_\mu$ if, and only if, there is $w \in W$ such that
$w\cdot \lambda = \mu$.

\begin{cor}\label{coro:hightes_weight_vector}
  Let $\lambda\in X^*(\underline T)$ be a dominant weight and let
  $\chi_\lambda$ be the associated character of $Z(\mathfrak g)$. If $M$ is an object of
  $\widetilde{\cO}_{\chi_\lambda}^I$ (resp. $\cO_{\chi_\lambda}^{I,\infty}$), then $M^\lambda=(M^\lambda)^{\mathfrak n}$.
\end{cor}

\begin{proof}
  Assume that this is false. Then there exists $\alpha\in\Phi^+$ and
  $x\in\frakg_\alpha$ such that $xM^{\lambda}\neq0$. Thus there exists
  $\mu>\lambda$ such that $M^{\mu}\neq0$. As $M$ lies in the category
  $\widetilde{\cO}_\alg^I$ (resp. $\cO_\chi^{I,\infty}$), we can choose $\mu$ to be
  maximal which then implies $\mathfrak n M^\mu=0$. 
  As $M^\mu \neq 0$ Nakayama's lemma implies that there exists $v \in M^\mu$ which is non zero in $M^\mu/\mathfrak mM^\mu$.
   Then $v$ defines a 
  map $\widetilde M_I(\mu)\rightarrow M$ with $\mu >\lambda$, which is non-zero after reduction by $\mathfrak m$. Thus it induces a non-zero map $M_I(\mu) \fleche M/\mathfrak mM \in \mathcal O_{\chi_\lambda}$.  It follows that
  $\mu=w\cdot\lambda$ which is a
  contradiction.
\end{proof}

\subsubsection{Deformed Verma modules}
Let $\lambda\in X^*(\underline T)$ and let $V$ be a finite dimensional
$U(\frakg)$-module. Then we have an isomorphism of
$U(\frakg)_A$-modules
\[ \widetilde M(\lambda)\otimes_LV\simeq
  U(\frakg)_A\otimes_{U(\frakb)_A}(V_{|\frakb}\otimes_LA(\lambda)). \]
  Indeed there is a canonical map from the left to the right, which then is easily checked to be an isomorphism.
As $V_{|\frakb}$ is a successive extension of one dimensional
$U(\frakb)$-modules, and as $U(\frakg)_A\otimes_{U(\frakb)_A}(-)$ is an
exact functor (as follows from the PBW Theorem), we have a filtration
$(\Fil_i)$ of $\widetilde M(\lambda)\otimes_LV$ such that each
subquotient $\Fil_i/\Fil_{i-1}$ is isomorphic to
$\widetilde M(\lambda+\nu_i)$ for $\nu_i$ a weight of $V$. Moreover
the family $(\nu_i)$ is the family of weights of $V$ (counted with
multiplicity).

\begin{prop}\label{prop:filt_split}
 Let $K$ denote the fraction field of $A$.
 Then the filtration $(\Fil_i\otimes_AK)$ of
  $(\widetilde M(\lambda)\otimes_LV)\otimes_AK$ splits in the category
  of $U(\frakg)_K$-modules, i.e.~there exists an isomorphism of
  $U(\frakg)_K$-modules
  \[ (\widetilde
    M(\lambda)\otimes_LV)\otimes_AK\simeq\bigoplus_i(\widetilde
    M(\lambda+\nu_i)\otimes_AK) \] compatible with the filtration
  $(\Fil_i\otimes_AK)$.
\end{prop}

\begin{proof}
  This is a consequence of the paragraph preceding
  \cite[Thm.~8]{SoergelHC}.
\end{proof}

\begin{lemma}\label{lemm:dom_Verma_is_projective}
  Let $\lambda\in X^*(\underline T)^+_I$ be a dominant weight (with respect to $\underline{P}_I$) and let $V$
  be a finite dimensional $U(\frakg)$-module. Let $M$ be an object of
  $\cO_\alg^{I,\infty}$. Then the map
  \[\Hom_{U(\frakg)_{A_I}}(\widetilde
  M_I(\lambda)\otimes_LV,M)\rightarrow\Hom_{U(\frakg)}(M_I(\lambda)\otimes_LV,M/\frakm_I
  M)\] given by reduction modulo $\frakm_I$ is surjective.
\end{lemma}

\begin{proof} The $L$-vector space $\Hom_L(V,L)$ has the structure of an
  $U(\frakg)$-module induced by $\frakg$-action defined by $x\cdot\phi=-\phi(x\cdot)$ for
  $x\in\frakg$ and $\phi\in\Hom_L(V,L)$. For any
  $U(\mathfrak g)$-modules $M_1$ and $M_2$, the adjunction isomorphism
  $ \Hom_L(M_1\otimes_L V,M_2) \simeq \Hom_L(M_1,M_2\otimes_L
  \Hom_L(V,L))$ is $\frakg$-equivariant and hence induces an isomorphism,
  \[ \Hom_{U(\mathfrak g)}(M_1\otimes_L V,M_2) \simeq
    \Hom_{U(\mathfrak g)}(M_1,M_2\otimes_L \Hom_L(V,L)).\] Thus, as
  $M\otimes_L \Hom_L(V,L)$ lies in $\mathcal O^{I,\infty}_\alg$ we can
  assume that $V = L$. Using Lemma \ref{lemm:block_dec}, we can assume
  that $M$ is in $\cO_\chi^{I,\infty}$ for some character $\chi$ and by Remark
  \ref{rema:central_char}, it is sufficient to consider the case where
  $\chi=\chi_\lambda$. By construction of the deformed generalized Verma
  modules we have
  $\Hom_{U(\frakg)_{A_I}}(\widetilde
  M_I(\lambda),M)=(M^\lambda)^{\mathfrak n_I}$ and
  $\Hom_{U(\frakg)}(M_I(\lambda),M/\frakm_I M)=((M/\frakm_I
  M)^\lambda)^{\mathfrak n_I}$. However it follows from Corollary
  \ref{coro:hightes_weight_vector} that
  $(M^\lambda)^{\mathfrak n_I}=M^\lambda$ and
  $((M/\frakm_I M)^\lambda)^{\mathfrak n_I}=(M/\frak m_I)^\lambda$. It
  is thus sufficient to prove that the map
  $M^\lambda\rightarrow (M/\frakm_I M)^\lambda$ is surjective, which
  is obvious.
\end{proof}

\begin{prop}\label{prop:lift}
  Let $M$ be an object of the category $\cO_{\alg}^{I,\infty}$. Then there
  exist weights $\lambda_1,\dots,\lambda_r \in X^*(\underline T)^+_I$ and finite
  dimensional $U(\frakg)$-modules $W_1,\dots, W_r$ and a surjective
  map of $U(\frakg)_{A_I}$-modules
  \begin{equation}
    \label{eq:lift}
    (\widetilde
    M_I(\lambda_1)\otimes_LW_1)\oplus\cdots\oplus(\widetilde
    M_I(\lambda_r)\otimes_LW_r)\twoheadrightarrow M.
  \end{equation}
  In particular
  $M$ is a quotient of an object of the category
  $\widetilde{\cO}_{\alg}^I$. Moreover there exists an integer
  $N \geq 0$ such that the map~(\ref{eq:lift}) factors through 
  \[ \left((\widetilde
      M_I(\lambda_1)\otimes_LW_1)\oplus\cdots\oplus(\widetilde
      M_I(\lambda_r)\otimes_LW_r)\right)\otimes_{A_I}A_I/\frakm_I^N. \]
\end{prop}

\begin{proof}
  By \cite[Thm.~9.8]{HumBGG} (and its proof), there exist dominant
  weights $\lambda_1,\dots,\lambda_r$, finite dimensional
  $U(\frakg)$-modules $W_1,\dots,W_r$ and a surjective map
  \[ (M_I(\lambda_1)\otimes_LW_1)\oplus\cdots
    (M_I(\lambda_r)\otimes_LW_r)\twoheadrightarrow M/\frakm_I M. \] By
  Lemma \ref{lemm:dom_Verma_is_projective}, this map can be lifted
  into a $U(\frakg)_{A_I}$-equivariant map
  \[ \widetilde M_I(\lambda_1)\otimes_LW_1\oplus\cdots \widetilde
    M_I(\lambda_r)\otimes_LW_r\twoheadrightarrow M \] which is
  surjective by Nakayama's Lemma. The last assertion is a consequence of
  the fact that $M$ is finitely generated as a $U(\frakg)$-module and
  all its elements are killed by some power of $\frakm_I$ so that $M$
  is killed by $\frakm_I^N$ for some $N\geq0$.
\end{proof}

\subsection{Bimodule structure}
\label{sec:bimodule}

Let $\xi : Z(\mathfrak g) \rightarrow U(\mathfrak t)$ be the Harish-Chandra
map. Recall that it is defined as follows: for $x\in Z(\mathfrak g)$ there exists a
unique element $\xi(x) \in U(\mathfrak t)$ such that $x\in\xi(x)+U(\frakg)\mathfrak n$
(see \cite[Lem.~8.17]{Knapp}). For any $\nu\in X^*(\underline T)$ we denote by 
$t_{\nu}$ the unique endomorphism of $U(\mathfrak t)$ such that $t_{\nu}(x)=x+\nu(x)$ for $x\in\mathfrak t$. 
Note that 
$t_{-\delta_G}\circ\xi$ induces 
an isomorphism from $Z(\mathfrak g)$ on to $U(\mathfrak t)^W$ (see
\cite[Thm.~6.18]{Knapp}). 
For a dominant weight $\lambda\in X^\ast(\underline{T})$ we define
a map
\[ h_\lambda : A\otimes_L Z(\mathfrak g)\xrightarrow{\Id\otimes\xi}A\otimes_L\otimes U(\mathfrak
  t)\xrightarrow{\Id\otimes t_{\lambda}} A\otimes_{A^W}A \] 
 following \cite[\S3.2]{SoergelHC},
It follows from \cite[Thm.~9]{SoergelHC} that $h_\lambda$ is surjective
(note that $\mathcal W_\lambda$ in \emph{loc.~cit.} is trivial in our situation). If
$I\subset\Delta$ is a finite subset, tensorization on the left with $p_I : A\twoheadrightarrow A_I$ yields a map
$h_\lambda : A_I\otimes_LZ(\mathfrak g)\rightarrow A_I\otimes_{A^W}A$.

For $w\in W$, let $I_w\subset A_I\otimes_LZ(\mathfrak g)$
denote the kernel of the map
\[ h_{\lambda,w} : A_I\otimes_LZ(\mathfrak g)\xrightarrow{\Id\otimes
    h_\lambda}A_I\otimes_{A^W}A\xrightarrow{x\otimes y\mapsto
    (xp_I(\Ad(w)y))}A_I. \] It is not hard to see that this kernel only depends on the choice of
$\overline{w}\in W_I\backslash W$.

\begin{prop}\label{prop:deformed_Verma_annihilator}
  For $w\in {}^IW$, the $A_I\otimes_LZ(\mathfrak g)$-modules
  $\widetilde M_I(w\cdot\lambda)$ and
  $\widetilde M_I(w\cdot\lambda)^\vee$ are annihilated by $I_w$.
\end{prop}

\begin{proof}
The result for $\widetilde M_I(w\cdot\lambda)^\vee$ follows from
  the result for $\widetilde M_I(w\cdot\lambda)^\vee$ and the
  inclusion
  \[ \widetilde M_I(w\cdot\lambda)^\vee\subset\Hom_A(\widetilde
    M_I(w\cdot\lambda),A). \]
     Hence it is enough to check that
  the action of $A_I\otimes_LZ(\mathfrak g)$ on
  $\widetilde M_I(w\cdot\lambda)$ factors through $h_{\lambda,w}$. As
  this action is central and $\widetilde M_I(w\cdot\lambda)$ is
  generated by $\widetilde M_I(w\cdot\lambda)^{w\cdot\lambda}$ as an
  $U(\frakg)_{A_I}$-module, it is sufficient to check that the action
  $A_I\otimes_LZ(\mathfrak g)$ on $\widetilde M_I(w\cdot\lambda)^{w\cdot\lambda}$
  factors through $h_{\lambda,w}$. Using the fact that $\mathfrak n$
  acts trivially on $\widetilde M_I(w\cdot\lambda)^{w\cdot\lambda}$, an
  element $x\in Z(\mathfrak g)$ acts on this space via $\xi$. For the
  clarity of the computation let us write
  $\eps_\nu : U(\mathfrak t)\rightarrow A_I$ for the $L$-algebra
  homomorphism associated to an $L$-linear map
  $\nu : \mathfrak t \rightarrow A_I$ and let
  $\iota : \mathfrak t\hookrightarrow A\twoheadrightarrow A_I$. Then
  for $x\in Z(\mathfrak g)$ and $v\in \widetilde M_I(w\cdot\lambda)$, we
  have
  \begin{align*}
    \eps_{w\cdot\lambda+\iota}(\xi(x))=\eps_{w(\lambda+\delta_G+w^{-1}(\iota))}(t_{-\delta_G}(\xi(x)))&=\eps_{\lambda+\delta_G+w^{-1}(\iota)}(t_{-\delta_G}(\xi(x)))\\ &=\eps_{w^{-1}(\iota)}(h_\lambda(x))=p_I(\Ad(w)(h_\lambda(x)))
  \end{align*}
  (where we use that the image of $t_{-\delta_G}\circ\xi$ lies in
  $U(\mathfrak t)^W$). As an element $y\in U(\mathfrak t)$
  acts by multiplication by $\eps_{w\cdot\lambda+\iota}(y)$ on
  $\widetilde M_I(w \cdot \lambda)^{w\cdot\lambda}$, we conclude that an element
  $x\otimes z\in A_I\otimes_LZ(\mathfrak g)$ acts by multiplication by
  $xp_I(\Ad(w)(h_\lambda(x)))$ on $\widetilde M_I(w \cdot \lambda)^{w\cdot\lambda}$,
  which is the desired formula.
\end{proof}

\begin{rema}\label{rema:t_and_tdual}   
  The ring $U(\mathfrak t)$ (resp. $U(\mathfrak z_I)$) is the affine coordinate ring of the (affine) $L$-scheme associated to the dual
  $\mathfrak t^*$ of $\mathfrak t$ (resp. ~to the dual $\mathfrak z_I^*$ of $\mathfrak z_I$) so that $A$ (resp. $A_I$) is the stalk of the structure sheaf of  $\mathfrak t^*$ (resp.~of $\mathfrak z_{I}^*$) at the origin. The ideal $I_w$
  is the ideal defining the irreducible component $T_{I,w}$ of
  $\mathfrak (\mathfrak z_I^*\times_{\mathfrak t^*/W}\mathfrak t^*)_{(0,0)}$
  consisting of pairs $(\lambda,\mu)\in\mathfrak z^*_I \times \mathfrak t^*$ of
  characters such that $\mu=w(\lambda)$. 
  
Later in the paper we will view the $L$-scheme $\mathfrak{t}^*$ as the Lie algebra $\mathfrak{t}^{\vee}$ of the dual torus $\underline{T}_L^\vee$ of the Langlands dual group $\underline{G}_L^\vee$, that we consider as an algebraic group over $L$. 
As we will later specialize to the case where $\underline G$ is isomorphic
  to a product of $r$ copies of $\GL_n$ the reductive group $\underline G$ is self dual and we will identify $\mathfrak{t}^*=\mathfrak{t}^\vee$ with $\mathfrak{t}$ in order to avoid the additional $(-)^\vee$ in the notation.
 In particular we will consider $U(\mathfrak t)$ as the affine coordinate ring
  of $\mathfrak t$. 
  The inclusion $\mathfrak{z}_I^*\hookrightarrow \mathfrak{l}_I^*$ induced by the projection $p_I:\mathfrak{l}_I\rightarrow \mathfrak{z}_I$ is then identified with the inclusion $\mathfrak{z}_I^\vee\hookrightarrow \mathfrak{l}_I^\vee$ of the center of the Lie algebra of the Langlands dual group of $\underline{L}$ and again we use self duality (in the case of products of copies of ${\rm GL}_n$) to identify this map with $\mathfrak{z}_I\hookrightarrow \mathfrak{l}_I$. Hence we obtain a canonical map $\mathfrak{z}_I\hookrightarrow \mathfrak{t}$ of $L$-schemes corresponding to the morphism $U(\mathfrak{t})\rightarrow U(\mathfrak{z}_I)$.
  With this identification the ideal $I_w$ defines
  the irreducible component $T_{I,w}$ of
  $(\mathfrak z_I \times_{\mathfrak t/W}\mathfrak t)_{(0,0)}$ whose
  points are the pairs $(x,y)\in\mathfrak t^2$ such that
  $y=w^{-1}(x)$.
  \end{rema}

We finally recall the following result of Soergel (Endomorphismensatz 7
\cite{Soerg_KatO}).

\begin{prop}\label{prop:endomorphismensatz}
  The action of $Z(\mathfrak g)$ on $P(w_0\cdot\lambda)$ factors through the map
  $t_\lambda\circ\xi : Z(\mathfrak g) \twoheadrightarrow L\otimes_{A^W}A$
  and induces an isomorphism
  $L\otimes_{A^W}A\simeq\End_{\cO}(P(w_0\cdot\lambda))$.
\end{prop}

\section{The Emerton--Jacquet functor}
\label{sec:emert-jacq-funct}

Let $\underline{G}$ be a quasi-split reductive group defined over
$\QQ_p$. Let $\underline{B}$ be a Borel subgroup and $\underline{T}$
be a maximal torus of $\underline{G}$ contained in $\underline{B}$. We
set $G\coloneqq\underline{G}(\QQ_p)$,
$B\coloneqq\underline{B}(\QQ_p)$, $T\coloneqq\underline{T}(\QQ_p)$. We
also fix $L$ a finite extension of $\QQ_p$ which will be the
coefficient field of our representations. We assume that $L$ is big
enough so that the torus $\underline T\times_{\QQ_p} L$ is split (and
then $\underline G\times_{\QQ_p}L$ is split). We denote $\frakg$,
$\frakb$ etc. the Lie algebras of $\underline G\times_{\QQ_p}L$,
$\underline B\times_{\QQ_p}L$ etc.
In the following we will consider the category ${\rm Rep}_L^{\rm la}G$ of locally analytic $G$-representations on locally convex $L$-vector spaces, as well as the corresponding variants for the ($\mathbb{Q}_p$-analytic) groups $B,T$, etc.
In \cite[Def.~3.4.5]{EmertonJacquetI} Emerton constructs a functor
\[J_B:{\rm Rep}_L^{\rm la}G\rightarrow {\rm Rep}_L^{\rm la}T\]
that we refer to as the Emerton--Jacquet functor. 
It is defined as follows: Let $N_0$ be a compact open subgroup
of $N$ and let $T^+\coloneqq\set{t\in T \mid tN_0t^{-1} \subset N_0}$. If $V$ is a
$L$-linear representation of $B$, we endow the $L$-vector space
$V^{N_0}$ with the action of the monoid $T^+$ defined by
\[ [t] v\coloneqq[N_0:tN_0t^{-1}]^{-1}\sum_{u\in
    N_0/tN_0{t^{-1}}}ut(v). \]
 Then $J_B(V)$ is the finite slope space $(V^{N_0})_{\rm fs}$ of $V^{N_0}$ with respect to the action of $T^+$ on which the $T^+$-action extends to a locally analytic representation of $T$. 
 
\subsection{Families of locally analytic representations of the Borel
  subgroup}
\label{sec:Borel}

Let $s\in\ZZ_{\geq0}$ be an integer and let $\Pi$ be a locally
analytic $L$-representation of $\ZZ_p^s\times B$. We consider the
following hypothesis on $\Pi$:
\begin{hypothese}\label{hyp:S_adm_B}
  There exists a locally analytic representation of $N_0$ on a locally
  convex $L$-vector space of compact type $V$ such that
  \[ \Pi_{|\ZZ_p^s\times N_0}\simeq\mathcal{C}^{\la}(\ZZ_p^s,L)\hat{\otimes}_LV. \]
\end{hypothese}

Given $s$, we set $S\coloneqq\cO_L[[\ZZ_p^s]]$ and write $\Spf(S)^{\rig}$ for the rigid analytic generic fiber of $\Spf(S)$. This space is a rigid analytic open polydisc and we write 
\[S^{\rig}=\Gamma(\Spf(S)^{\rig},\mathcal{O}_{\Spf{S}^{\rig}})\]
for its ring of rigid analytic functions, which is a Fr\'echet $L$-algebra (when endowed with its natural topology).
We note that a finitely generated, projective $S^{\rig}$-module $C$ defines a vector bundle on $\Spf(S)^{\rig}$. As every vector bundle on a rigid analytic polydisc is free, it follows that $C$ is free as well, i.e.~every finitely generated projective $S^{\rig}$-module is finite free. 
Moreover, finite dimensional quotients of $S^{\rig}$ admit resolutions by a perfect complexes:
\begin{lemma}\label{lemm:resolution_S}
  Let $\fraka\subset S^{\rig}$ be a closed strict ideal such that
  $\dim_LS^{\rig}/\fraka<\infty$. Then there exists perfect complex
  $C_\bullet$ of $S^{\rig}$-modules which is a resolution
  of $S^{\rig}/\fraka$ and such that $C_0=S^{\rig}$.
\end{lemma}

\begin{proof}
  As $S[1/p]$ is dense in $S^{\rig}$, its image in $S^{\rig}/\fraka$
  is dense $L$-vector space and, as $S^{\rig}/\fraka$ is finite
  dimensional, is in fact equal to $S^{\rig}/\fraka$. Setting
  $\fraka_0\coloneqq\fraka\cap S[1/p]$, we have
  $S[1/p]/\fraka_0\simeq S^{\rig}/\fraka$. As $S^{\rig}$ is a flat
  $S[1/p]$-module, it is sufficient to prove that $S[1/p]/\fraka_0$
  has a finite resolution by finite projective $S[1/p]$-modules, which is a
  consequence of the fact that $S[1/p]$ is a regular noetherian ring.
\end{proof}

Let $C_\bullet$ be a complex of finite free
$S^{\rig}$-modules. For each $n\geq0$, $C_n$ is endowed with its
canonical topology induced by the topology of $S^{\rig}$, then the
differentials in the complex $C_\bullet$ are continuous. The complex
$\Pi^{\bullet}\coloneqq\Hom_{S^{\rig}}(C_\bullet,\Pi)$ is then a
complex of locally analytic $L$-representations of
$\ZZ_p^s\times B$. We also set
$\Pi^{N_0,\bullet}\coloneqq\Hom_{S^{\rig}}(C_\bullet,\Pi^{N_0})$ and
$J_B(\Pi)^{\bullet}\coloneqq\Hom_{S^{\rig}}(C_\bullet,J_B(\Pi))$.

\begin{lemma}\label{lemm:produit_tens_exact}
  Let $0\rightarrow U\rightarrow V\rightarrow W\rightarrow 0$ be a
  short exact sequence of topological $L$-vector spaces of compact
  type (resp.~nuclear Fr\'echet spaces) and let $X$ be a topological
  $L$-vector space of compact type (resp.~nuclear Fr\'echet
  space). Then the following sequence is exact
  \[ 0\rightarrow U\hat{\otimes}_LX\rightarrow
    V\hat{\otimes}_LX\rightarrow W\hat{\otimes}_LX\rightarrow0.\]
\end{lemma}

\begin{proof}
  The claim follows from \cite[Lemm.~4.13]{SchraenGL3},
  \cite[Cor.~1.4]{STlocan} and from \cite[Prop.~1.1.32]{Emertonlocan}.
\end{proof}

\begin{lemma}\label{lemm:N0inv}
  Let $\Pi$ be a locally analytic representation of $\ZZ_p^s\times B$
  satisfying Hypothesis \ref{hyp:S_adm_B}. Then the two complexes
  $\Pi^\bullet$ and $\Pi^{N_0,\bullet}$ are complexes of $L$-vector
  spaces of compact type with strict continuous transition
  maps. Moreover for any integer $n\geq0$, we have an isomorphism of
  topological $T^+$- modules
  \[ H^n(\Pi^{N_0,\bullet})\simeq H^n(\Pi^{\bullet})^{N_0}.\]
\end{lemma}

\begin{proof}
  Fix an isomorphism
  $\Pi_{|\ZZ_p^s\times
    N_0}\simeq\mathcal{C}^{\la}(\ZZ_p^s,L)\hat{\otimes}_LV$ whose
  existence comes from hypothesis \ref{hyp:S_adm_B}. As any $C_m$ is a
  finite free $S^{\rig}$-module and as the completed tensor product
  $-\hat{\otimes}_L-$ commutes with finite direct sums
  (\cite[Lem.~1.2.13]{kohlhaase}), we have an isomorphism of complexes
  of topological representations of $\ZZ_p^s\times N_0$:
  \[ \Pi^{\bullet} \simeq
    \Hom_{S^{\rig}}(C_\bullet,\mathcal{C}^{\la}(\ZZ_p^s,L))\hat{\otimes}_LV. \]
  As $\mathcal{C}^{\la}(\ZZ_p^s,L)$ is an admissible locally analytic
  representation of $\ZZ_p^s$, the complex
  $\Hom_{S^{\rig}}(C_\bullet,\mathcal{C}^{\la}(\ZZ_p^s,L))$ has strict
  transition maps with closed images (\cite[Prop.~6.4]{STdist}). We
  deduce from this fact and from Lemma \ref{lemm:produit_tens_exact}
  that the complex $\Pi^\bullet$ has strict transition maps and that
  we have topological isomorphisms
  $H^n(\Pi^\bullet)\simeq
  H^n(\Hom_{S^{\rig}}(C_\bullet,\mathcal{C}^{\la}(\ZZ_p^s,L)))\hat{\otimes}_LV$
  for any $n\geq0$. The commutation of $\hat{\otimes}_L$ with finite
  direct sum implies that we have a topological isomorphism of
  $L$-vector spaces for any $m\geq0$:
  \[ (\Hom_{S^{\rig}}(C_m,\Pi^{N_0})\simeq
    \Hom_{S^{\rig}}(C_m,\mathcal{C}^{\la}(\ZZ_p^s,L))\hat{\otimes}_LV^{N_0}. \]
  We deduce as before that the complex $\Pi^{\bullet,N_0}$ has strict
  transition maps and that we have isomorphisms
  \[ H^n(\Pi^{N_0,\bullet})\simeq H^n(\Pi^\bullet)^{N_0} \]
  for any $n\geq0$.
\end{proof}

\begin{prop}\label{prop:Jacquet_complex}
  For any integer $n\geq0$, there is an isomorphism 
  \[ H^n(J_B(\Pi)^{\bullet})\simeq
    J_B(H^n(\Pi^{\bullet}))\]
    of locally analytic
  $L$-representations of $\ZZ_p^s\times T$.
\end{prop}

\begin{proof}
 It follows from \cite[Prop.~3.2.4.(ii)]{EmertonJacquetI} that there is a natural
  continuous $T^+$-equivariant map of complexes
  $(\Pi^{N_0,\bullet})_{\fs}\rightarrow\Pi^{N_0,\bullet}$ inducing a
  continuous $T^+$-equivariant morphism
  $H^n(\Pi^{N_0,\bullet}_{\fs})\rightarrow H^n(\Pi^{N_0,\bullet})$. By
  \emph{loc.~cit.}, the universal property of the functor $(-)_{\fs}$
  provides a $T$-equivariant map
  $H^n(\Pi^{N_0,\bullet}_{\fs})\rightarrow
  H^n(\Pi^{N_0,\bullet})_{\fs}$. It follows from Lemma
  \ref{lemm:N0inv} that it is sufficient to prove that this map is a
  topological isomorphism.

  We now deduce from \cite[Prop.~3.2.27]{EmertonJacquetI} and
  \cite[Thm.~4.5]{Fu_derived_Jacquet} that given an exact sequence
  $0\rightarrow U\rightarrow V\rightarrow W\rightarrow0$ of spaces of compact type with continuous action of $T^+$,
  then
  $0\rightarrow U_{\fs}\rightarrow V_{\fs}\rightarrow
  W_{\fs}\rightarrow0$ is exact, the image of $U_{\fs}$ is closed in
  $V_{\fs}$ and the map $V_{\fs}\rightarrow W_{\fs}$ is strict. The open mapping theorem then implies that the sequence is strict
  exact. As the complex $\Pi^{N_0,\bullet}$ has strict transition maps
  by Lemma \ref{lemm:N0inv}, we conclude that the map
  $H^n(\Pi^{N_0,\bullet}_{\fs})\rightarrow
  H^n(\Pi^{N_0,\bullet})_{\fs}$ is a topological isomorphism.
\end{proof}

\begin{prop}\label{prop:almost_flatness}
  Let $\Pi$ be a locally analytic $L$-representation of
  $\ZZ_p^s\times B$ satisfying the hypothesis \ref{hyp:S_adm_B}. Let
  $\fraka$ be a closed strict ideal of $S^{\rig}$ such that
  $\dim_LS^{\rig}/\fraka<+\infty$. Then the map
  \[\fraka\otimes_{S^{\rig}}J_B(\Pi)'\longrightarrow
    J_B(\Pi)'\] is injective.
\end{prop}

\begin{proof}
  By Lemma \ref{lemm:resolution_S}, there exists a perfect complex $C_\bullet$
  of $S^{\rig}$-modules such that, $C_0=S^{\rig}$,
  $H_0(C_\bullet)\simeq S^{\rig}/\fraka$ and $H_i(C_\bullet)=0$ for
  $i>0$. By Hypothesis \ref{hyp:S_adm_B}, we have
  $\Pi|_{\ZZ_p^s\times N_0}\simeq\mathcal{C}^{\la}(\ZZ_p^s,L)\hat{\otimes}_LV$ for
  some topological $L$-vector space of compact type $V$. As
  $C_\bullet$ has strict transition maps, it follows from Lemma
  \ref{lemm:produit_tens_exact} that the complex
  $C_\bullet\otimes_{S^{\rig}}\Pi'\simeq C_\bullet\hat{\otimes}_LV'$
  is a resolution of $(S^{\rig}/\fraka)\hat{\otimes}_LV'$. We then deduce from
  $\Hom_{S^{\rig}}(C_i,\Pi)'\simeq C_i\otimes_{S^{\rig}}\Pi'$ for any
  $i\geq0$, that $H^i(\Hom_{S^{\rig}}(C_\bullet,\Pi))=0$ for
  $i>0$. Therefore Proposition \ref{prop:Jacquet_complex} implies that
  $H^i(\Hom_{S^{\rig}}(C_\bullet,J_B(\Pi)))=0$ for $i>0$. 
  We denote by $(-)'$ the duality between spaces of compact type and Fr\'echet spaces. This duality implies that $H_i(C_\bullet\otimes_{S^{\rig}}J_B(\Pi)')=0$ for $i>0$. As
  $\fraka=\coker(C_2\rightarrow C_1)$, we deduce that
  \begin{align*}
    \fraka\otimes_{S^{\rig}}J_B(\Pi)'&=\coker(C_2\otimes_{S^{\rig}}J_B(\Pi)'\rightarrow
    C_1\otimes_{S^{\rig}}J_B(\Pi)')\\ &\subset
    C_0\otimes_{S^{\rig}}J_B(\Pi)'=J_B(\Pi)'. \qedhere \end{align*}
\end{proof}

\subsection{Families of locally analytic representations of $G$}
\label{sec:families_G}

Let $\Pi$ be an admissible locally
analytic $L$-representation of $\ZZ_p^s\times G$. 
The aim of this section is to use $\Pi$ in order to construct a functor \[M\mapsto  \Hom_{U(\mathfrak{g})}(M,\Pi)\] from the category $\mathcal{O}_{\rm alg}^\infty$ to the category of locally analytic $\mathbb{Z}_p^s\times B$-representations, and then, by composing with $J_B$, to locally analytic $\mathbb{Z}_p^s\times T$-representations.
We will usually assume that we are in the following situation:
\begin{hypothese}\label{hyp:S_adm_G}
  There exists a uniform open pro-$p$-subgroup $H$ of $G$, an integer
  $m\geq0$ and a topological $\ZZ_p^s\times H$-equivariant isomorphism
  \[ \Pi_{|\ZZ_p^s\times H}\simeq\mathcal{C}^{\la}(\ZZ_p^s\times H,L)^m.
  \]
\end{hypothese}

Recall from section \ref{sec:nilp-acti-umathfr} that if $M$ is an
object of $\mathcal{O}_{\alg}^\infty$, there is a unique algebraic
action of $\underline{B}(L)$ on $M$ which lifts the structure of
$U(\frakb)$-module on $M^{\mathrm{ss}}$. We endow $M$ with the action
of $B=\underline B(\QQ_p)$ obtained by restriction to $B$.

Let $M$ be an object of $\mathcal{O}_\alg^\infty$ with its semi-simplified $B$-action. We define an
action of $B$ on $\Hom_L(M,\Pi)$ by
\[ b\cdot f=b f(b^{-1}-) \] for $f\in\Hom_L(M,\Pi)$ and $b\in B$. It
follows from Lemma \ref{lemm:B_action} that this action preserves the
subspace $\Hom_{U(\frakg)}(M,\Pi)$. We moreover endow $\Hom_{U(\frakg)}(M,\Pi)$ with the left $\ZZ_p^s$-action inherited from the one on $\Pi$. 
While the definition of the $B$-action using the semi-simplified action on $M$ might not seem very natural at a first glance, the following lemma says that this definition applied to deformed Verma modules allows us to compute generalized eigenspaces. Given an $U(\mathfrak t)$-module $X$ we write
\[ X[(\mathfrak t-\lambda)^k] = \{ x \in X | \forall t \in \mathfrak t, (t-\lambda(t))^kx = 0 \}.\]
With this notation we have the following result: 
\begin{lemma}
\label{lemma:exdefvermaHom}
Let $\lambda \in X^*(\underline T)^+_I$ and $M = \widetilde{M}_I(\lambda) \otimes_{A_I} A_I/\mathfrak m_I^k$.
Then there is an isomorphism 
\[ \Hom_{U(\mathfrak g)}(M,\Pi) \simeq (\Pi^{\mathfrak n_I} \otimes_L L_I(\lambda)')[\mathfrak m_I^k]\]
of $B$-representations, where $(-)'$ denote the dual (algebraic) representation. In particular, when $I = \emptyset$,
\[ \Hom_{U(\mathfrak g)}(\widetilde{M}(\lambda),\Pi) \simeq (\Pi^{\mathfrak n}(\lambda^{-1}))[\mathfrak m^k] \simeq (\Pi^{\mathfrak n}[(\mathfrak t - \lambda)^k])(\lambda^{-1}).\]
\end{lemma}

\begin{proof}
We compute using the $U(\mathfrak g)$-structure
\begin{eqnarray*} \Hom_{U(\mathfrak g)}(M,\Pi) & 
=& \Hom_{U(\mathfrak g)}(U(\mathfrak g)\otimes_{U(\mathfrak b)} (L_I(\lambda)\otimes_L A_I/\mathfrak m_I^k),\Pi) \\
&=& \Hom_{U(\mathfrak b)}(L_I(\lambda)\otimes_L A_I/\mathfrak m_I^k,\Pi^{\mathfrak n_I})\\
&=& \Hom_{U(\mathfrak t)}(A_I/\mathfrak m_I^k,\Pi^{\mathfrak n_I}\otimes L_I(\lambda)')\\
&=& (\Pi^{\mathfrak n_I}\otimes L_I(\lambda)'))[\mathfrak m_I^k].
 \end{eqnarray*}
 Moreover each equality is compatible with the semi-simplified $B$-actions.
\end{proof}

\begin{lemma}\label{lemm:M_isot_is_la}
  Let $\Pi$ be a locally analytic representation of $\ZZ_p^s\times G$
  and let $M$ be an object of $\mathcal{O}_\alg^\infty$. Then the
  $\ZZ_p^s\times B$-representation $\Hom_{U(\frakg)}(M,\Pi)$ is locally
  analytic.
\end{lemma}

\begin{proof}
  Let $U(\frakg)\otimes_{U(\frakb)}V_1\rightarrow
  U(\frakg)\otimes_{U(\frakb)}V_0\rightarrow M\rightarrow0$ be a
  resolution as in Lemma \ref{lemm:resolution}. Then
  $\Hom_{U(\frakg)}(M,\Pi)$ is the kernel of the map
  \[ \Hom_{U(\frakg)}(U(\frakg)\otimes_{U(\frakb)}V_0,\Pi)\simeq
    (V_0'\otimes_L\Pi)^{\frakb}
    \longrightarrow\Hom_{U(\frakg)}(U(\frakg)\otimes_{U(\frakb)}V_1,\Pi)\simeq
    (V_1'\otimes_L\Pi)^{\frakb} \] which is continuous and
  $B$-equivariant. Therefore $\Hom_{U(\frakg)}(M,\Pi)$ is isomorphic
  to a closed $B$-stable subspace of $V_0'\otimes_L\Pi$. As $V_0$ is
  an algebraic finite dimensional representation of $B$, the
  representation $V_0'\otimes_L\Pi$ is locally analytic and hence so is
  $\Hom_{U(\frakg)}(M,\Pi)$.
\end{proof}

As $\Hom_{U(\frakg)}(M,\Pi)$ is a locally analytic representation of $B$ this action may be derived and induces the structure of an $U(\mathfrak{b})$-module on $\Hom_{U(\frakg)}(M,\Pi)$. Via restriction to $U(\mathfrak{t})\subset U(\mathfrak{b})$ we may view $\Hom_{U(\frakg)}(M,\Pi)$ as an $U(\mathfrak{t})$-module. 
\begin{lemma}\label{lemm:action_Ut}
  Let $\Pi$ be a locally analytic representation of $\ZZ_p^s\times G$
  and let $M$ be an object of $\mathcal{O}_\alg^\infty$. Then the $U(\mathfrak{t})$ action on $\Hom_{U(\frakg)}(M,\Pi)$ factors through a finite dimensional quotient.
\end{lemma}

\begin{proof}
  By Proposition \ref{prop:lift} there exist dominant weights $\lambda_1,\dots,\lambda_r$, finite dimensional $\frakg$-modules $V_1,\dots, V_r$ and a surjective map
  \[ \widetilde M(\lambda_1)\otimes_LV_1\oplus\cdots\oplus \widetilde
    M(\lambda_r)\otimes_LV_r\twoheadrightarrow M. \] Moreover by Lemma
  \ref{prop:lift} there exists $k\geq1$ such that this map factors
  through $\frakm^k$ (recall that $A$ is 
  the localization of $U(\mathfrak t)$ at its augmentation
  ideal $\frakm$). Therefore we have an inclusion of $U(\mathfrak t)$-modules
  \[
    \Hom_{U(\frakg)}(M,\Pi)\hookrightarrow\bigoplus_{i=1}^r
    \Hom_{U(\frakg)}(\widetilde{M}(\lambda_i)\otimes_A A/\mathfrak m^k \otimes_L V_i,\Pi). \]
  By Lemma \ref{lemma:exdefvermaHom}, $ \Hom_{U(\frakg)}(\widetilde{M}(\lambda_i)/\mathfrak m^k \otimes_L V_i,\Pi) = (\Pi \otimes V_i(\lambda_i)')^{\mathfrak n}[\mathfrak m^k]$.
  
  Let $\mu_1,\dots,\mu_s$ be the finitely many characters which
  appears in the restriction to $U(\mathfrak t)$ of
  $V_1(\lambda_1),\dots,V_{r}(\lambda_r)$. Then the action of $U(\mathfrak{t})$ on
  $\Hom_{U(\mathfrak{g})}(M,\Pi)$ factors through the quotient of
  $U(\mathfrak{t})$ by the intersection of the $k$-th powers of the
  kernels of the $\mu_i$.
\end{proof}

\begin{lemma}\label{lemm:hyp_ok}
  Assume that $\Pi$ is an admissible locally analytic
  $L$-representation of $\ZZ_p^s\times G$ satisfying Hypothesis
  \ref{hyp:S_adm_G} and $M \in \cO_\alg^{\infty}$. Then $\Hom_{U(\frakg)}(M,\Pi)$ satisfies
  Hypothesis \ref{hyp:S_adm_B}
\end{lemma}

\begin{proof}
  We can assume that $N_0\subset H$. As we assume Hypothesis \ref{hyp:S_adm_G}, there is an isomorphism
  $\Pi\cong \mathcal{C}^{\la}(\ZZ_p^s\times
  H,L)^m\simeq\mathcal{C}^{\la}(\ZZ_p^s,L)\widehat{\otimes}_L\mathcal{C}(H,L)^m$
  of $\ZZ_p^s\times H$-representation. 
  
  Let
  $[U(\frakg)\otimes_{U(\frakb)}V_1\rightarrow
  U(\frakg)\otimes_{U(\frakb)}V_0]$ be a resolution of $M$ as in Lemma
  \ref{lemm:resolution}. Then
  $\Hom_{U(\frakg)}(M,\mathcal{C}^{\la}(H,L)^m)$ is the kernel of the
  map
  \begin{equation}\label{eqn complex of Hrepns}
  (V_0'\otimes_L\mathcal{C}^{\la}(H,L)^m)^{\frakb}\rightarrow
  (V_1'\otimes_L\mathcal{C}^{\la}(H,L)^m)^{\frakb}.
  \end{equation} 
  We claim that this is a strict map, then the lemma follows, as 
  exactness of the functor
  $\mathcal{C}^{\la}(\ZZ_p^s,L)\hat{\otimes}_L(-)$ implies that we have an
  isomorphism of locally analytic $\ZZ_p^s\times N_0$-representation
  \[
    \Hom_{U(\frakg)}(M,\Pi)\simeq\mathcal{C}^{\la}(\ZZ_p^s,L)\hat{\otimes}_L\Hom_{U(\frakg)}(M,\mathcal{C}^{\la}(H,L))^m.\]
In order to prove that (\ref{eqn complex of Hrepns}) is strict, we use an additional $H$-action.  
We let $H$-act on $\mathcal{C}^{\la}(H,L)$ by right translation and extend this to $V_i'\otimes_L\mathcal{C}^{\la}(H,L)$ by acting trivially on $V_i'$. This action commutes with the (diagonal) action of $U(\mathfrak b)$, as the $U(\mathfrak b)$ action on  $\mathcal{C}^{\la}(H,L)$ is induced by left translations. It follows that $(V_i'\otimes_L\mathcal{C}^{\la}(H,L)^m)^{\frakb}$ is a closed $H$-stable subspace of an admissible locally analytic $H$-representation, and hence an admissible locally analytic $H$-representation itself. Hence (\ref{eqn complex of Hrepns}) is an $H$-equivariant map between admissible locally analytic $H$-representations and hence a strict map which proves the claim.
\qedhere
\end{proof}

\begin{prop}\label{prop:ess_adm}
  Let $\Pi$ be an admissible locally analytic representation of $\ZZ_p^s\times G$ satisfying the
  hypothesis \ref{hyp:S_adm_G} and let $M$ be an object of
  $\cO_{\alg}^{\infty}$. Then the locally analytic representation
  $J_B(\Hom_{U(\frakg)}(M,\Pi))$ of $\ZZ_p^s\times T$ is
  essentially admissible.
\end{prop}

\begin{proof}
  Let
  $U(\frakg)\otimes_{U(\frakb)}V_1\rightarrow
  U(\frakg)\otimes_{U(\frakb)}V_0\rightarrow M\rightarrow0$ be a
  resolution of $M$ given by Lemma \ref{lemm:resolution}. Then we have
  an exact sequence 
  \[  0\longrightarrow\Hom_{U(\frakg)}(M,\Pi)\longrightarrow \Hom_{U(\frakg)}(U(\frakg)\otimes_{U(\frakb)}V_0,\Pi)
    \longrightarrow\Hom_{U(\frakg)}(U(\frakg)\otimes_{U(\frakb)}V_1,\Pi) \]
    of locally analytic representations of
  $\ZZ_p^s\times B$ (see Lemma \ref{lemm:M_isot_is_la}).
  As the functor $J_B$ is left exact
  (\cite[Lem.~3.4.7.(iii)]{EmertonJacquetI}), this induces a short exact
  sequence 
  \begin{align*} 0\longrightarrow J_B(\Hom_{U(\frakg)}(M,\Pi))&\longrightarrow
    J_B(\Hom_{U(\frakg)}(U(\frakg)\otimes_{U(\frakb)}V_0,\Pi))\\
    &\longrightarrow  J_B(\Hom_{U(\frakg)}(U(\frakg)\otimes_{U(\frakb)}V_1,\Pi)) 
    \end{align*}
    of locally analytic representations of $\ZZ_p^s\times T$.
  As the kernel of a morphism between essentially admissible
  representations is essentially admissible
  (\cite[Thm.~3.1.3]{EmertonJacquetI}), it is sufficient to prove that
  $J_B(\Hom_{U(\frakg)}(U(\frakg)\otimes_{U(\frakb)}V,\Pi))$ is
  essentially admissible for any finite dimensional algebraic
  representation $V$ of $B$. As an algebraic representation of $B$ is an
  extension of rank $1$ object, it is sufficient to prove this when $V$ is
  $1$-dimensional and $V^{\mathfrak n}=V$. The left exactness of $J_B$ implies
  that
  \[ J_B(\Hom_{U(\frakg)}(U(\frakg)\otimes_{U(\frakb)}V,\Pi))\simeq
    J_B(\Hom_{U(\frakb)}(V,\Pi^{\mathfrak{n}}))\simeq\Hom_{U(\mathfrak{t})}(V,J_B(\Pi)). \]
  By \cite[Prop.~3.4]{BHS1} (whose proof follows
  \cite[Thm.~0.5]{EmertonJacquetI}), the locally analytic
  representation $J_B(\Pi)$ of $\ZZ_p^s\times T$ is essentially
  admissible. As $U(\mathfrak{t})$ is finitely generated, we conclude
  that $\Hom_{U(\mathfrak{t})}(V,J_B(\Pi))$ is essentially admissible.
\end{proof}

\begin{lemma}\label{lemm:exactness}
  Let $\Pi$ be a locally analytic representation of $\ZZ_p^s\times G$
  satisfying Hypothesis \ref{hyp:S_adm_G}.
  \begin{enumerate}[(i)]
  \item\label{lemm:exactness1} The functor
    $M\mapsto\Hom_{U(\frakg)}(M,\Pi)$ from $\cO_{\alg}^\infty$ to the
    category of locally analytic representations of $\ZZ_p^s\times B$
    is exact.
  \item\label{lemm:exactness2} The functor
    $M\mapsto\Hom_{U(\frakg)}(M,\Pi)^{N_0}$from $\cO_{\alg}^\infty$
to the category of locally convex
    $L$-vector spaces sends short exact sequences on short exact
    sequences with strict maps.
  \end{enumerate}
\end{lemma}

\begin{proof}
  The assertion \ref{lemm:exactness1} is \cite[Lem.~5.2.5]{BHS3}. We
  recall the proof as we will need notation for the proof of
  \ref{lemm:exactness2}. Let $M$ be an object of the category
  $\cO_{\alg}^\infty$. Let $H\subset G$ be a uniform compact open
  pro-$p$-subgroup. Recall (see for example the proof of
  \cite[Prop.~6.5]{STdist}) that
  $\Pi_{|\ZZ_p^s\times H}=\varinjlim_{r<1}\Pi_r$ with
  \[\Pi_r=\Hom_L^{\cont}(D_r(\ZZ_p^s\times H)\otimes_{D(\ZZ_p^s\times
    G,L)}\Pi,L).\] As $M$ is a finitely presented $U(\frakg)$-module,
  we have
  \[
    \Hom_{U(\frakg)}(M,\Pi)\simeq\varinjlim_{r<1}\Hom_{U(\frakg)}(M,\Pi_r)=\varinjlim_r\Hom_{U_r(\frakg)}(M_r,\Pi_r) 
    \]
  with $M_r\coloneqq U_r(\frakg)\otimes_{U(\frakg)}M$.
    Note that
  there exists an integer $m\geq0$ such that
  $\Pi_r\simeq\Hom_L^{\cont}(D_r(\ZZ_p^s\times H),L)^m$. Therefore we
  have
  \begin{align*}
    &\Hom_{D_r(H)}(D_r(H)\otimes_{U(\frakg)}M,\Pi_r)\\\simeq &\Hom_L^{\cont}(D_r(H)\otimes_{U_r(\frakg)}M_r,\Hom_L^{\cont}(D_r(\ZZ_p^s,L),L))^m, \end{align*}
    for $r<1$.
  As the functor $M\mapsto M_r$ is exact and $D_r(H)$ is a finite free
  $U_r(\frakg)$-module, this proves \ref{lemm:exactness1}.

  Now we prove \ref{lemm:exactness2}. As $N_0$ is a compact group and
  $L$ is of characteristic $0$, it is equivalent to prove
  \ref{lemm:exactness2} after replacing $N_0$ by an open
  subgroup. Therefore we can assume that $N_0=H\cap N$ and that
  $H=(\overline{N}\cap H)(T\cap H)(N\cap H)$ where $\overline{N}$ is
  the group of $\QQ_p$-points of the unipotent subgroup of
  $\underline{G}$ opposite to $\underline{N}$. Let $r<1$. The space
  $\Hom_{U(\frakg)}(M,\Pi_r)^{N_0}$ is the space of maps from $M$ to $\Pi_r$ that are equivariant for the actions of $N_0$ and $U(\mathfrak{g})$. Therefore we
  have
  \begin{multline*}
    \Hom_{U(\frakg)}(M,\Pi_r)^{N_0}=\Hom_{U_r(\frakg)\otimes_{U_r(\mathfrak{n})}D_r(N_0)}(M_r,\Pi_r)\\
    \simeq\Hom_L^{\cont}(D_r(H)\otimes_{(U_r(\frakg)\otimes_{U_r(\mathfrak{n})}D_r(N_0))}M_r,\Hom_L^{\cont}(D_r(\ZZ_p^s,L),L))^m.
  \end{multline*}
  As $D_r(H)$ is a finite free right
  $U_r(\frakg)\otimes_{U_r(\mathfrak{n})}D_r(N_0)$-module (see
  \cite[Thm~1.4]{kohlhaase}), this proves the claim.
\end{proof}

\begin{theor}\label{theo:exactness_Jacquet}
  The functor $M\mapsto J_B(\Hom_{U(\frakg)}(M,\Pi))$ from the category
  $\cO_{\alg}^\infty$ to the category of essentially admissible
  representations of $T$ is exact.
\end{theor}

\begin{proof}
  This is essentially a consequence of Lemma
  \ref{lemm:exactness}~\ref{lemm:exactness2} and we conclude as at the
  end of the proof Proposition \ref{prop:Jacquet_complex}.
\end{proof}

\subsection{The case of Banach representations with coefficients}
\label{sec:case-banach-repr}

Let $R$ be a complete local noetherian $\cO_L$-algebra. As above we will write $R^{\rig}$ for the ring of rigid analytic functions on $(\Spf R)^{\rig}$. 
Let $\Pi$ be
an $R$-admissible $R$-Banach representation of the group $G$ (see
\cite[Def.~3.1]{BHS1}). We assume that our representations satisfies
the following property:
\begin{hypothese}\label{hyp:patching}
  there exists an integer $s\geq0$, a local morphism of
  $\cO_L$-algebras $S\coloneqq\cO_L[[\ZZ_p^s]]\rightarrow R$ such that,
  for some (resp.~any) open pro-$p$-subgroup $G_0\subset G$, the
  $S[[G_0]][1/p]$-module $\Pi'\coloneqq\Hom_L^{\cont}(\Pi,L)$ is finite
  free (as a consequence $\Pi$ is also $S$-admissible).
\end{hypothese}
Using the hypothesis, one shows that the $R$-analytic vectors $\Pi^{R-\an}$ and the $S$-analytic vectors $\Pi^{S-\an}$ of $\Pi$ coincide    and they also coincides with the
subspace of $\ZZ_p^s\times G$-locally analytic vectors in $\Pi$ (see
\cite[Prop.~3.8]{BHS1}). We will simply denote this subspace by $\Pi^{\la}$ in what follows. This is a locally analytic representation of
$\ZZ_p^s\times G$ with an action of $R^{\rig}$ commuting with
$G$. Moreover if we forget the $R^{\rig}$-action, the representation
$\Pi^{\la}$ satisfies Hypothesis \ref{hyp:S_adm_G}.

In the following we will write $\widehat T$ for the rigid analytic space of continuous characters of $T$ and $\widehat T_0$ for the space of continuous characters of the maximal compact subgroup $T_0\subset T$. We recall that the ring of rigid analytic functions on $\widehat{T}_0$ is identified with the algebra $D(T_0,L)$ of $L$-valued distributions on $T_0$. 
Restriction to $T_0$ defines a canonical projection $\widehat T\rightarrow \widehat T_0$. Moreover, the derivative of a character at $1$ defines a \emph{weight map} 
\begin{equation}\label{eqn: weight map}
{\rm wt}: \widehat T_0\rightarrow \mathfrak{t}^\ast,
\end{equation}
where by abuse of notation we write $\mathfrak{t}^\ast$ for the rigid analytic space associated to the $L$-vector space $\mathfrak{t}^\ast$. The map ${\rm wt}$ is \'etale and locally finite. Moreover, \'etaleness implies that for any character $\delta_0:T_0\rightarrow L^\times$ we can identify the tangent space of $\widehat T_0$ at $\delta_0$ with the $L$-vector space $\mathfrak{t}^\ast$. 

\begin{lemma}\label{lemm:JB_R_ess_adm}
  For any object $M$ in $\cO_\alg^\infty$, the dual $J_B(\Hom_{U(\frakg)}(M,\Pi^{\la}))'$ of the Emerton-Jacquet module $J_B(\Hom_{U(\frakg)}(M,\Pi^{\la}))$ 
   is coadmissible as an $R^{\rig}\hat\otimes_L\mathcal{O}(\widehat T)$-module. 
\end{lemma}

\begin{proof}
  This is essentially the same proof than for Proposition
  \ref{prop:ess_adm} using the fact that $J_B(\Pi^{\la})$ is
  essentially admissible as a representation of $\ZZ_p^{s'}\times T$
  for any $s'$ and surjection
  $\cO_L[[\ZZ_p^{s'}]]\twoheadrightarrow R$ by \cite[Prop.~3.4]{BHS1}.
\end{proof}

Let $M$ be an object of $\cO_{\alg}^{\infty}$. It follows from Lemma
\ref{lemm:JB_R_ess_adm} that there exists a unique up to unique
isomorphism coherent sheaf $\mathcal{M}_{\Pi}(M)$ on
$\Spf(R)^{\rig}\times\widehat{T}$ such that
\[
  \Gamma(\Spf(R)^{\rig}\times\widehat{T},\mathcal{M}_\Pi(M))=J_B(\Hom_{U(\frakg)}(M,\Pi^{\la}))'. \]
In particular we obtain a functor from $\mathcal{O}_{\rm alg}^\infty$ to the category of coherent sheaves on $\Spf(R)^{\rig}\times \widehat T$. 
\begin{theor}\label{prop:Cohen_Mac}
  The coherent sheaf $\mathcal{M}_\Pi(M)$ on
  $\Spf(R)^{\rig}\times\widehat{T}$ is, locally on
  $\Spf(R)^{\rig}\times\widehat{T}$, finite free over
  $\Spf(S)^{\rig}$. In particular, if nonzero, it is Cohen--Macaulay
  of dimension $s$.
\end{theor}

\begin{proof}
  Let $T_0$ be the maximal compact subgroup of $T$ and let
  $\widehat{T_0}$ be the rigid analytic space of characters of $T_0$
  over $L$. Set $N\coloneqq J_B(\Hom_{U(\frakg)}(M,\Pi^{\la}))'$. It
  follows from the proof of \cite[Prop.~3.11]{BHS1} that there exists
  a family $\mathcal{I}$ of pairs $(U,V)$ where $U$ is a rational open
  subset of $\Spf(R)^{\rig}\times\widehat{T}$ and $V$ is a rational open
  subset of $\Spf(S)^{\rig}\times\widehat{T_0}$ such that $V$ is the
  image of $U$ and such that $\Supp(\mathcal{M}_\Pi(M))\subset\bigcup_{(U,V)\in\mathcal{I}}U$.  Moreover, we may assume that $\Gamma(U,\mathcal{M}_\Pi(M))$ is a finite
  projective $\cO(V)$-module that is a direct factor of
  $\mathcal{O}(V)\hat{\otimes}_{S^{\rig}\hat{\otimes}_LD(T_0,L)}N$. 
  
 After shrinking each $U$ and $V$ if necessary, we may even assume (by the construction of the family $\mathcal{I}$) that for each
  $(U,V)\in\mathcal{I}$, the rational open $V$ is of the form
  $V_1\times V_2$ with $V_1$ rational open in $\Spf(S)^{\rig}$ and
  $V_2$ rational open in $\widehat{T_0}$. It is sufficient to prove
  that, for any pair $(U,V_1\times V_2)\in\mathcal{I}$, the
  $\cO(V_1)$-module $\Gamma(U,\mathcal{M}(M))$ is finitely generated
  and flat.

  The map $V_2\rightarrow\mathfrak{t}^*$ has finite fibers (as the weight map is locally finite), and hence there
  are only finitely many points of $V_2$ lying over a given character of
  $U(\mathfrak{t})$. It thus follows from Lemma \ref{lemm:action_Ut}
  that the action of $L[T_0]$ on $\Gamma(U,\mathcal{M}(M))$ factors
  through a finite dimensional quotient. It follows that
  $\Gamma(U,\mathcal{M}(M))$ is finitely generated over $\cO(V_1)$.

  Let $\frakm\subset\cO(V_1)$ be a maximal ideal. As $\cO(V_1)$ is an
  affinoid $L$-algebra, $\frakm$ is closed in $\cO(V_1)$ and
  $\cO(V_1)/\frakm$ is a finite extension of $L$. As the image of
  $S^{\rig}$ in $\cO(V_1)$ is dense, we have
  $S^{\rig}/(S^{\rig}\cap\frakm)\simeq\cO(V_1)/\frakm$. The ideal
  $\fraka\coloneqq S^{\rig}\cap\frakm$ of $S^{\rig}$ is finitely
  generated by Lemma \ref{lemm:resolution_S}, so that the sheaf
  $\fraka\otimes_{S^{\rig}}\mathcal{M}_\Pi(M)$ is coherent and
  \[
    \Gamma(\Spf(R)^{\rig}\times\widehat{T},\fraka\otimes_{S^{\rig}}\mathcal{M}_\Pi(M))\simeq
    \fraka\otimes_{S^{\rig}}\Gamma(\Spf(R)^{\rig}\times\widehat{T},\mathcal{M}_\Pi(M)). \]
  As the functor $\mathcal{M}\mapsto\Gamma(U,\mathcal{M}_\Pi)$ is
  exact on the category of coherent sheaves, we have an isomorphism
  \[
    \Gamma(U,\fraka\otimes_{S^{\rig}}\mathcal{M}_\Pi(M))\simeq\fraka\otimes_{S^{\rig}}\Gamma(U,\mathcal{M}_\Pi(M))\simeq\frakm\otimes_{\cO(V_1)}\Gamma(U,\mathcal{M}_\Pi(M)). \]
  Therefore we deduce from Proposition \ref{prop:almost_flatness} that
  the map
  \[
    \frakm\otimes_{\cO(V_1)}\otimes\Gamma(U,\mathcal{M}_\Pi(M))\longrightarrow\Gamma(U,\mathcal{M}(M)) \]
  is injective. This implies that $\Gamma(U,\mathcal{M}_\Pi(M))$ is a flat
  $\cO(V_1)$-module.
\end{proof}

\begin{cor}\label{coro:Jacquet_conclusion}
  Assume that the representation $\Pi$ satisfies Hypothesis \ref{hyp:patching}. Then the functor $M\mapsto\mathcal{M}_\Pi(M)$ is an exact functor from
  the category $\cO_{\alg}^\infty$ to the category of Cohen--Macaulay
  sheaves on $\Spf(R)^{\rig}\times\widehat{T}$. Moreover if
  $\mathcal{M}_\Pi(M)$ is nonzero, its support is $s$-dimensional, where $s$ is as in Hypothesis \ref{hyp:patching}.
\end{cor}

\subsection{Comparison with the parabolic Jacquet functor}
\label{sec:comp-with-jacq}

Let $\Pi$ be an $R$-admissible Banach representation of $G$ satisfying
hypothesis \ref{hyp:patching}. We end this section by computing the evaluation of $\mathcal M_\Pi$ on generalized (deformed) Verma  modules in terms of Emerton's parabolic Jacquet-module.

Let $I\subset \Delta$ be a subset of simple roots. Let
$\lambda\in X^*(\underline T)_I^+$ be an algebraic character dominant
with respect to $\frakp_I$. Recall that, by
\cite[\S3.4]{EmertonJacquetI}, the $L$-representation
$J_{P_I}(\Pi^{\la})$ of $L_I$ is locally analytic. Following
\cite[\S5.2]{Wu2}, we define
\begin{align*} 
J_{P_I}(\Pi^\la)_\lambda&\coloneqq\Hom_{U(\mathfrak
    l^{\mathrm{ss}}_I)}(L_I(\lambda),J_{P_I}(\Pi^{\la}))\otimes_LL_I(\lambda) \\
 J_{I,\lambda}(\Pi^{\la})&\coloneqq J_{B\cap L_I}(J_{P_I}(\Pi^{\la})_\lambda). 
 \end{align*}
Similarly to Lemma \ref{lemm:JB_R_ess_adm} we have the following finiteness result:
\begin{prop}\label{prop:J_I_essadm}
  The $R^{\rig}\widehat\otimes_L\cO(\widehat T)$-module
  $J_{I,\lambda}(\Pi^{\la})'$ is coadmissible.
\end{prop}

\begin{proof}
  This is a consequence of \cite[Lemm.~5.1 \& 5.2]{Wu2}.
\end{proof}
By the above proposition there is a coherent sheaf $\mathcal M_{\Pi}^{I,\lambda}$  on
$\Spf(R)^{\rig}\times\widehat T$ such that
\[\Gamma(\Spf(R)^{\rig}\times\widehat T,\mathcal
M^{I,\lambda}_\Pi)=J_{I,\lambda}(\Pi^{\la})'.\]
For $k\geq1$, let $\widehat T_k^{\sm}$ be the $k$-th infinitesimal
neighborhood of the closed subspace $\widehat T^{\rm sm}$ of smooth characters in $\widehat T$ and let $i_k$ be
the closed immersion of $\widehat T_k^{\sm}$ in $\widehat T$. Moreover, for $\lambda \in X^*(T) \subset \widehat{T}$, we write $t_\lambda : \widehat{T}\fleche \widehat{T}$ for the map defined by $t_\lambda(\delta) = \delta\lambda$.

\begin{prop}\label{prop:comp_Jacquet_parabolic}
  Let $\lambda\in X^*(\underline{T})_I^+$ be an algebraic character of
  $\underline{T}$ dominant with respect to $\underline P_I$ and let
  $M=\widetilde
  M_I(\lambda)\otimes_{A_I}A_I/\frakm^k\in\cO_{\alg}^\infty$. Then
  there is an isomorphism of coherent sheaves on
  $\Spf(R)^{\rig}\times\widehat{T}$:
  \[ \mathcal{M}_\Pi(M)\simeq
    i_{k,*}i_k^*t_\lambda^*\mathcal{M}_\Pi^{I,\lambda}. \] \end{prop}

\begin{proof}
  Using the left exactness of the functor $J_{P_I}(-)$, we have an
  isomorphism an $R^{\rig}$-equivariant morphism of locally analytic
  representations of $L_I$ :
  \[
    J_{P_I}(\Hom_{U(\frakg)}(\widetilde
    M_I(\lambda)\otimes_{A_I}A_I/\frakm^k,
    \Pi^{\la}))\simeq\Hom_{U(\mathfrak
      l_I)}(L_I(\lambda)\otimes_LA_I/\frakm^k,J_{P_I}(\Pi^{\la})). \]
  Therefore
  \begin{multline*} \Hom_{U(\mathfrak
      t)}(\lambda\otimes_LA_I/\frakm^k,J_{B\cap
      L_I}(J_{P_I}(\Pi^{\la})_\lambda))\\
      \simeq J_{B\cap L_I}(\Hom_{U(\mathfrak
        t)}(A_I/\frakm^k,\Hom_{U(\mathfrak
        l_I^{\mathrm{ss}})}(L_I(\lambda),J_{P_I}(\Pi^\la))))\\
        = J_{B\cap L_I}(\Hom_{U(\mathfrak
          l_I)}(L_I(\lambda)\otimes_LA_I/\frakm^k,J_{P_I}(\Pi^\la))) \\
        \simeq J_{B\cap
          L_I}(J_{P_I}(\Hom_{U(\mathfrak{p}_I)}(L_I(\lambda)\otimes_LA_I/\frakm^k,\Pi^\la))) \\
        \simeq J_B(\Hom_{U(\frakg)}(\widetilde
        M_I(\lambda)\otimes_{A_I}A_I/\frakm^k,\Pi^\la))
  \end{multline*}
  where the first isomorphism comes from \cite[Lemm.~5.3]{Wu2}. The claim now follows form the fact that the source of this chain of isomorphisms is the dual (of the global sections) of $ i_{k,*}i_k^*t_\lambda^*\mathcal{M}_\Pi^{I,\lambda}$ and the target is the dual of $\mathcal{M}_\Pi(M)$.
\end{proof}

\section{Quasi-trianguline local deformation rings}
\label{sec:quasi-triang-local}

Let $F$ be a finite extension of $\QQ$. We keep the notation of
section \ref{sec:emert-jacq-funct} but we specialize ourselves to the
case
$\underline G =
\Res_{(F\otimes_{\QQ}\QQ_p)/\QQ_p}(\GL_{n,F\otimes_{\QQ}\QQ_p})\simeq\prod_{v
  | p} \Res_{F_v/\QQ_p} \GL_{n,F_v}$. We fix $\underline B$ the upper
triangular Borel subgroup and $\underline T$ the diagonal torus. It is
therefore sufficient to choose $L$ a finite extension of $\QQ_p$
splitting all the $F_v$. 
We point out that, though the field $L$ of coefficients is the same as in the preceding section, the group $\underline G$ in this section should be considered as the Langlands dual group of the group in section \ref{sec:emert-jacq-funct}.

Let $\Sigma_F$ be the set of embeddings of
$F$ in $L$. This set can be decomposed as  $\Sigma_F=\coprod_{v|p}\Sigma_{F_v}$, where
$\Sigma_{F_v}$ is the set of $\QQ_p$-linear embeddings of $F_v$ into
$L$ and where the index set is the set  of places $v$ of $F$
that divide $p$. We have a decomposition
\[ \frakg\simeq(\bigoplus_{\tau\in\Sigma_F}\Lie(\underline
  G)\otimes_{F\otimes_{\QQ}\QQ_p,\tau}L)\simeq\bigoplus_{\tau\in\Sigma_F}\Lie(\GL_{n,L}). \]
Let $\Delta$ be the set of simple roots of $\underline G_L$ with
respect to $\underline B_L$. Then
\[ \Delta=\coprod_{\tau\in\Sigma_F}\Delta_\tau, \quad
  \Delta_\tau=\set{\alpha_{1,\tau}, \dots,\alpha_{n-1,\tau}} \] where
$\alpha_{1,\tau},\dots,\alpha_{n-1,\tau}$ are the simple roots of the
copy of $\Lie(\GL_{n,L})$ corresponding to $\tau$. For
$I\subset\Delta$ we denote $\underline P_I$ the standard parabolic
subgroup of $\underline G_L$ corresponding to $I$. 

\subsection{Local models}
\label{sec:local-models}

Let
$\widetilde{\mathfrak g} \coloneqq \underline G_L\times^{\underline
  B_L}\frakb$ be the Grothendieck--Springer resolution of $\frakg$ (which is considered as a scheme over $L$ not just as a vector space in this section). We
have a closed embedding
$\widetilde{\frakg}\hookrightarrow\underline
G_L/\underline B\times\frakg$ given by
$(g\underline B,X)\mapsto(b\underline B,\Ad(g)X)$ and set
\[ X \coloneqq \widetilde{\frakg} \times_{\mathfrak g}
  \widetilde{\frakg}\subset\underline G_L/\underline
  B_L\times\frakg\times\underline G_L/\underline B_L. \] More
generally if $I\subset\Delta$, we set
\[ \widetilde{\frakg}_{\frakp_I}^0\coloneqq \underline
  G_L\times^{\underline P_I}(\mathfrak z_I\oplus\mathfrak n_I) \]
where we recall that $\underline{P}_I$ is the parabolic subgroup of $\underline{G}$ associated to $\Delta$ and $\mathfrak p_I$ is its Lie algebra. Moreover, we write $\mathfrak{z}_I$ for the center of $\mathfrak{p}_I$ and $\mathfrak{n}_I$ for its unipotent radical.
Again we consider all these $L$-vector spaces as $L$-schemes.
We
have also a closed embedding
$\widetilde{\frakg}_{\frakp_I}^0\hookrightarrow\underline
G_L/\underline P_I\times\frakg$ given by
$(g\underline P_I,X)\mapsto(g\underline P_I,\Ad(g)X)$ and we set
\[ X_{\frakp_I} \coloneqq
  \widetilde{\frakg}_{\frakp_I}^0\times_{\frakg}\widetilde{\frakg}\hookrightarrow
  \underline G_L/\underline P_I\times\frakg\times\underline
  G_L/\underline B_L. \] In particular we have $X_\mathfrak b = X$. There scheme $X_{\frakp_I}$ decomposes into irreducible components as follows:
\[ X_{\frakp_I} = \bigcup_{w \in W_I\backslash W} X_{\frakp_I,w}
  \subset \underline G_L/\underline P_I\times \frakg \times \underline
  G_L/\underline B.\] Here $X_{\frakp_I,w}$ is the closure of on open subset $V_{\mathfrak p_I,w}\subset X_{\frakp_I}$, which is by definition the preimage of the $\underline G$-orbit of
$\underline G\cdot(1,\widetilde w)\subset \underline G/\underline P_I \times \underline G/\underline B$, where $\widetilde w\in W$ is a lift of $w\in W_I\backslash W$ (see \cite[Cor.~5.2.2]{BreuilDing} for details). 
In this paper we need to control the singularities of $X_{\frakp_I}$. Even though, for our purpose, the result of \cite[Rk.~4.1.6]{BHS3} would be sufficient, we mention the following more general result.
\begin{prop}\label{prop:smoothpartialsteinbergcomp}
 Let $ w \in W$. Then $X_w$ is smooth if, and only if, $w$ is a product of distinct simple
  reflections.
\end{prop}

\begin{proof}
We note that the natural action
\[t \cdot (g\underline B,h\underline B,N) = (g\underline B,h\underline B,tN)\] of $\mathbb{G}_m$ on $X$ by scaling on the $\mathfrak g$-factor extends to an action of the monoid $\mathbb{A}^1$. This action obviously preserves each $X_w$. 
As the singular locus is closed the non-singular locus, if non-empty, contains a point of the form $(g\underline{B}h\underline{B}, 0)$
We will thus prove the previous proposition using \cite{BHS3} Proposition 2.5.3 (ii). 

We first assume that $w$ is a product of distinct simple reflections. In this case it is enough to prove that
\begin{enumerate}[a)]
\item\label{smooth1} $\overline U_w$ is smooth in $\underline G_L/\underline B\times \underline G_L/\underline B$;
\item\label{smooth2} $\mathfrak t^{ww'^{-1}}$ has codimension $\lg(w)-\lg(w')$ in $\mathfrak t$ for all $w' \leq w$ for Bruhat ordering (with $\lg$ the Bruhat length). 
\end{enumerate}
By Fan's Theorem \cite[Theorem 7.2.14]{BL}, if $w$ is a product of
distinct simples reflexions, then $\overline U_w$ is smooth and
\ref{smooth1} is true. Thus we only need to prove \ref{smooth2}.
For $w \in W$, let us introduce
\[ \ell(w) \coloneqq \min \{ k \geq 0 \mid w = r_1 \dots r_k, r_k \in W \text{ a reflection}\}\]
(we recall that reflection is an element of the form $s_\alpha$ where $\alpha\in\Phi$ is a root, but not necessarily a simple root).
By \cite[Lemma 2]{Carter} and \cite[Lemma 2.7]{BHS2} we have $\ell(w) = \dim_L \mathfrak t - \dim_L \mathfrak t^w = d_w$ (in the notations of \cite{BHS2}).

\begin{claim}\label{claim:ell}
If $w$ is a product of distinct simple reflections, we have \[\ell(ww'^{-1}) = \ell(w)-\ell(w') = \lg(w)-\lg(w')\] for all $w' \leq w$.
\end{claim}

If Claim \ref{claim:ell} is true, we have $\ell(ww'^{-1}) = \dim \mathfrak t - \dim \mathfrak t^{ww'^{-1}} = \lg(w)-\lg(w')$ thus Proposition 2.5.3 of \cite{BHS3} applies and $X_w$ is smooth. We now prove the claim. The second equality of the claim is a consequence of \cite[Lemma 3]{Carter} as $w$ and $w'$ are products of pairwise distinct simple reflections. Indeed, a product of pairwise distinct simple reflexions $s_1\dots s_k$ is always a composition of reflections $s_i$ along vectors $v_i$ such that $v_1,\dots,v_k$ are linearly independent. Thus \cite[Lemma 3]{Carter} implies $\ell(w) = \lg(w)$ and $\ell(w') = \lg(w')$.

We write $w' = s_{i_1} \dots s_{i_k}$ and $w = t_1 \dots t_b$ as reduced expressions of pairwise distinct simple roots such that there exists $a_1\leq \dots \leq a_k$ satisfying $t_{a_j} = s_{i_j}$. For $a_t \leq j < a_{t+1}$ let $r_j$ denote the reflection $r_j := s_{i_1} \dots s_{i_t}t_j s_{i_t}\dots s_{i_1}$. We then have
\begin{align*} ww'^{-1} =& t_1 \dots t_b s_{i_k} \dots s_{i_1} \\ =& t_1 \dots t_{a_1-1} [\underbrace{s_{i_1}t_{a_1+1}s_{i_1}}_{r_{a_1+1}}] \dots [\underbrace{s_{i_1} t_{a_2-1} s_{i_1}}_{r_{a_2-1}}] [\underbrace{s_{i_1}s_{i_2}t_{a_2+1} s_{i_2}s_{i_1}}_{r_{a_2+1}}] \\ &\dots [s_{i_1} \dots s_{i_k} t_{a_k+1} s_{i_k}\dots s_{i_1}] \dots [\underbrace{s_{i_1}\dots s_{i_k}t_{b}s_{i_k}\dots s_{i_1}}_{r_{b}}]\\
 = &t_1 \dots t_{a_1-1} r_{a_1+1} \dots r_{a_2-1} r_{a_2+1}\dots  \dots r_b.
\end{align*}
In particular, $\ell(ww'^{-1}) \leq \lg(w) -
\lg(w')=\ell(w) - \ell(w')$. Now Claim \ref{claim:ell} follows from
\begin{claim}\label{claim:elllowerbound}

Let $w \in W$ and $w'$ be a product of distinct simple reflections. Then $\ell(ww'^{-1}) \geq \ell(w) - \ell(w') = \ell(w)-\lg(w')$.
\end{claim}

We now prove Claim \ref{claim:elllowerbound}. By induction on 
the number of simple reflexions appearing in $w'$, it is enough to prove $\ell(ws) \geq \ell(w)-1$ when $w' = s$ is a simple reflexion.
Note that for any $w$ we have $\dim_L \mathfrak t^{ws} \cap \mathfrak t^s \geq \dim_L \mathfrak t^{ws} - 1$ as $\mathfrak t^s$ is a hyperplane in $\mathfrak t$. Moreover, $\mathfrak t^{ws} \cap \mathfrak t^s = \mathfrak t^w \cap \mathfrak t^s\subset \mathfrak t^w$. Thus $\dim \mathfrak t^w \geq \dim \mathfrak t^{ws} - 1$. 
 Using $\ell(w) = \dim \mathfrak t - \dim \mathfrak t^{w}$ we hence find
 \[  \ell(w) \leq \ell(ws)+1.\]
Thus $\ell(ws) \geq \ell(w) - 1$, which proves Claim \ref{claim:elllowerbound}.

We now prove the converse, i.e.~that  $X_w$ is singular, if $w$ is \emph{not} a product
of distinct simple reflections.  We hence assume that $w$ is not a product of distinct simple reflections.

It is enough (but actually equivalent) to prove that $X_w$ is singular at $(\underline B,\underline B,0)$. We will use Mowlavi's
results \cite{MowlaviGS}. The pair $(1,w)$ is a good pair
(\cite{MowlaviGS}), and thus \cite[Theorem 6]{MowlaviGS} applies. Hence \cite[Proposition 3.2.2]{MowlaviGS} gives an exact formula for the tangent space at $x = (\underline B,\underline B,0) \in (X_w \cap V_1)(L)$. This can be rewritten as
\begin{align*} 
\dim_L T_{x}X_w &= \dim_L T_{\pi(x)}\overline{U_w} - d_w + \dim_L \mathfrak t +
  \lg(w_0) \\ &> \dim \underline B + \lg(w) - \lg(w) + \dim_L \mathfrak t +
  \lg(w_0),
  \end{align*}
 as $w$ is not a product of distinct simples so $\lg(w) > d_w$ (\cite{BHS2} Lemma 2.7), and where we use the notation \footnote{see \cite{BHS3} just before Proposition 4.1.5} $d_w = \dim_L \mathfrak t - \dim_L \mathfrak t^w$.
 Thus
\[ \dim_L T_{x}X_w > 2 \dim \underline B + \dim_L \mathfrak t = \dim \underline G_L = \dim X_w,\]
i.e.~$X_w$ is not smooth at $x$.
\end{proof}

We write $X_I$ for the inverse image of $X_{\frakp_I}$ under the canonical projection
$\underline G_L/\underline B_L\times\frakg\times\underline G_L/\underline
B_L \rightarrow \underline G_L/\underline P_I\times\frakg\times\underline G_L/\underline
B_L$. This scheme can also be defined as
\[ X_I\coloneqq(\underline G_L\times^{\underline B_L}(\mathfrak
  z_I\oplus\mathfrak n_I))\times_{\frakg}\widetilde{\frakg},\] in
particular $X_\emptyset = X$.  The map $X_I\rightarrow X_{\frakp_I}$
is a $\underline P_I/\underline B_L$-torsor and thus is projective and
smooth. We deduce that we have a decomposition in irreducible
components
\[ X_I=\bigcup_{w\in W_I\backslash W} X_{I,w}, \] where each
$X_{I,w}\rightarrow X_{\frakp_I,w}$ is projective and smooth. Moreover,
we have a closed embedding $X_I \hookrightarrow X$ induced by the
closed embedding $\mathfrak z_I \oplus \mathfrak n_I \hookrightarrow
\mathfrak b$, and this induces a closed embedding $X_{I,w} \hookrightarrow X_{w^{\rm max}}$, as each fiber of $X_I \fleche X_{\mathfrak p_I}$ over a point in $V_{\mathfrak p_I,w}$ 
contains a (dense) open subset consisting of points that lie in the Schubert cell 
$\underline G_L(1,w^{\rm max}) \subset \underline G/\underline B\times \underline G/\underline B.$

\begin{lemma}\label{lemm:generically_reduced}
  The schemes $X_I$ and $X_{\frakp_I}$ are generically reduced.
\end{lemma}

\begin{proof}
 As $X_I$ is smooth
  over $X_{\frakp_I}$, it suffices to prove the claim for $X_{\frakp_I}$. For $w\in W$, let $U_w=\underline
  G_L(1,w)\subset\underline G_L/\underline P_I\times\underline
  G_L/\underline B$ and let $V_w\subset X_{\frakp_I}$ be the inverse
  image of $U_w$. It follows from \cite[Prop.~5.2.1]{BreuilDing} that
  the $V_w$ are smooth $L$-schemes, and they all have the same dimension. As they also cover $X_{\frakp_I}$, their generic points are the
  generic points of the irreducible components of $X_{\frakp_I}$. This shows that $X_{\frakp_I}$ is
  generically reduced.
\end{proof}

Recall that we have two maps $\kappa_1, \kappa_2 : X
\rightarrow\mathfrak t$ (see \cite[\S2.3]{BHS3}) defined by
$\kappa_i(g_1\underline B,N,g_2\underline B)=g_i^{-1}Ng_i (\mathrm{mod}\ \mathfrak n)$. By construction, the image of $\kappa_{1|X_I}$ lands in $\mathfrak
z_I$ and the map $\kappa_{1|X_I}$ factors through
$X_{\frakp_I}$. This provides a commutative diagram
\[
  \begin{tikzcd}
    X_I \ar[r,twoheadrightarrow] \ar[rd,"\Theta_I"'] & X_{\frakp_I}
    \ar[d,"\Theta_{\frakp_I}"] \\
    & \mathfrak z_I\times_{\mathfrak t/W} \mathfrak t
  \end{tikzcd} \]
where $\Theta_I$ is the restriction of the map
$(\kappa_1,\kappa_2)$ to $X_I$.

The following result is the analogue of \cite[Lem.~2.5.1]{BHS3} in our
context, with analogous proof.

\begin{lemma}\label{lemm:comp_z_ixt}
  The irreducible components of $\mathfrak z_I\times_{\mathfrak
    t/W}\mathfrak t$ are the $(T_{I,w})_{w\in W_I\backslash W}$ where
  \[ T_{I,w}=\set{(z,\Ad(w^{-1})(z)) \mid z\in\mathfrak z_I} .\] Moreover, the irreducible component
  $X_{I,w}$ (resp.~$X_{\frakp_I,w}$) is the unique   component of $X_I$ (resp.~$X_{\frakp_I}$) whose image under $\Theta_I$
  (resp.~$\Theta_{\frakp_I}$) dominates $T_{I,w}$.
\end{lemma}

\begin{rema}\label{defi:Xn}
For future use, we make the following notational convention:  When $F=\QQ$, we have $\underline G_L=\GL_{n,L}$, we will use the notations $X_n$,
  $X_{n,I}$, $X_{n,I,w}$ etc.~for the schemes $X$, $X_I$, $X_{I,w}$
  etc.
\end{rema}

\subsection{Partially de Rham deformation rings}
\label{sec:quasi-parab-deform}

For each place $v|p$ of $F$, we fix
$r_v : \Gal_{F_v}\rightarrow\GL_n(L)$ a framed
$\varphi$-generic Hodge--Tate regular crystalline representation, that
we assume that the 
$(\varphi,\Gamma)$-module $D_{\rig}(r_v)$ associated to $r_v$ is crystalline
$\varphi$-generic with regular Hodge--Tate type in the sense of
\cite[\S3.3\&\S3.4]{HeMaS}. We also fix a refinement
$\mathcal R_v=(\varphi_1,\dots,\varphi_n)\in L^n$  of
$r_v$ (see \emph{loc.~cit.}). We will use the notation $r=(r_v)_{v|p}$ and
$\mathcal R=(\mathcal R_v)_{v|p}$ and  say that $r$ is
\emph{$\varphi$-generic Hodge--Tate regular} and that $\mathcal R$ is a
\emph{refinement} of $r$.

Let $\mathcal C_L$ be the category of local artinian $L$-algebras. Fix
$v|p$ a place of $F$. Let $\mathcal X_{r_v}^{\Box}$ be the groupoid
over $\mathcal C_L$ of deformations of $r_v$. It is represented by a
formal scheme over $L$ that we also denote by $\mathcal X_{r_v}^\Box$ by abuse of
notation.  We recall from  \cite[3.6]{BHS3} that, given the refinement  $\mathcal{R}_v$, the groupoid of trianguline deformations of $\mathcal M_{\bullet,v}$ 
is representable by a closed formal subscheme $\mathcal X_{r_v,\mathcal R_v}^{\qtri}\subset \mathcal X_{r_v}^\Box$. Here $\mathcal M_{\bullet,v}$ the
$(\varphi,\Gamma)$-module over $\mathcal R_{K,L}[1/t]$ obtained from $D_{\rig}(r_v)$ by inverting $t$ which is equipped with the unique triangulation corresponding
to the refinement $\mathcal R_v$. We set
$W_v=W_{\dR}(D_{\rig}(r_v)[1/t])$ and
$W_{\bullet,v}=W_{\dR}(\mathcal M_{\bullet,v})$ and let
$X_{W_v,W_{\bullet,v}}$ denote the groupoid of deformations of
$(W_v,W_{\bullet,v})$ as defined in \cite[\S3.3]{BHS3}.

Fix a finite subset  $I_v\subset\Delta_v$. For an object $A$ of $\mathcal C_L$, we define
$X_{W_v,W_{\bullet,v}}^{\underline P_{I_v}}(A)$ to be the subset of
all $(W_A,W_{A,\bullet})\in X_{W_v,W_{\bullet_v}}(A)$ such that for any
$\tau\in\Sigma_{F_v}$ and
$\alpha_{i,\tau}\in \Delta_\tau\setminus I_v$, the
$\B_{\dR}^+$-representation
$W_{A,i}\otimes_{K,\tau}L/W_{A,j+1}\otimes_{K,\tau}L$ is de Rham,
where $j$ is the largest integer $<i$ such that
$\alpha_{\tau,j}\notin I$ (and $j=0$ if $i$ is the smallest integer
such that $\alpha_{i,\tau}\notin I$). It is obvious from the definition that
$X_{W_v,W_{\bullet,v}}^{\underline P_I}$ is a subgroupoid of
$X_{W_v,W_{\bullet,v}}$.

For an object $A$ of $\mathcal C_L$ and
$r_A\in X_{r_v,\mathcal R_v}^{\qtri}(A)$, we denote by
$\mathcal M_{A,\bullet}$ the unique triangulation of $D_{\rig}(r_A)$
lifting $\mathcal M_{\bullet,v}$. We say that $r_A$ is
\emph{$\underline P_I$-de Rham} if
\[(W_{\dR}(r_A),W_{\dR}(\mathcal M_{A,\bullet}))\in
X_{W_v,W_{\bullet,v}}^{\underline P_{I_v}}(A)\]  (see \cite[Def.~3.10 ]{Wu2}). It now follows from \cite[Lemm.~3.11]{Wu2} that this functor is
representable by a closed formal subscheme of
$\mathcal X_{r_v,\mathcal R_v}^{\qtri}$ that we denote
$\mathcal X_{r_v,\mathcal R_v}^{I_v-\qtri}$.
More precisely, we have an isomorphism of groupoids
\[ \mathcal X_{r_v,\mathcal R_v}^{I_v-\qtri}\simeq\mathcal
  X_{r_v,\mathcal
    R_v}^\qtri\times_{X_{W_v,W_{\bullet,v}}}X_{W_v,W_{\bullet,v}}^{P_{I_v}}. \]

Fix an $L\otimes_{\QQ_p}F_v$-basis $\alpha_v$ of $W_v^{\Gal_K}$ and
let $X_{W_v}^\Box$ be the groupoid of deformations of the pair
$(W_v,\alpha_v)$. We set
\begin{align*}
X_{W_v,W_{\bullet,v}}^\Box&=X_{W_v}^\Box\times_{X_{W_v}}X_{W_v,W_{\bullet,v}}\\
\mathcal X_{r_v,\mathcal R_v}^{I_v-\qtri,\Box}&=\mathcal
X_{r_v,\mathcal R_v}^{I_v-\qtri}\times_{X_{W_v}}X_{W_v}^\Box.
\end{align*}
As the map
$\mathcal X_{r_v,\mathcal R_v}^\qtri\rightarrow
X_{W_v^+}\times_{X_{W_v}}X_{W_v,W_{\bullet,v}}$ is formally smooth by
\cite[Cor.~3.5.6]{BHS3}, we deduce that the map
$\mathcal X_{r_v,\mathcal R_v}^{I_v-\qtri,\Box}\rightarrow
X_{W_v^+}\times_{X_{W_v}}X_{W_v,W_{\bullet,v}}^{P_{I_v},\Box}$ is
formally smooth as well.

If $I=\coprod_{v|p}I_v\subset\Delta$ and if $\alpha=(\alpha_v)_{v|p}$
is fixed, we set
$\mathcal X_{r,\mathcal R}^{I-\qtri}\coloneqq\prod_{v|p}\mathcal
X_{r_v,\mathcal R_v}^{I_v-\qtri}$ and
$\mathcal X_{r,\mathcal R}^{I-\qtri,\Box}\coloneqq\prod_{v|p}\mathcal
X_{r_v,\mathcal R_v}^{I_v-\qtri,\Box}$.

We consider the point
\begin{equation}\label{eqn: def xpdr}
x_{\pdR} \coloneqq (g\underline B_L,0,h\underline B_L) \in X_I(L)
\subset (\underline G_L/\underline B_L\times \frakg\times \underline
G_L/\underline B_L)(L),
\end{equation} where $g\in \underline G(L)$ (resp.~$h$) is the matrix sending the standard
flag (corresponding to our fixed basis $\alpha$) of $\prod_{v|p}W_v^{\Gal_K}$ to
the complete flag
$\prod_{v|p}W_{\dR}(\mathcal M_{\bullet,v})^{\Gal_K}$ (resp.~to the
Hodge flag). We deduce the following result (see \cite[\S6.3]{BreuilDing} in a
slightly different context):
\begin{theor}\label{thm:formsmoothdiag}
  There exists a diagram of formal $L$-schemes with formally smooth maps  \[
    \begin{tikzcd}
      \mathcal X_{r,\mathcal R}^{I-\qtri} & \mathcal X_{r,\mathcal
        R}^{I-\qtri,\Box} \ar[l,"g"'] \ar[r,"f"] &
      \widehat{X}_{I,x_{\pdR}}
    \end{tikzcd} \]
\end{theor}

\begin{proof}
  Let $I=\coprod_{v\in S_p}I_v$, with $I_v\subset \Delta_v$ for
  $v\in S_p$. Note that we have a decomposition
  $X_I\simeq\prod_{v\in S_p} X_{I_v}$ where $X_{I_v}$ is the
  $L$-scheme defined in the same way as $X_I$ but for the group
  $\Res_{F_v/\QQ_p}\GL_{n,F_v}$. We also write
  $x_{\pdR}=(x_{\pdR,v})_{v\in S_p}$ where $x_{\pdR,v}$ is the image
  of $x_{\pdR}$ in $X_{I_v}$. We just have to check that the groupoid
  \[X_{W_v^+}\times_{X_{W_v}}X_{W_v,W_{\bullet,v}}^{\underline
    P_{I_v}}\times_{X_{W_v}}X_{W_v}^\Box\] is represented by the
  completion of $X_{I_v}$ at $x_{\pdR,v}$. This can be checked easily
  as in the proof of \cite[Lemm.~3.11]{Wu2} using \cite[Cor.~3.1.9
  \&Thm.~3.2.5]{BHS3}.
\end{proof}

We finally note that the map $\kappa_1$ from above induces a map of formal schemes $\kappa_1 : \widehat
X_{I,x_{\pdR}}\rightarrow\widehat{\mathfrak z_I}$, where
$\widehat{\mathfrak z_I}$ is the completion of $\mathfrak z_I$ at $0$,
and thus a map \[\mathcal X_{r,\mathcal
  R}^{I-\qtri,\Box}\rightarrow\widehat{\mathfrak z_I}.\] This maps
factors into a map of formal schemes
$ \kappa_1 : \mathcal X_{r,\mathcal
    R}^{I-\qtri}\longrightarrow\widehat{\mathfrak z_I}$.

For $w\in W$ such that $x_{\pdR}\in X_{I,w}(L)$, we denote by
$\mathcal X_{r,\mathcal R}^{I-\qtri,w}$ the schematic image of
\[\mathcal X_{r,\mathcal
  R}^{I-\qtri,\Box}\times_{\widehat{X}_{I,x_{\pdR}}}\widehat{X}_{I,w,x_{\pdR}}\rightarrow \mathcal X_{r,\mathcal R}^{I-\qtri}\] and by
$\overline{\mathcal X}_{r,\mathcal R}^{\qtri}$
(resp.~$\overline{\mathcal X}_{r,\mathcal R}^{I-\qtri,w}$) the
schematic inverse image of $\set{0}$ under $\kappa_1$ in
$\mathcal X_{r,\mathcal R}^{I-\qtri}$
(resp.~$\mathcal X_{r,\mathcal R}^{I-\qtri,w}$).

The schemes $\mathcal X_{r,\mathcal R}^{I-\qtri}$ and
$\mathcal X_{r,\mathcal R}^{I-\qtri,w}$ are formal spectra of complete local noetherian rings that we denote by $R_{r,\mathcal R}^{I-\qtri}$ and $R_{r,\mathcal R}^{I-\qtri,w}$. It follows from the constructions that moreover $R_{r,\mathcal R}^{I-\qtri,w}$ is an
integral local ring.

\section{Global construction}
\label{sec:global}

Let $F$ be a totally real number field and let $E/F$ be a totally
imaginary CM extension of number fields, in particular $[E:F] = 2$. 
We assume that all places of $F$ dividing $p$ are unramified and split in $E/F$
and denote by $S_p$ the set of places above $p$ in $F$. We fix a set $\Sigma$ of
places of $E$ dividing $p$ such that, for each place $v\in S_p$, there
is exactly one place of $\Sigma$ above $v$.  Let
$U$ be a unitary group in $n$ variables for $E/F$ that we regard, via Weil restriction, as an
algebraic group over $\QQ$. We assume that
$U(\RR)$ is compact and and that $U_{\QQ_p}$ is quasi-split.  This implies in particular that there exists an isomorphism
$U_{\QQ_p} \simeq \prod_{v\in S_p} \Res_{F_v/\QQ_p}\GL_{n,F_v}$ that
we fix from now on.  
From now we note $\underline G=U_{\QQ_p}$ identified with
$\prod_{v\in S_p}\Res_{F_v/\QQ_p}\GL_{n,F_v}$ via this fixed
isomorphism and we use notations of section
\ref{sec:emert-jacq-funct}, i.e.~$L$ is the choice of a field of coefficients that is assumed to be big enough so that all
embeddings of $E$ (equivalently of $F$) in $\overline \QQ_p$ factor
through $L$. Moreover, $\underline B\subset\underline G_L$
is the Borel subgroup of upper triangular matrices,
$\underline T\subset\underline B$ is the maximal torus of diagonal
matrices, $\underline N$ is the unipotent radical of $\underline B$ etc.

\subsection{Classical and $p$-adic automorphic forms}
\label{sec:p_adic_mf}
We write $T=\underline{T}(\QQ_p)\simeq\left(\prod_{v \in S_p}
    F_v^\times\right)^n$ and let $T_0\simeq\left(\prod_{v\in S_p} \mathcal
    O_{F_v}^\times\right)^n\subset T$ denote its maximal compact subgroup.  We denote by $\widehat{T}$ (resp.~$\widehat{T}_0$) the rigid analytic
  spaces over $L$ parametrizing the continuous characters of
  $T$ (resp.~of
  $T_0$) and recall from \ref{eqn: weight map} that there is a weight map 
  \[{\rm wt}: \widehat{T}_0\rightarrow \mathfrak{t}^\ast\]
   with values in the dual Lie algebra $\mathfrak{t}^\ast$ of $\underline{T}$ (considered as a rigid space over $L$).
    We will often, by abuse of notation, also write ${\rm wt}$ for the composition of ${\rm wt}$ with the canonical projection $\widehat{T}\rightarrow \widehat{T}_0$.
  Recall that we had identified $X^\ast(\underline{T})$ with a $\ZZ$-lattice in $\mathfrak{t}^\ast$. 
  Often we will identify $X^\ast(\underline{T})$ with $\ZZ^{n[F:\QQ]}$.
\begin{defin}\label{defin:weights}
Let $\delta\in \widehat{T}$ (resp.~$\in \widehat{T_0}$) be a character.\\
\noindent (i) The \emph{weight} of $\delta$ is the image ${\rm wt}(\delta)$ under the weight map. \\
\noindent (ii) The character $\delta$ is called \emph{of algebraic weight} if ${\rm wt}(\delta)\in X^\ast(\underline{T})\subset \mathfrak{t}^\ast$. \\
\noindent (iii) The character $\delta$ is called \emph{algebraic} if it is of the form 
\[\delta_{\underline k}: (z_1\otimes1,\dots,z_n\otimes1)\longmapsto \prod_{\tau}
    \left(\tau(z_1)^{k_1^\tau} \cdots \tau(z_n)^{k_n^\tau}\right)\]
for some $\underline k=(k_1^\tau,\dots,k_n^\tau)_{\tau : F\hookrightarrow
    L}\in\ZZ^{n[F:\QQ]}$. It is called \emph{dominant algebraic} if $\underline k\in X^\ast(\underline{T})^+$, i.e.~if $k_1^\tau \geq \dots \geq k_n^\tau$ for all $\tau$.
\end{defin}
Note that $\underline{k}\mapsto \delta_{\underline{k}}$ defines a section of the weight map over the algebraic weights, and we use this map to identify $X^\ast(\underline{T})$ with a subset of $\widehat{T}$ (resp.~$\widehat{T}_0$).

Let $K^p \subset U(\mathbb A^{\infty,p})$ be a compact open subgroup,
called a \emph{tame level} that we assume to be of the form
$\prod_{\ell\neq p}K_\ell$ where $K_\ell$ is a compact open subgroup
of $U(\QQ_\ell)$. Let $I_p$ be the Iwahori subgroup of
$G=\underline G(\QQ_p)=U(\QQ_p)$ with respect to our choice of $\underline
B$. 
For any compact open $K_p \subset U(\QQ_p)$ we consider the Shimura set
\[ Sh_{K^pK_p} \coloneqq U(\QQ)\backslash U(\mathbb A^\infty)/K^pK_p.\]
As $U(\RR)$ is compact, this is indeed a finite set of points.

\begin{defin}\label{def:completed_cohomology}
The \emph{completed cohomology} of the tower $(Sh_{K^pK_p})_{K_p\subset U(\QQ_p)}$ of Shimura sets is:
\[\Pi \coloneqq \Pi^\circ\otimes_{\cO_L}L, \quad \text{with}\quad \Pi^\circ \coloneqq \varprojlim_n \varinjlim_{K_p}
    H^0(Sh_{K^pK_p},\mathcal
    O_L/\pi_L^n),\]
    see \cite{Emint}.
    \end{defin}

The completed cohomology is an $L$-Banach space endowed with a continuous action of
$U(\QQ_p)$. This space is naturally identified with the space of continuous functions
\begin{equation}\label{eq:padic_forms}
f : U(\QQ)\backslash U(\mathbb A^\infty)/K^p \fleche L.
\end{equation}
We denote $\Pi^{\la}$ the subspace of locally analytic vectors in
$\Pi$ for $U(\QQ_p)$. This is the subspace of functions in
(\ref{eq:padic_forms}) which are locally analytic. As $\Pi^{\rm la}$ is a locally analytic representation, there is a natural $U(\mathfrak{g})$-action on $\Pi^{\rm la}$ obtained by deriving the $G=\underline{G}(\QQ_p)$-action. Here, as above, we write $\mathfrak{g}$ for the Lie algebra of $G$, and $\mathfrak{b}, \mathfrak{t}, \mathfrak{n}$ for the Lie algebras of the Borel $B$, of the torus $T$ and of the unipotent radical $N$ of $B$. 
\begin{defin}
The space of \emph{overconvergent $p$-adic automorphic forms of tame weight $K^p$} is the space 
\[S^\dagger(K^p)=(\Pi^{\la})^{\mathfrak n} = \varinjlim_{N_0\subset\underline
    N(\QQ_p)} (\Pi^{\la})^{N_0},\] where $N_0$ varies among the
  compact open subgroups of $\underline N(\QQ_p)$.
Given a weight $\kappa\in \mathfrak{t}^\ast$, the space of overconvergent $p$-adic automorphic forms of tame weight $K^p$ and \emph{weight} $\kappa$ is the eigenspace  \[S^\dagger_\kappa(K^p)\subset S^\dagger(K^p)\]
of eigenvalue $\kappa$ for the $U(\mathfrak{t})$-action. 
\end{defin}

Denote by
$\mathbb T(K^p)\coloneqq\ZZ[K^p\backslash U(\mathbb
A^{\infty,p})/K^p]$ the Hecke algebra of Hecke operators over $\ZZ$ of
tame level $K^p$. Then $\mathbb T(K^p)$ acts by convolution on
$S^\dag(K^p)$ and $S^\dagger_\kappa(K^p)$. Let $S$ be a finite set of prime numbers
containing $p$ and all the $\ell$ such that $K_\ell$ is not
hyperspecial. The subalgebra
$\mathbb T^S \coloneqq \bigotimes_{\ell \notin S} \mathbb T_\ell
\subset \mathbb T(K^p)$ is commutative.

\begin{defin}\label{def:Hecke_action}
 Let
  \[\underline T(\QQ_p)^+ \coloneqq \set{ \diag(a_1^v,\dots,a_n^v)_v \in \underline
      T(\QQ_p) \mid v(a_1^v) \geq \dots \geq v(a_n^v), \ \forall v\in
      S_p}.\] The \emph{Atkin--Lehner ring} $\mathcal{A}(p)$ is the sub-algebra of
  $\ZZ[\underline T(\QQ_p)]$ generated by the elements
  $t\in\underline T(\QQ_p)^+$. \\
  \end{defin}
   Let $\delta: T\rightarrow L^\times$ be a continuous character. Then we can extend $\delta$ to a character $\mathcal{A}(p)\rightarrow L$ whose restriction to $T^+$ is given by $\delta$. By abuse of notation we still write $\delta$ for this character of $\mathcal{A}(p)$.

Note that there is a cofinal system of compact open subgroups $N_0\subset N=\underline N(\QQ_p)$ such that $tN_0t^{-1}\subset N_0$ for all $t\in T^+$.  
We hence can define a Hecke action of $\mathcal{A}(p)$ on $S^\dagger(K^p)=(\Pi^{\la})^{\mathfrak n}$ by letting $t\in\underline T(\QQ_p)^+$ act on $f\in (\Pi^{\la})^{N_0}$ via
  \[ [t] f \coloneqq\left(x \mapsto \frac{1}{[N_0:tN_0t^{-1}]}\sum_{n
        \in N_0/tN_0t^{-1}} f(xnt)\right),\] where $N_0$ is a sufficiently small compact open subgroup of $N$ such that
  $f \in (\Pi^{\la})^{N_0}$ and such that $tN_0t^{-1}\subset N_0$.

Let $\mathbb{T}$ be the commutative algebra
$\mathbb T^S \otimes_{\ZZ} \mathcal A(p)$. Definition
\ref{def:Hecke_action} provides a structure of $\mathbb{T}$-module on
$S^\dagger(K^p)$ and $S^\dagger_\kappa(K^p)$.

\begin{defin}\label{def:finite_slope}
An overconvergent $p$-adic automorphic form $f \in S^\dagger(K^p)=(\Pi^{\la})^{\mathfrak n}$ is called a \emph{finite slope
eigenvector} for the $\mathcal{A}(p)$-action if, for any $t \in \underline T(\QQ_p)^+$, there exists $a_t \in
L^\times$ such that
\[ [t] f = a_t f.\] 
More generally
$f$ is of \emph{finite slope} for the $\mathcal{A}(p)$-action if for
all $t \in \underline T(\QQ_p)^+$, there exists a polynomial $P\in L[X]$ such that
$P(0) \neq 0$ and $P([t])f = 0$.
\end{defin}
Given a continuous character $\delta: T\rightarrow L^\times$, we write $S^\dagger(K^p)[\delta]$ for the eigenspace with respect to the $\mathcal{A}(p)$-action of eigensystem $\delta: \mathcal{A}(p)\rightarrow L$. Note, that by definition this eigensystem is automatically of finite slope and of weight $\kappa={\rm wt}(\delta)$. Moreover, the $\mathcal{A}(p)$-action on $S^\dagger(K^p)[\delta]$ uniquely extends to an action of $\ZZ[T(\QQ_p)]$. 

\begin{rema}
\label{rema:overconvergentmffunction}
An overconvergent automorphic form of tame level $K^p$ with eigenvalue $\delta:T\rightarrow L^\times$ for the Hecke-action at $p$ (i.e.~for the action of the Atkin--Lehner ring)  is thus the same as a locally analytic function
\[ f : U(\QQ)\backslash U(\mathbb A^\infty)/K^p \fleche L,\] such that
there exists a compact open subgroup $N_0\subset \underline N(\QQ_p)$
so that, for all $g \in U(\mathbb A^\infty), t \in T_0, n \in N_0$,
\[ f(gtn) = \delta(t)f(g),\]
and such that moreover, for all $t \in \underline T(\QQ_p)^+$, $[t]f =
\delta(t)f$. 
\end{rema}

\begin{defin}\label{def:classical}
The space of \emph{classical automorphic forms of tame level $K^p$} is the subspace 
$S^{\rm cl}(K^p)=(\Pi^{\rm cl})^{\mathfrak n})$ of $S^\dagger(K^p)=(\Pi^{\rm la})^{\mathfrak
  n}$ of elements which are $K_p$-finite for some (resp.~any) compact
open $K_p \subset U(\QQ_p)$. 
\end{defin}
We note that this subspace is stable under the action of $\mathbb{T}$. 

For any character $\chi^S : \mathbb T^S \fleche L$, we let 
$\Pi[\chi^S]$ (resp.~$S^\dag(K^p)[\chi^S]$,
resp.~$S^{\rm cl}(K^p)[\chi^S]$) denote the subspace of
$\chi^S$-eigenvectors for $\mathbb T^S$ in $\Pi$
(resp.~$S^\dag(K^p)$, resp.~$S^{\rm cl}(K^p)$). 
If $\delta:T\rightarrow L$ is a character of $T$ (defining a character of $\mathcal{A}(p)$) and if $\chi=\chi^S\otimes\delta$ is the corresponding character of $\mathbb{T}=\mathbb{T}^S\otimes_{\ZZ}\mathcal{A}(p)$, we write $S^\dag(K^p)[\chi]$ etc.~for the corresponding eigenspace.

Let $\mathfrak m$ be a maximal ideal in $\mathbb T^S$. We then define
\[ \Pi_{\mathfrak m}\coloneqq\Pi^\circ_{\mathfrak m}\otimes_{\mathcal{O}_L} L, \quad
  \text{where} \quad \Pi^\circ_{\mathfrak
    m}\coloneqq\varprojlim_n(\Pi^\circ/\pi_L^n\Pi^\circ)_{\mathfrak
    m}. \] As there are only finitely many maximal ideals
$\mathfrak m$ of $\mathbb T^S$ such that
$(\Pi^\circ/\pi_L\Pi^\circ)_{\mathfrak m}$ is nonzero,
the space $\Pi_{\mathfrak m}$ is a topological direct summand of $\Pi$ stable
under the actions of $U(\QQ_p)$ and $\mathbb{T}$.

Recall that if $\mathfrak{m}$ is a maximal ideal (whose residue field is assumed to equal $k_L$) such that $\Pi_{\mathfrak m}$ is non zero, then we may associate to $\mathfrak{m}$ a continuous representation $\overline \rho : \Gal_E \fleche \GL_n(k_L)$ which is conjugate autodual, and unramified away from
$S$. Such representations $\overline\rho$ are called modular (see for example
\cite[\S2.4]{BHS1}).

\subsection{Patching the completed cohomology}
\label{sec:appl-patch-repr}

We fix a maximal ideal $\mathfrak{m}\subset \mathbb{T}^S$ such that $\Pi_{\mathfrak{m}}\neq 0$ is non-zero and denote by $\overline{\rho} : \Gal_E \fleche \GL_n(k_L)$ the corresponding modular Galois representation. For each place $v$ of $F$ which splits in $E$ we write
\[ \overline\rho_v := \overline{\rho}_{|\Gal_{E_{\tilde v}}},\] for a
choice of $\widetilde v | v$ of $E$. From now on we assume that, for
$v\in S$, the place $v$ splits in $E/F$, we make a fixed choice $\widetilde{v} | v$ as before such that $\widetilde{v} \in \Sigma$ if $v | p$, and denote $\widetilde S = \{ \widetilde{v} | v \in S\}$ so that $\widetilde S$ is in bijection with $S$ and contains $\Sigma$. For $v\in S$ we write 
$R_{\overline\rho_v}^\Box$ for the universal lifting (i.e. framed
deformation) ring of $\overline\rho_v$ and  define
\[R_{\overline\rho_v}^\Box\twoheadrightarrow \overline{R}_{\overline\rho_v}^{\Box}\] to be the maximal reduced
$\ZZ_p$-flat quotient. 
\begin{rema}
If $v | p$ we
have in fact, by the main results of \cite{BIP},
$\overline{R}_{\overline\rho_v}^{\Box}= R_{\overline\rho_v}^\Box$. Using the main result of \cite{dat2024moduli} we find that the same applies to places $v\nmid p$, as the deformation rings $R_{\overline\rho_v}^\Box$ may be identified with versal rings to the moduli space of L-parameters. We still keep the notations introduced above in order to be consistent with the notations from the references for the patching construction below.
\end{rema}
We
denote by $R_{\overline\rho,\mathcal S}$ the quotient of
$R_{\overline\rho}$ corresponding to the deformation problem
\[ \mathcal S = (E/F, S,\widetilde{S},\mathcal O_L,\overline\rho,\eps^{1-n}\delta^n_{E/F},\{\overline{R}_{\overline\rho_v}^{\Box}\}_{v \in S}) \]
in the notations of \cite[§2.3]{CHT}, where $\delta_{E/F} : \Gal_F \fleche \{\pm 1\}$ is the quadratic character associated to $E/F$, and
\[ R^{\rm loc} \coloneqq \widehat{\bigotimes_{v \in S}} \overline{R}_{\overline \rho_v}^{\Box}.\]
We assume the following (strong) Taylor-Wiles hypothesis
\begin{hypothese}\label{hyp:TWtext}
\begin{enumerate}
\item $p >2$ ;
\item the extension $E/F$ is unramified and $E$ does not contain a (non-trivial) $p$-th
  root $\zeta_p$ of $1$ ;
\item the group $U$ is quasi-split at all finite places of $\QQ$ ;
\item the level $K^p$ is chosen such that $K_v$ is hyperspecial whenever the
  finite place $v$ of $F$ is inert in $E$ ;
\item the representation $\overline{\rho}_{|\Gal_{E(\zeta_p)}}$ is adequate.
\end{enumerate}
\end{hypothese}

By \cite{CEGGPS} sections 2.7,2.8, (see also \cite[Théorème 3.5]{BHS1}), we have the following data.

\begin{prop}
There exist
\begin{enumerate}
\item an integer $g \geq 1$ ;
\item a continuous, admissible, unitary $R_\infty$-representation $\Pi_\infty$ of $U(\QQ_p)$ over $L$, where \[R_\infty \coloneqq R^{\rm loc}[[x_1,\dots,x_g]];\]
\item a local map of local rings $S_\infty := \mathcal
  O_L[[y_1,\dots,y_t]] \fleche R_\infty$ with \[t = g + \dim
  R^{loc}-[F^+:\QQ]\frac{n(n+1)}{2}\] and a local map of local rings $R_\infty \fleche
  R_{\overline\rho,\mathcal S}$ such that
\begin{enumerate}[(i)]
\item there exists an $\cO_L$-lattice $\Pi_\infty^0 \subset \Pi_\infty$ stable by $U(\QQ_p)$ and $R_\infty$ such that
\[ (\Pi_\infty^0)' = \Hom_{\mathcal O_L}(\Pi_\infty^0,\mathcal O_L),\]
is a projective $S_\infty[[K_p]]$-module of finite type (via $S_\infty
\fleche R_\infty$) for some (equivalently all) compact open subgroup
$K_p\subset U(\QQ_p)$ ;
\item the map $R_\infty\rightarrow R_{\overline\rho,\mathcal S}$ induces an isomorphism
\[ R_\infty/\mathfrak aR_\infty \simeq R_{\overline\rho,\mathcal S},\]
of local noetherian $\mathcal O_L$-algebras and an isomorphism of continuous admissible unitary $R_\infty/\mathfrak aR_\infty$-representations of $U(\QQ_p)$ on $L$ 
\[ \Pi_\infty[\mathfrak a] \simeq \Pi_{\mathfrak m} ,\]
where $\mathfrak a = (y_1,\dots,y_t)$ denotes the augmentation ideal of $S_\infty$,
\end{enumerate}
\end{enumerate}
\end{prop}

It is a direct consequence of this proposition that the $R_\infty$-representation
$\Pi_\infty$ of $U(\QQ_p)$ satisfies Hypothesis
\ref{hyp:patching}. We note that the same applies to a slightly more general context:
\begin{lemma}
\label{lemma:piinftysathyp}
  Let $V$ be a finite dimensional algebraic representation of
  $U(\QQ_p)$ over $L$. Then the $R_\infty$-Banach representation
  $\Pi_\infty\otimes_LV$ satisfies Hypothesis \ref{hyp:patching}.
\end{lemma}

\begin{proof}
  As $\Pi_\infty$ satisfies Hypothesis \ref{hyp:patching}, for any
  open pro-$p$-subgroup $H$ of $U(\QQ_p)$ there exists an isomorphism
  of $\ZZ_p^t\times H$-representations $\Pi_{\infty|\ZZ_p^t\times
    H}\simeq\mathcal{C}(\ZZ_p^t\times H,L)^m$ for some $m\geq1$. But then
  \[ (\Pi_\infty\otimes_LV)_{|\ZZ_p^t\times
      H}\simeq\mathcal{C}(\ZZ_p^t\times H,V)^m\simeq
    \mathcal{C}(\ZZ_p^t\times H,L)^{m\dim_LV}. \qedhere\]
\end{proof}

In the reminder of this paper we will use the following notations:
we set
\[\mathcal X^p\coloneqq\Spf(\widehat{\bigotimes}_{v\in S\backslash
  S_p}\overline{R}_{\overline\rho_v}^\Box)^{\rig}\simeq\prod_{v\in
  S\backslash S_p}\Spf(\overline R_{\overline\rho_v}^\Box)^\rig,\] where
$\mathbb{U}^g\coloneqq\Spf(\cO_L[[x_1,\dots,x_g]])^\rig$ is an open polydisc. Moreover, we set
\begin{align*}
\mathcal X_{\overline\rho_p}&\coloneqq\Spf(\widehat{\bigotimes}_{v\in
  S_p}R_{\overline\rho_v}^\Box)^\rig,\\
  \mathcal X_\infty &\coloneqq \Spf(R_\infty)^{\rig} \simeq\mathcal X^p\times\mathcal
X_{\overline\rho_p}\times \mathbb{U}^g.
\end{align*} 
By construction the space $\mathcal{X}_\infty$ contains  $\mathcal X_{\rhobar,\mathcal S}=(\Spf R_{\rhobar,\mathcal S}
)^{\rig}$ as a closed subspace. For a point $x=(x^p,x_p,z)\in\mathcal X_\infty(L)$ and a place $v$ of $F$
dividing $p$, we denote by $\rho_{x,v}$ the framed representation
$\Gal_{F_v}\rightarrow\GL_n(L)$ associated to $x$. Finally we write $\rho_{x,p}$ for the the family of representations $(\rho_{x,v})_{v\mid p}$.

\section{Patching functors}
\label{sec:patching_functors}

In this section, we keep notations and conventions of section
\ref{sec:global}. In particular, we have
$\underline G\simeq\prod_{v\in
  S_p}(L\times_{\QQ_p}\Res_{F_v/\QQ_p}\GL_{n,F_v})$ which is an
algebraic group over $L$ and we consider the associated categories
$\mathcal O, \mathcal O_{\alg}^\infty$ and $ \widetilde{\mathcal O}_{\alg}$
as in section
\ref{sec:cat_O} (for the choice of the upper triangular Borel subgroup $\underline B$).

We fix once and for all a point $x\in\mathcal X_\infty(L)$ such that $x$ maps to the origin in $(\Spf\, S_\infty)^{\rig}$ (i.e.~the point defined by the augmentation ideal of $S_\infty$) and we denote by 
$\widehat R_{\infty,x}$ the completed local ring of
$\mathcal X_{\infty}$ at $x$.

\subsection{Locally analytic patching functors}

We fix a smooth and unramified character
$\eps : \underline T(\QQ_p)\rightarrow L^\times$ and consider $\eps$ as a point of $\widehat{T}$. 

By Lemma \ref{lemma:piinftysathyp}, we can apply Corollary \ref{coro:Jacquet_conclusion} to the admissible locally analytic representation $\Pi_\infty^\la$, and obtain a functor
\begin{align*}
\cO_\alg^{\infty}&\rightarrow {\rm Coh}(\mathcal{X}_\infty\times \widehat{T})\\
M&\mapsto \mathcal M_{\Pi_\infty}(M).
\end{align*}
\begin{defin} For $M \in \cO_\alg^{\infty}$ we define
\[\mathcal M_{\infty,x,\eps}(M)\coloneqq \mathcal M_{\Pi_\infty}(M)_{x,\eps}\] to be the stalk of 
$\mathcal M_{\Pi_\infty}(M)$ at $(x,\eps)$. 
\end{defin}
It follows from
Proposition \ref{prop:Cohen_Mac} that $\mathcal M_{\infty,x,\eps}(M)$ is a Cohen--Macaulay
$\widehat R_{\infty,x}$-module and is follows from Theorem
\ref{theo:exactness_Jacquet} that the functor
$M\mapsto\mathcal M_{\infty,x,\eps}(M)$ is exact.

\begin{rema}\label{rema:other_description}
  We also have the following description:
  \[ \mathcal M_{\infty,x,\eps}(M) \simeq \left( \Hom_{U(\mathfrak
        g)}(M,\Pi^{\la}_\infty[\mathfrak m_x^\infty])^{N_0}[\mathfrak
      m_{\eps}^\infty]\right)'\] 
   where $\mathfrak m_{\eps}$ is the
  maximal ideal of
  $\mathcal A(p)\otimes_{\ZZ_p}\QQ_p=\QQ_p[\underline T(\QQ_p)^+]$
  corresponding to the character $\eps$ and $\mathfrak m_x$ is the
  maximal ideal of $R_{\infty}[1/p]$ corresponding to $x$.
\end{rema}

\begin{rema}\label{rema:2_U(t)actions}
  Note that we have two $U(\mathfrak t)$-module structures on
  $\mathcal M_{\infty,x,\eps}(M)$: The first one comes
  from the nilpotent $U(\mathfrak t)$-module structure on $M$ as in
  section \ref{sec:nilp-acti-umathfr}. The second one comes from the
  action of $U(\mathfrak t)$ induced from the locally analytic
  $T$-structure on $\Pi_\infty^{\la}$. It is a tautological consequence of the construction, but we point out that these two actions coincide.
\end{rema}
\begin{defin}
Let $I\subset\Delta$ be a finite subset of simple roots and let $M$ be
an object of $\widetilde{\mathcal O}_{\alg}^I$. Then we define
\[ \mathcal M_{\infty,x,\eps}(M)\coloneqq\varprojlim_n\mathcal
  M_{\infty,x,\eps}(M/\frakm_I^n).\]
  \end{defin}

\begin{prop}\label{prop:deformed_functor}
 The functor
  $M\mapsto \mathcal M_{\infty,x,\eps}(M)$ is exact on
  $\widetilde{\mathcal O}_\alg^I$ and for each $M\in \widetilde{\mathcal O}_{\alg}^I$ the $\widehat R_{\infty,x}$-module $\mathcal M_{\infty,x,\eps}(M)$
  is finitely generated and Cohen--Macaulay of dimension
  $t+\dim_K\mathfrak z_I$.  Moreover
  $\mathcal M_{\infty,x,\eps}(M)$ is flat over $U(\mathfrak z_I)$.
\end{prop}

\begin{proof}
  Let $\widehat S_\infty$ be the completion of $S_\infty[1/p]$ along
  the maximal ideal generated by the augmentation ideal $\mathfrak a$
  of $S_\infty$. Moreover, we write $\widehat{U}_I$ for the completion of $U(\mathfrak{z}_I)$ at the maximal ideal $\mathfrak{m}_I$.
  
  By exactness of the functor
  $\mathcal M_{\infty,x,\eps}$, we have
  \[\mathcal M_{\infty,x,\eps}(M/\frakm_I^{n+1})/\frakm_I^n\simeq
  \mathcal M_{\infty,x,\eps}(M/\frakm_I^n)\] for any $n\geq1$. It
  follows from Theorem \ref{prop:Cohen_Mac} that $\mathcal M_{\infty,x,\eps}(M/\frakm_I)$
  is a finite projective 
  $\widehat S_\infty$-module. We denote its rank by $r\geq0$. The exactness of
  $\mathcal M_{\infty,x,\eps}$ implies that
  $\mathcal M_{\infty,x,\eps}(M/\frakm_I^n)$ is a finite projective
  $\widehat S_\infty\otimes_LU(\mathfrak z_I)/\frakm_I^n$-module of
  rank $r$ and it follows that $\mathcal M_{\infty,x,\eps}(M)$
  is a finite projective
  $\widehat S_\infty\hat\otimes_L\widehat{U}_I$-module of rank $r$. As the action of
  $\widehat S_\infty\hat\otimes_L\widehat{U}_I$ factors through $\widehat R_{\infty,x}$ we deduce the
  result. The exactness of the functor
  $\mathcal M_{\infty,x,\eps}$ is a consequence of the exactness
  of $\mathcal M_{\infty,x,\eps}$ restricted to $\cO_{\alg}^{\infty}$
  and the fact that each system
  $(\mathcal M_{\infty,x,\eps}(M/\frakm_I^n))_n$ satisfies the
  Mittag-Leffler condition.

  Let $\underline t=(t_1,\dots,t_m)$ be a regular sequence generating
  the maximal ideal of $U(\mathfrak z_I)_{\mathfrak m_I}$. This is
  also a regular sequence generating the maximal ideal of the
  completion $\widehat{U}_I$. By exactness of the functor
  $\widehat S_\infty\otimes_L-$ on strict exact sequences of Fr\'echet
  $L$-algebras, the sequence $\underline t$ is
  $\widehat S_\infty\widehat\otimes_L\widehat{U}_I$-regular. As
  $\mathcal M_{\infty,x,\eps}(M)$ is a finite free
  $\widehat S_\infty\widehat\otimes_L\widehat{U}_I$-module, the sequence $\underline t$ is
  $\mathcal M_{\infty,x,\eps}(M)$-regular. This is equivalent
  to flatness over $U(\mathfrak z_I)_{\mathfrak m_I}$.
\end{proof}

\subsection{A factorization property}
\label{sec:fact-prop}

We use the spaces and notations introduced in section \ref{sec:quasi-triang-local}. A point $x\in\mathcal X_\infty(L)$ is said to be crystalline $\varphi$-generic and Hodge--Tate regular if for all $v|p$ the representation $\rho_{x,v}$ is crystalline $\varphi$-generic and Hodge--Tate regular. Let $x=(\rho^p,\rho_p,z)\in\mathcal X_\infty(L)$ be such a
$\varphi$-generic Hodge--Tate regular point. We fix a refinement $\mathcal R$ of $\rho_p$. 

Recall that
$\underline G\simeq\prod_{v\in
  S_p}(L\times_{\QQ_p}\Res_{F_v/\QQ_p}\GL_{n,F_v})$.  If $I$ is a set
of simple roots of $\underline G$, we set
\[ \mathcal X_{\infty,x,\mathcal R}^{I-\qtri} \coloneqq
  \widehat{\mathcal X^p}_{\rho^p} \times \mathcal X_{\rho_p,\mathcal
    R}^{I-\qtri} \times \widehat{\mathbb U^g}.\] This is a closed
subscheme of $(\widehat{\mathcal X_\infty})_x$ and we write
$\widehat R_{\infty,x}\twoheadrightarrow R_{\infty,x,\mathcal R}^{I-\qtri}$ the corresponding
quotient map. Moreover, for $w\in W$, we set
\[ \mathcal X_{\infty,x,\mathcal R}^{I-\qtri,w} \coloneqq
  \widehat{\mathcal X^p}_{\rho^p} \times \mathcal X_{\rho_p,\mathcal
    R}^{I-\qtri,w} \times \widehat{\mathbb U^g}.\]

If
$\mathcal
R=(\varphi_{1,v},\dots,\varphi_{n,v})_{v|p}\in\prod_{v|p}(L^\times)^n$,
we define $\delta_{\mathcal R}$ to be the smooth unramified character of $T$
defined by
\[
  (x_{1,v},\dots,x_{n,v})_{v|p}\mapsto\prod_{v|p}\prod_i(\varphi_{i,v}^{v_{F_v}(x_{i,v})}q_v^{i-n}) \]
where $q_n$ denotes the cardinality of the residue field of
$F_v$. We use the notation
$\mathcal M_{\infty,x,\mathcal R}\coloneqq\mathcal
M_{\infty,x,\delta_{\mathcal R}}$. The goal of this section is to prove the following result.

\begin{theor}\label{theo:qtri}
  Let $x\in\mathcal X_\infty(L)$ be a $\varphi$-generic Hodge--Tate
  regular crystalline point and let $\mathcal R$ be a refinement of
  $x$. Then, for any $M \in \mathcal O_\alg^{\infty,I}$, the
  $\widehat R_{\infty,x}$-module $\mathcal M_{\infty,x,\mathcal R}(M)$
  is killed by the kernel of the map
  $\widehat R_{\infty,x}\twoheadrightarrow 
  R_{\infty,x,\mathcal R}^{I-\qtri}$. Equivalently its support is
  contained in $\mathcal X_{\infty,x,\mathcal R}^{I-\qtri}$.
\end{theor}

\begin{proof}
  This is a consequence of Proposition
  \ref{prop:comp_Jacquet_parabolic},   Proposition \ref{prop:lift} and Corollary
  \ref{coro:annihilator_I_final} which will be proved below.
\end{proof}

We will prove the auxiliary statements in (the proof of) this theorem by making use of variants of the construction of eigenvarieties.
More precisely, for a subset $I\subset \Delta$, a character $\lambda\in X^\ast(\underline T)^+_I$ (dominant with respect to $\underline{P}_I$) and an algebraic representation $V$ of $\underline G$ we will consider the scheme-theoretic supports 
\begin{align*}
\mathcal{E}_\infty^I(\lambda)={\rm supp}(\mathcal M_{\Pi_\infty}^{I,\lambda})&\subset \mathcal X_\infty\times\widehat T\\
\mathcal{E}_\infty^I(\lambda,V)={\rm supp}(\mathcal M_{\Pi_\infty}^{I,\lambda,V})&\subset \mathcal X_\infty\times\widehat T,
\end{align*}
where $\mathcal M_{\Pi_\infty}^{I,\lambda}$ respectively $\mathcal M_{\Pi_\infty}^{I,\lambda,V}$ are the coherent sheaves associated to $J_{I,\lambda}(\Pi_\infty^\la)'$ 
respectively to $J_{I,\lambda}((\Pi_\infty\otimes_LV)^\la)'$ (see section \ref{sec:comp-with-jacq} for the notation).
We will link the completions of $\mathcal{E}_\infty^I(\lambda)$ resp.~$\mathcal{E}_\infty^I(\lambda,V)$ at points $(x,\delta)\in \mathcal{X}_\infty\times \widehat{T}$ to the quasi-trianguline deformation rings of section \ref{sec:quasi-triang-local}. 
This is done in two steps: we first show that the set-theoretic support of $\mathcal M_{\Pi_\infty}^{I,\lambda}$ resp.~of $\mathcal M_{\Pi_\infty}^{I,\lambda,V}$ is contained in the (quasi-)trianguline locus (see the proof of Proposition \ref{prop:case_of_a_Verma}). 
We then prove that $\mathcal{E}_\infty^I(\lambda)$ resp.~$\mathcal{E}_\infty^I(\lambda,V)$ is reduced (see the proof of Proposition \ref{prop:reduced_tens_W}). The proof of the latter statement follows the usual argument in the case of eigenvarieties, see e.g.~\cite[Corollaire 3.12 and Corollaire 3.20]{BHS1}: the general properties of eigenvarieties (deduced from the fact that the sheaves $\mathcal M_{\Pi_\infty}^{I,\lambda}$ resp.~$\mathcal M_{\Pi_\infty}^{I,\lambda,V}$ are locally finite projective over $(\Spf S_\infty)^{\rig}\times\widehat{T}_0$ imply that $\mathcal{E}_\infty^I(\lambda)$ resp.~$\mathcal{E}_\infty^I(\lambda,V)$ have no embedded components. Hence it is enough to produce on each of their irreducible components a point $y$ such that   $\mathcal{E}_\infty^I(\lambda)$ resp.~$\mathcal{E}_\infty^I(\lambda,V)$ are reduced in a neighborhood of $y$. By the same projectivity argument as above, the point $y$ can be chosen so that the weight map to $\widehat{T}_0$ is smooth at this point. Reducedness then boils down to checking that the Hecke operators (that generate the local ring of $\mathcal{E}_\infty^I(\lambda)$ resp.~$\mathcal{E}_\infty^I(\lambda,V)$ at $y$) act semi-simply on the fiber of $\mathcal M_{\Pi_\infty}^{I,\lambda}$ resp.~$\mathcal M_{\Pi_\infty}^{I,\lambda,V}$ over $\widehat{T}_0$ which in turn follows from the fact that Hecke-operators act semi-simply on spaces of classical automorphic forms. 
We now give the details of these arguments. 

Let $\delta=(\delta_{1,v},\dots, \delta_{n,v})_{v|p} \in \hat{T}(L)$ be a parameter for a quasi-triangulation of $x$ at $p$, i.e. the trianguline filtration of the $(\varphi,\Gamma)$-module $D^\dagger_{\rm rig}(\rho_v)[1/t]$ over $\mathcal R_{K,L}[1/t]$ has graded pieces $\mathcal R_{K,L}(\delta_{i,v})[1/t]$.
As $x$ is Hodge--Tate regular, there is a natural map
\[ \omega_{\delta} : \mathcal X_{\infty,x,\mathcal R}^{\qtri} \fleche \widehat{T}^\wedge_{\delta},\]
mapping a deformation at $p$ of the  $(\varphi,\Gamma)$-module $D^\dagger_{\rm rig}(\rho_v)[1/t]$, equipped with its trianguline filtration, to its parameter (see e.g. \cite[eq (3.15)]{BHS3}). If $\delta$ is locally algebraic of the form $\delta = \lambda \delta_{\mathcal R}$ for $\lambda \in X^*(T)$ and some smooth character $\delta_{\mathcal R} \in \widehat{T}(L)$, we shift the previous map to get
\[ \omega = t_{-\lambda}\omega_{\delta} : \mathcal X_{\infty,x,\mathcal R}^{\qtri} \fleche \widehat{T}^\wedge_{\delta_{\mathcal R}}\]
which only depends on the chosen refinement. This induces a map
\[  i \times \omega : \mathcal X_{\infty,x,\mathcal R}^{\qtri} \fleche \widehat{\mathcal X_{\infty,x}}\times  \widehat{T}^\wedge_{\delta_{\mathcal R}},\]
or equivalently, a homomorphism $\widehat{R}_{\infty,x} \otimes \mathcal O^\wedge_{\hat{T},\delta_{\mathcal R}} \fleche  R_{\infty,x,\mathcal R}^{I-\qtri}$.

\begin{prop}\label{prop:case_of_a_Verma}
  Let $\lambda\in X^*(\underline T)_I^+$ be a weight dominant with respect
to $\underline P_I$. The $\widehat R_{\infty,x}$-module
  $\mathcal M_{\infty,x,\mathcal R}(\widetilde M_I(\lambda))$ is
  annihilated by the kernel of
  $\widehat R_{\infty,x}\rightarrow  R_{\infty,x,\mathcal
    R}^{I-\qtri}$. More precisely, $\mathcal M_{\infty,x,\mathcal R}(\widetilde M_I(\lambda))$ is an $\widehat{R}_{\infty,x} \otimes \mathcal O^\wedge_{\hat{T},\delta_{\mathcal R}}$-module and 
  annihilated by the kernel of \[\widehat{R}_{\infty,x} \otimes \mathcal O^\wedge_{\hat{T},\delta_{\mathcal R}} \fleche  R_{\infty,x,\mathcal R}^{I-\qtri}.\]
\end{prop}

\begin{proof}
It follows from
  Proposition \ref{prop:comp_Jacquet_parabolic} and the definition of
  $\mathcal M_{\infty,x,\mathcal R}(\widetilde M_I(\lambda))$ that
  \[\mathcal M_{\infty,x,\mathcal R}(\widetilde M_I(\lambda))=(t_\lambda^*\mathcal M_{\Pi_\infty}^{I,\lambda})^\wedge_{(x,\delta_{\mathcal R})}\]
 as an $\widehat R_{\infty,x}\otimes \mathcal O^\wedge_{\hat{T},\delta_{\mathcal R}}$-module. 
 It is thus enough to show that the
  completion of $\mathcal M_{\Pi_\infty}^{I,\lambda}$ at the point
  $(x,\lambda\delta_{\mathcal R}) \in\mathcal X_{\infty}(L)\times\widehat T(L)$
  is supported at the closed subspace \[i \times w_{\delta} : \mathcal X_{\infty,x,\mathcal R}^{I-\qtri} \fleche \widehat{\mathcal X_{\infty,x}}\times  \widehat{T}^\wedge_{\delta_\lambda\delta_{\mathcal R}}.\]
  
We closely follow the proof of \cite[Prop.~5.13]{Wu2}. Let us write $\mathcal{E}_\infty\subset \mathcal{X}\times \widehat{T}$ for the scheme-theoretic support of the coherent sheaf defined by $J_B(\Pi_\infty)'$. By \cite[5.4]{Wu2} this contains $\mathcal{E}_\infty^I(\lambda)$ as a closed subspace.
 As in the proof of \cite[Prop.~5.13]{Wu2} we consider a proper birational map $f : \mathcal E_\infty'\rightarrow\mathcal E_\infty$ 
 such that the universal $(\varphi,\Gamma)$-module over $\mathcal E_\infty'$ has a quasi-triangulation, and write $\mathcal E_\infty''$ for the preimage of
  $\mathcal E_\infty^I(\lambda)$ in $\mathcal E_\infty'$. 
  Let
  $Y\subset\mathcal E_\infty''$ be the Zariski closed reduced subspace
  of $\mathcal E_\infty''$ whose points are exactly the points of
  $\mathcal E_\infty''$ where the universal filtered
  $(\varphi,\Gamma)$-module over $\mathcal R[1/t]$ is $P_I$-de
  Rham. As in \cite{Wu2}, the existence of $Y$ is a consequence of
  \cite[Prop.~A.10]{Wu2}. It follows that for any $y\in Y$ lying above $(x,\delta_{\mathcal R})$ the map 
  \[\widehat{Y}_y\rightarrow \mathcal X_\infty \times \widehat{T}\]
  factors through
  $\mathcal X_{\infty,y,\mathcal R_y}^{I-\qtri}$.
  Let $U\subset \mathcal E_\infty^I(\lambda)$
  be an affinoid open subset containing $x$ and a Zariski dense subset
  of points which are de Rham (and in particular $P_I$-de Rham) and trianguline with parameter given by
  $\mathcal E_{\infty}^I(\lambda) \fleche \widehat{T}$. Such a neighborhood  exists by \cite[Prop.~5.11 \& 5.12]{Wu2}. We deduce that
  $f(Y)\supset U$ and  hence $f^{-1}(U)\subset Y$ and we conclude as in the
  proof of \cite[Prop.~3.7.2]{BHS3} (see the erratum in \cite{BreuilDing})
  that the map \[\widehat U_{x,\lambda\delta_{\mathcal R}}\rightarrow \mathcal X_\infty\times\widehat{T}\] factors through
  $\mathcal X_{\infty,x,\mathcal R}^{I-\qtri}$. 
\end{proof}

\begin{cor}\label{coro:annulateur_puissance}
  Let $V$ be an algebraic representation of $\underline G$, then
  \[\mathcal{M}=\mathcal M_{\infty,x,\mathcal R}(\widetilde
  M_I(\lambda)\otimes_LV)\] is annihilated by some power of the kernel
  of
 $\widehat{R}_{\infty,x} \otimes \mathcal O^\wedge_{\hat{T},\delta_{\mathcal R}} \fleche  R_{\infty,x,\mathcal R}^{I-\qtri}$.
\end{cor}

\begin{proof}
   We recall that 
  \[\widetilde M_I(\lambda)\otimes_L V=U(\mathfrak{g})\otimes_{U(\mathfrak{p}_I)}(L_I(\lambda))\otimes_LV\cong U(\mathfrak{g})\otimes_{U(\mathfrak{p}_I)}(L_I(\lambda)\otimes V_{|P_I})\] 
and that $V_{|P_I}$ is an extension of algebraic
  irreducible representations of $\underline L_I$. Exactness of
  $\mathcal M_{\infty,x,\mathcal R}$ (see Proposition
  \ref{prop:deformed_functor}) implies that the $\widehat R_{\infty,x}$-module
  $\mathcal M$ is an extension of
  $\widehat R_{\infty,x,\mathcal R}$-module of the form
  $\mathcal M_{\infty,x,\mathcal R}(\widetilde M_I(\mu))$ for
  $\mu\in X^*(\underline T)_I^+$. We deduce the result from
  Proposition \ref{prop:case_of_a_Verma}.
\end{proof}

\begin{prop}\label{prop:reduced_tens_W}
  Let $V$ be an algebraic representation of $\underline G$. Then 
  the schematic support $\mathcal{E}_\infty^I(\lambda,V)$ of the coherent sheaf associated to
  $J_{I,\lambda}((\Pi_\infty\otimes_LV)^\la)'$ is reduced.
\end{prop}

\begin{proof}
  We follow closely the proof of \cite[Cor.~3.20]{BHS1} replacing,
  where it is nedeed, some arguments by results of \cite{Wu2}. To simplify notations we just write $\mathcal E=\mathcal{E}_\infty^I(\lambda,V)$ and $\mathcal M=\mathcal M_{\Pi_\infty}^{I,\lambda,V}$ for the reminder of this proof.

  Let
  $\mathcal N$ be the radical ideal of $\cO_{\mathcal E}$. Assume that
  $\mathcal N\neq0$ and let $x\in\mathcal E$ be a point in the support
  of $\mathcal N$. Let $\widehat T^\circ_{\lambda}$ be the preimage of
  $\lambda_{|\mathfrak t \cap\mathfrak l_I^{\mathrm{ss}}}\in (\mathfrak
  t\cap\mathfrak l_I^{\mathrm{ss}})^*$ under the map
  \[\widehat T\rightarrow \mathfrak t^*\rightarrow(\mathfrak
  t\cap\mathfrak l_I^{\mathrm{ss}})^*,\]
  where the first map is the weight map (\ref{eqn: weight map}). 
  According to \cite[\S5.4]{Wu2} there exists an open affinoid neighborhood $U$ of $x$ and an open affinoid subset $W\subset \widehat T^\circ_\lambda\times\Spf(S_\infty)^\rig$ such that
  $\Gamma(U,\mathcal M)$ is a finite free $\cO(W)$-module (such a data
  exists according to the results of \cite[\S5.4]{Wu2}). Then
  $\Gamma(U,\mathcal N)$ is the radical ideal of $\cO(U)$. Moreover, as $\mathcal{O}(U)=\Gamma(U,\mathcal{O}_{\mathcal{E}})$ is a
  sub-$\cO(W)$-module of $\End(\Gamma(U,\mathcal M))$ (by the same argument as in the proof of Theorem \ref{prop:Cohen_Mac} respectively of \cite[Prop.~3.11]{BHS1}), the same is true for $\Gamma(U,\mathcal N)$. Therefore
  $\Gamma(U,\mathcal N)$ is a torsion free $\cO(W)$-module and its
  support has the same dimension as $W$ and hence contains an irreducible component $U_0$ of $U$. As a
  consequence the support of $\mathcal N$ contains an admissible open subset of $\mathcal E$. As the support of $\mathcal N$ is also a
  closed analytic subset of $\mathcal E$, it follows from
  \cite[Lemm.~2.2.3]{ConradIrred} that the support of $\mathcal N$
  contains an irreducible component of $\mathcal E$. It hence suffices to produce on each irreducible component of $\mathcal{E}$ a point $y$ such that $\mathcal{E}$ is reduced in a neighborhood of $y$.
  
  By
  \cite[Prop.~5.11]{Wu2} every irreducible component of
  $\mathcal E$ contains a point with algebraic weight.
  
Therefore we fix a point
  $x\in\mathcal E(L)$ with integral weight $\lambda' \in \widehat T^\circ_{\lambda}$. Let $U$ be an
  open affinoid neighborhood of $x$ and
  $W\subset\widehat T^\circ_\lambda\times\Spf(S_\infty)^\rig$ an open
  affinoid open subset such that $M=\Gamma(U,\mathcal M)$ is a direct
  factor of
  $\cO(W)\hat{\otimes}_LJ_{B_I}(J_{P_I}(\Pi_\infty\otimes_LV)_\lambda)'$. Let
  $A=\cO(W)$ and $B=\cO(U)$. Then $M$ is a finitely generated
  $B$-module and a finite projective $A$-module. Let $C>0$ and $C'>0$
  as in the proof of \cite[Prop.~5.11]{Wu2}. We set $Z\subset W$ be
  the subset of algebraic character $\delta_{\lambda'}$ such that, for
  any simple root $\alpha\notin I$, $\scalar{\lambda'+\nu,\alpha}>C'$
  for any $\nu$ weight of $V^\vee$. This is a Zariski dense subset of
  $W$. Then for $z=\delta_{\lambda'}\delta_{\sm}$ with $\delta_{\sm}$ a smooth character, using Proposition
  \ref{prop:comp_Jacquet_parabolic}, we see that the $B$-module
  $M_z=M\otimes k(z)$ is a direct factor of
  $J_B(\Hom(M_I(\lambda'),\Pi_\infty\otimes_LV))'$. Let
  $(x,\delta)\in U$ be a point above $z$,
  i.e.~$\delta=\delta_{\lambda'}\delta_\sm$, then arguing as in
  \emph{loc.~cit.}, we have
  $\Hom_G(\mathcal{F}_{\overline{B}}^G(N\otimes_LV^\vee,\delta_\sm\delta_B^{-1}),\Pi_\infty[\frakp_x])=0$
  for any subquotient $N$ of $M_I(\lambda')$ different from
  $L(\lambda')$. This implies that $M_z$ is actually a quotient of
  $J_B(\Hom_{U(\frakg)}(L(\lambda')\otimes_LV^\vee,\Pi_\infty))$ which
  is isomorphic to a finite direct sum of
  $J_B(\Hom_{U(\frakg)}(L(\mu),\Pi_\infty))$ with $\mu$ dominant. The proof of
  \cite[Cor.~3.20]{BHS1} shows that the global sections of the
  coherent sheaf associated to each
  $J_B(\Hom_{U(\frakg)}(L(\mu),\Pi_\infty))'$ on
  $U\cap\kappa^{-1}(\set{\delta_{\lambda'}})$ is a semisimple
  $B$-module. This concludes the proof.
\end{proof}

\begin{cor}\label{cor:reduced}
  The rigid analytic space $\mathcal E_\infty^I(\lambda)$ is reduced.
\end{cor}

\begin{proof}
  This is Proposition \ref{prop:reduced_tens_W} with $V$ the trivial
  representation.
\end{proof}

\begin{cor}\label{coro:annihilator_I_final}
Let $V$ be an irreducible algebraic representation of $\underline{G}$.
Then the $\widehat R_{\infty,x} \otimes \mathcal O_{\widehat{T}^\wedge_{\delta_{\mathcal R}}}$-module
  $\mathcal M_{\infty,x,\mathcal R}(\widetilde
  M_I(\lambda)\otimes_LV)$ is killed by the kernel of the map 
  \[\widehat R_{\infty,x}\otimes \mathcal O_{\widehat{T}^\wedge_{\delta_{\mathcal R}}}\rightarrow R_{\infty,x,\mathcal
    R}^{I-\qtri}.\]
\end{cor}

\begin{proof}
  By Proposition \ref{prop:reduced_tens_W}, the support of the module
  $\mathcal M_{\infty,x,\mathcal R}(\widetilde
  M_I(\lambda)\otimes_LV)$ is reduced for any
  $\lambda\in X^*(\underline T)$ dominant with respect to
  $\underline P_I$ and any algebraic representation $V$ of
  $\underline G$. Therefore the result follows from Corollary
  \ref{coro:annulateur_puissance}.
\end{proof}

\subsection{Bi-module structure on the patched functor}
\label{sec:two-acti-umathfr}

Let $M$ be an object of $\cO_\alg^\infty$ or $\widetilde \cO_\alg^I$. As
seen in section \ref{sec:nilp-acti-umathfr}, there is a natural
structure of $A=U(\mathfrak t)_{\frakm}$-module on $M$ which provides, by functoriality,
the structure of an $A$-module on $\mathcal M_{\infty,x,\mathcal R}(M)$.
This $A$-module structure extends to an action of the completion $\widehat A$ of $A$ with respect to the maximal ideal
$\frakm$. We recall from Remark
\ref{rema:2_U(t)actions} that this action coincides with the structure of an
$\widehat A$-module on $\mathcal M_{\infty,x,\mathcal R}(M)$ induced
from the $T$-action on $\Pi_\infty$. 

On the other hand, the ring
$R_{\infty,x,\mathcal R}^{\qtri}$ also carries a structure of
an $\widehat A$-module induced from the map $\kappa_1$ defined in section
\ref{sec:local-models}. This gives a further structure of an $\hat A$-module on the $R_{\infty,x,\mathcal R}^{\qtri}$-module $\mathcal M_{\infty,x,\mathcal R}(M)$. We will show that these $\hat A$-module structures agree.

For $a\in\widehat A$, we denote by $a$
(resp.~$\widetilde a$) the endomorphism of
$\mathcal M_{\infty,x,\mathcal R}(M)$ defined by the first
(resp.~second) action. Note that if $M$ is an object of
$\widetilde\cO_\alg$, then $\mathcal M_{\infty,x,\mathcal R}(M)$ is a
finite free $\widehat A\widehat\otimes_L\widehat S_\infty$-module for
the first $\hat A$-module structure by the proof of Proposition
\ref{prop:deformed_functor}. Thus it is $A$-torsion free (since
$\widehat A$ is domain).

\begin{lemma}\label{lemm:first_coincides}
  For any $a\in\widehat A$ and any $M$ in $\cO_\alg^{\infty}$ or
  $\widetilde \cO_\alg$, there is an equality \[a=\widetilde a\in{\rm End}(\mathcal M_{\infty,x,\mathcal R}(M)).\]
\end{lemma}

\begin{proof}
  If $M=\widetilde M(\mu)\otimes_{U(\mathfrak t)}U(\mathfrak t)/m^n$
  for some $\mu\in X^*(\underline T)$, this is a consequence of
  \cite[Thm.~3.21]{BHS1}, the commutative diagram \cite[(3.30)]{BHS3}
  and Remark \ref{rema:2_U(t)actions}. This implies that for any
  $\mu \in X^*(\underline T)$, we have $a=\widetilde a$ on
  $\mathcal M_{\infty,x,\mathcal R}(\widetilde M(\mu))$.

  Now we consider the general case. It follows from Proposition
  \ref{prop:lift} that it is sufficient to prove the equality
  $\widetilde a=a$ when $M=\widetilde M(\mu)\otimes_LV$ for
  $\mu\in X^*(\underline T)$ dominant and $V$ a finite dimensional
  $U(\frakg)$-module. Let $(\Fil_i)$ be an increasing filtration of
  $\widetilde M(\mu)\otimes_LV$ such that
  $\Fil_i/\Fil_{i-1}\simeq\widetilde M(\mu_i)$ where
  $\mu_1,\dots,\mu_d\in X^*(\underline T)$ and $d=\dim_LV$ (such a
  filtration exists by \cite[Lem.~8]{SoergelHC}. Let $K$ denote the
  fraction field of $A$. It follows from Proposition
  \ref{prop:filt_split} that we have a decomposition of
  $U(\frakg)_K$-modules
  \[ (\widetilde M(\mu)\otimes_LV)\otimes_AK
    \simeq\bigoplus_{i=1}^d\widetilde M(\mu_i)_K \] splitting the
  filtration $(\Fil_i\otimes_AK)$. Let
  $p_i\in\End_{U(\frakg)_K}((\widetilde M(\mu)\otimes_LV)\otimes_AK)$
  be the projector on $\widetilde M(\mu_i)_K$. As
  \[\End_{U(\frakg)_K}((\widetilde
  M(\mu)\otimes_LV)\otimes_AK)\simeq \End_{U(\frakg)}((\widetilde
  M(\mu)\otimes_LV))\otimes_AK\] by \cite[Thm.~5]{SoergelHC}, there
  exists, for each $1\leq i\leq d$, a nonzero element $q_i\in A$ such
  that $q_ip_i$ actually restricts to an endomorphism of $\widetilde
  M(\mu)\otimes_LV$. We set $q=q_1\cdots q_r$ and 
  $\alpha_i=qp_i$. Then the $\alpha_i$ are endomorphisms of $\widetilde
  M(\mu)\otimes_LV$ that stabilize the filtration $\Fil_\bullet$.  As each $\Fil_i/\Fil_{i-1}$ is a free
  $A$-module, the endomorphisms $\alpha_i$ induce the zero endomorphism of $\Fil_{i-1}$ and
  $\widetilde M(\mu)\otimes_LV/\Fil_i$ and 
  the multiplication by $q$ on $\Fil_i/\Fil_{i-1}$.

  In order to simplify notations we set
  \begin{align*}
  M_\infty&=\mathcal M_{\infty,x,\mathcal R}(\widetilde
  M(\mu)\otimes_LV),\\ 
  \Fil_iM_\infty&=\mathcal M_{\infty,x,\mathcal R}(\Fil_i).
  \end{align*} By construction, for each $i$ the endomorphism 
  $\alpha_i$ induces an $R_{\infty,x}$-linear endomorphisms of
  $\Fil_jM_\infty$ for all $j$. By exactness of
  $\mathcal M_{\infty,x,\mathcal R}$, the family $(\Fil_iM_\infty)$ is a
  filtration of $M_\infty$ and
  $\Fil_iM_\infty/\Fil_{i-1}M_\infty\simeq\mathcal
  M_{\infty,x,\mathcal R}(\widetilde M(\mu_i))$ for any $i$, so that
  $a$ and $\widetilde a$ induces the same endomorphism of
  $\Fil_iM_\infty/\Fil_{i-1}M_\infty$. Finally, for  $1\leq i\leq d$, we denote by  $M_\infty^{(i)}=\alpha_i(\Fil_iM_\infty)$ the image of the $i$-th filtration step under $\alpha_i$. It
  follows from the properties of $\alpha_i$ that
  \begin{itemize}
  \item $M_\infty^{(i)}\subset\Fil_iM_\infty$;
  \item the quotient $\Fil_iM_\infty/(\Fil_{i-1}M_\infty+M_\infty^{(i)})$ is killed by $q$;
  \item $M_\infty^{(i)}$ is isomorphic to a quotient of
    $\Fil_iM_\infty/\Fil_{i-1}M_\infty$.
  \end{itemize}
  Therefore, we have $\widetilde a=a$ on $M_\infty^{(i)}$ for any
  $a\in\widehat A$ and the quotient of $M_\infty$ by the sum of the
  $M_\infty^{(i)}$ is killed by $q^d$. As $M_\infty$ is $A$-torsion free it follows that $\widetilde a=a$. 
\end{proof}

 Let $\xi : Z(\mathfrak g) \rightarrow U(\mathfrak t)$ be the
Harish-Chandra map as recalled in section \ref{sec:bimodule}. As in loc.cit.~we write $t_{\nu}$ for the unique endomorphism
of $U(\mathfrak t)$ mapping $x\in\mathfrak t$ to
$t_{\nu}(x)=x+\nu(x)$.  

Let $h=(h_{1,\tau,v}<\cdots<h_{n,\tau,v})_{\tau,v}\in
X^*(\underline T)$ be the weight corresponding to the Hodge--Tate
weights of $\rho_x=(\rho_v)_{v|p}$ and let
$\delta'_G=(0,-1,-2,\dots,1-n)_{\tau,v}\in X^*(\underline T)$ be fixed
central shift of the half sum of the positive roots
$\delta_G\in X^\ast(\underline T)\otimes\QQ$.  We have a map
\[\kappa_2 : \widehat A=\widehat{U(\mathfrak t)}_{\frakm}\rightarrow
R_{\rho_p,\mathcal R}^{\qtri}\] induced from the map $\kappa_2$ of
section \ref{sec:local-models} and we define the $L$-algebra homomorphism
\[\alpha=\kappa_2\circ t_{h-\delta'_G}\circ\xi:Z(\mathfrak g)\rightarrow R_{\rho_p,\mathcal R}^{\qtri}.\] As in \cite[Def.~4.23]{DPS}, we define, for any $v|p$, an $L$-algebra homomorphism
\[ \zeta_{\widetilde\rho_v^C}^C : Z({\rm
    Lie}(\Res_{F_v/\QQ_p}\GL_n))\longrightarrow
  R_{\overline\rho_v}^{\Box,\rig} \] where $\widetilde\rho_v$ is the
universal family of Galois representations over
$R_{\overline\rho_v}^{\Box,\rig}$. After completion at $\rho_v$ and
taking the tensor product over all $v|p$, we obtain an $L$-algebra
homomorphism
\[ \zeta^C : Z(\mathfrak{g})=\bigotimes_{v|p}Z({\rm
    Lie}(\Res_{F_v/\QQ_p}\GL_n))\longrightarrow
  R_{\rho_p}^\Box\twoheadrightarrow R_{\rho_p,\mathcal R}^{\qtri}. \]
Note that the definition of $\zeta_{\widetilde\rho_v^C}^C$ from
$\rho_v$ depends on a choice of a central shift of $\delta_G$ (see the
discussion ending \cite[\S4.7]{DPS}). We choose it equal to
$\delta'_G$. More concretely $\zeta^C$ is characterized by the
following property. This is the unique continuous homomorphism such
that, for any local artinian $L$-algebra and any local homomorphism
$f : R_{\rho_p,\mathcal R}^{\qtri}\rightarrow A$, corresponding to
$\rho_A=(\rho_{A,v} : \Gal_{F_v}\rightarrow\GL_n(A))_{v|p}$, the
composition map
$Z(\frakg)\xrightarrow{\zeta^C}R_{\rho_p,\mathcal R}^{\qtri}\rightarrow A$ is
$Z(\frakg)\xrightarrow{\xi}U(\mathfrak
t)\xrightarrow{t_{\nu-\delta'_G}}k(x)$ where
\[ \nu\in\Hom_L(U(\mathfrak t)^W,A)\simeq\Hom_L(U(\mathfrak
  t^*)^W,A)\simeq \Hom_L(U(\frakg^*)^{\underline G_L},A) \] is the map
induced by the conjugacy class of the Sen operators
\[(\Theta_{\mathrm{Sen},\rho_{A,v}})_{v|p}\in(\frakg\otimes_{L}A).\]

\begin{prop}\label{prop:compatibility_HC_center}
  The homomorphisms $\zeta^C$ and $\alpha$ defined above coincide.
\end{prop}

\begin{proof}
  It is sufficient to prove that for any local artinian $L$-algebra
  $A$ and any map
  $f : R_{\rho_p,\mathcal R}^{\qtri}\rightarrow A$, we have
  $f\circ\zeta^C=f\circ\alpha$. Note that the map $f$ gives rise to a
  family $(\rho_{A,v})_{v|p}$ of local Galois representations. It follows
  from \cite[Lem.~3.7.5]{BHS3} that, for any embedding
  $\tau : F_v\hookrightarrow L$, the $\tau$-part of the Sen polynomial
  of $\rho_v$ is $\prod_{i=1}^n(X-(h_{i,\tau}+\nu_{i,\tau}))$ where
  $(\nu_{i,\tau})\in\Hom_L(\mathfrak t,A)$ corresponds to
  $f\circ\kappa_2 : U(\mathfrak t)\rightarrow A$. The result is then a
  direct comparison of the definitions of $\alpha$ and
  $\zeta^C$.\qedhere
  \end{proof}

For each element $M$ of the category $\cO_\alg^\infty$ or
$\widetilde \cO_\alg$, there is a natural homomorphism of $L$-algebras
$Z(\mathfrak g) \rightarrow\End(M)$. By functoriality of
$\mathcal M_{\infty,x,\mathcal R}$, this gives a map
\[z : Z(\mathfrak g) \rightarrow\End_{\widehat R_{\infty,x}}(\mathcal
M_{\infty,x,\mathcal R}(M)).\] The following result tells us that this
map factors through $R_{\infty,x,\mathcal R}^{I-\qtri}$.
\begin{cor}\label{cor:compatibility_center}
  For any $x\in Z(\mathfrak g)$, the element $z(x)$ is the multiplication
  by
  $\alpha(x)\otimes 1\in  R_{\infty,x,\mathcal R}^{I-\qtri}$.
\end{cor}

\begin{proof}
  This is a consequence of Proposition
  \ref{prop:compatibility_HC_center} and of \cite[Thm.~9.27]{DPS}.
\end{proof}

\begin{rema}\label{rem:comparaison_HC}
Recall that $h=(h_{1,\tau,v}<\cdots<h_{n,\tau,v})_{\tau,v}$ denotes the weight corresponding to the Hodge--Tate weights of $\rho$. Let $\lambda\coloneqq w_0(h)-\delta'_G\in X^*(\underline T)$, which
  is still a  dominant character. Recall that
  $t_{-\delta_G}\circ\xi$ has image contained in $U(\mathfrak t)^W$. Hence we have
  \[ t_{h-\delta'_G}\circ\xi=t_h\circ\Ad(w_0)\circ
    t_{-\delta'_G}\circ\xi=\Ad(w_0)\circ t_{w_0(h)}\circ
    t_{-\delta'_G}\circ\xi=\Ad(w_0)\circ t_{\lambda}\circ\xi. \] Therefore \[\Id\otimes\alpha=(\Id\otimes\Ad(w_0))\circ h_\lambda:A\otimes_LZ(\mathfrak g)\rightarrow A\otimes_{A^W}A,\] where
  $h_\lambda$ is the map defined in section \ref{sec:bimodule}.
  \end{rema}

\subsection{Computation of a support}
\label{sec:support}

Now we can prove our main result of this section concerning the support of the patched
functor applied to a generalized Verma module respectively applied to its dual.

\begin{theor}\label{thm:component_Verma}
  Let $x\in\mathcal X_\infty(L)$ be a  point whose associated Galois representation is crystalline, $\varphi$-generic and Hodge--Tate
  regular. Let $\mathcal R$ be a refinement of
  $x$. Let
  $h=(h_{1,\tau}<\cdots <h_{n,\tau})_{\tau : F\hookrightarrow
    L}\in X^*(\underline T)$ be the character given by the Hodge--Tate
  weights of $\rho_x$. Let $
    \delta'_G=\det{}^{\frac{1-n}{2}}\delta_G=(0,-1,\dots,1-n)_{\tau
      : F\hookrightarrow L}\in X^*(\underline T)$, where $\delta_G$ is the half sum of the positive roots, and define
  $\lambda\coloneqq w_0(h)-\delta'_G \in X^*(\underline T)^+$.
  
  Then, for
  $I\subset\Delta$ and $w\in W$, the schematic supports of
  $\mathcal M_{\infty,x,\mathcal R}(\widetilde
  M_I(w^{\rm min}\cdot\lambda))$ and
  $\mathcal M_{\infty,x,\mathcal R}(\widetilde
  M_I(w^{\rm min}\cdot\lambda)^\vee)$ are either
  $\mathcal X_{\infty,x,\mathcal R}^{I-\qtri,w^{\rm min}w_0}$ or empty.
\end{theor}

\begin{proof} Let $M$ be $\widetilde M_I(w^{\rm min}\cdot\lambda)$ or
  $\widetilde M_I(w^{\rm min}\cdot\lambda)^\vee$. As
  $R_{\infty,x,\mathcal R}^{I-\qtri}$ is generically reduced and
  equi-dimensional by Lemma \ref{lemm:generically_reduced} and as
  $\mathcal M_{\infty,x,\mathcal R}(M)$ is  Cohen--Macaulay of
  dimension $\dim R_{\infty,x,\mathcal R}^{I-\qtri}$, its schematic
  support is reduced and is a union of irreducible
  components of $\Spec R_{\infty,x,\mathcal R}^{I-\qtri}$, i.e.~it is
  a union of $\Spec R_{\infty,x,\mathcal R}^{I-\qtri,w'}$ for some
  $w'\in W$.

  By Proposition \ref{prop:deformed_Verma_annihilator}, the module $M$ is annihilated by 
  $I_w\subset A_I\otimes_LZ(\mathfrak g)$. This implies in particular
  that the action of $A_I\otimes_LZ(\mathfrak g)$ on $M$ factors through
  $h_\lambda$. By functoriality, this gives rise to a structure of an
  $A_I\otimes_{A^W}A$-module on $\mathcal M_{\infty,x,\mathcal
    R}(M)$. Note that the map $(\kappa_1,\kappa_2)$ of section
  \ref{sec:local-models} provides a morphism of $L$-algebras
  $A_I\otimes_{A^W}A\rightarrow R_{\infty,x,\mathcal R}^{I-\qtri}$
  and, using Theorem \ref{theo:qtri}, a second structure of an
  $A_I\otimes_{A^W}A$-module on $\mathcal M_{\infty,x,\mathcal
    R}(M)$. It follows from Lemma \ref{lemm:first_coincides},
  Corollary \ref{cor:compatibility_center} and Remark
  \ref{rem:comparaison_HC} that this two actions coincide up to
  composition with $\Id\otimes\Ad(w_0)$. We deduce that
  $\mathcal M_{\infty,x,\mathcal R}(M)$ is killed by the ideal of
  $R_{\infty,x,\mathcal R}^{I-\qtri}$ defining the inverse
  image of
  $T_{I,ww_0}\subset\mathfrak z_I\times_{\mathfrak t/W}\mathfrak
  t$. Therefore Lemma \ref{lemm:comp_z_ixt} (see also Remark
  \ref{rema:t_and_tdual}) implies that the action of
  $R_{\infty,x,\mathcal R}^{I-\qtri}$ factors through
  $R_{\infty,x,\mathcal R}^{I-\qtri,ww_0}$ so that the schematic
  support of $\mathcal M_{\infty,x,\mathcal R}(M)$ is
  $\Spec R_{\infty,x,\mathcal R}^{I-\qtri,ww_0}$.
\end{proof}

\section{Main results}
\label{sec:localcomp}

Let $x =(\rho_p,\rho^p,z) \in\mathcal X_{\infty}(L)$ corresponding to a classical
automorphic form of tame level $K^p$. Moreover, we assume that (the Galois representation defined by) $x$ is
crystalline, Hodge--Tate regular and $\varphi$-generic (see section
\ref{sec:fact-prop}) at $p$. This means that
$x\in\mathcal X_{\rhobar,\mathcal S}(L)\subset\mathcal X_\infty(L)$
and that there exists an automorphic representation
$\pi=\pi_{\infty}\otimes_{\CC}\pi_f$ of $U(\mathbb A_{\QQ})$ whose associated Galois representation $\rho$ is the representation defined by $x$ and such that
$\pi_f \otimes W$ occurs in the locally algebraic vectors of
$\Pi_{\frakm}$ for some algebraic representation $W$ depending on $\rho$. 
In particular, the automorphic
representation $\pi$ is finite slope at $p$. It follows from the proof
of \cite[Cor.~3.12]{BHS2} that the image $\rho^p$ of $x$ in
$\Spf(\bigotimes_{v \in S, p \nmid v}
\overline{R}^{\Box}_{\overline\rho_v})^{\rig}$ lies in the smooth
locus.

We fix a refinement 
$\mathcal R = (\varphi_{1,v},\dots,\varphi_{n,v})_v$ of
$x$. Let us denote the $\tau$-Hodge--Tate weights of $\rho_{x,v}$ for $v | p$ in $F$ and $\tau :
F_v\hookrightarrow L$ by $h_{v,\tau}\coloneqq (h_{1,v,\tau} <
\dots < h_{1,v,\tau})$. 
Given this collection of Hodge--Tate weights we write $h=(h_{v,\tau})_{v,\tau}$ and $h_v=(h_{v,\tau})_{\tau}$.
We then define  $R_{\rho_v}^{\cris, h_v}$ to be the crystalline deformation ring of $\rho_v$ of labelled Hodge--Tate weight $h_v$ and set \[R_{\rho_p}^{\cris, h}=\widehat{\bigotimes}_{v|p}R_{\rho_v}^{\cris, h_v}.\] We further define
\[\mathcal X_{\infty,x,\mathcal R}^{\cris,h} = \widehat{\mathcal X^p}_{\rho^p} \times (\Spf R_{\rho_v}^{\cris, h_v}) \times \widehat{\mathbb U^g}.\] Note that is follows from the definitions that $\mathcal X_{\infty,x,\mathcal R}^{\cris,h} $ embeds into  $\mathcal X_{\infty,x,\mathcal R}^{\qtri,w_0}$ for any choice of a refinement $\mathcal R$.

We set \[\mu_{v,\tau} =
(h_{1,v,\tau},h_{2,v,\tau}+1,\dots,h_{n,v,\tau}+(n-1)) = h_{v,\tau} - \delta'_{G,v,\tau},\] and $\mu
= (\mu_{v,\tau})_{v,\tau}$, which is thus antidominant (for the upper
Borel), and $\lambda = w_0(h) - \delta_G' = w_0 \cdot \mu \in X^*(\underline T)^+$. For all $v | p$ in $F$, we denote by $W_v$ the Weyl group of ${\rm GL}_n(F_v)$, which we identify with $\mathfrak S_n$ and denote by 
$s_{1,v},\dots,s_{n-1,v}$ the simple reflections with respect to the choice of the upper Borel $B_v \subset \GL_{n,F_v}$. Moreover, 
$w_{0,v} = s_{n-1,v}\dots s_{2,v}s_{1,v} s_{2,v}\dots s_{n-1,v}$ will denote the longest element of $W_v$. We then write $W = \prod_v W_v$ the Weyl group of 
$\underline{G}_{\QQ_p} \simeq \prod_{v | p} \GL_{n,F_v}$ with respect to the Borel $B = \prod_{v | p} B_v$. Because of the product structure,
 we will sometimes abuse notations and simply write $s_i$ for the simple reflections and $w_0$ for the longest element.
 
 For a scheme $X$ of dimension $d$ we write $Z^0(X)=Z_d(X)$ for the free abelian group on the irreducible components of $X$. Moreover, for $d'\leq d$ we write $Z_{d'}(X)$ for the free abelian group on the irreducible and reduced closed subschemes of dimension $d'$. We recall that a coherent sheaf $\mathcal F$ on $X$ with $d'$-dimensional support defines a class $[\mathcal{F}]\in Z_{d'}(X)$, see e.g.~\cite[Equation (2.13)]{BHS3}.

\subsection{Sheaves and supports.}
\label{subsec:sheavesandsupports}

Let $\lambda=w_0\cdot\mu \in X^*(\underline T)^+$ dominant, integral. 

We moreover write 
\begin{equation}\label{eqn: def mx}
m_x= \dim \mathcal
M_{\infty,x,\mathcal R}(L(\lambda)) \otimes k(x).
\end{equation} It follows from
\cite[Thm.~5.1.3]{BHS3} that $m_x\geq1$ and the proof of
\cite[Thm.~5.3.3]{BHS3} implies that $m_x$ does not depend on the choice of a refinement  $\mathcal{R}$.
To $x$ and $\mathcal R$ we associate a permutation \[w_{x,\mathcal R} = (w_{x,\mathcal R_v})_{v\in \Sigma} = (w_{x,\mathcal R_v,\tau})_{v,\tau} \in W\] defined as in \cite[§ 3.7]{HMS}. We recall that these permutations encode the relative position of the Hodge--Tate flags with respect to the full flag corresponding to the refinement $\mathcal R$.
We recall that, for any object $M$ of $\mathcal O^{\infty}_\alg$ or
$\widetilde{\mathcal O}_\alg$, the sheaf
$\mathcal M_{\infty,x,\mathcal R}(M)$ is zero or Cohen--Macaulay of
dimension $d$.

\begin{lemma}\label{lemm:reduction_cycles}
  Let $R$ be a Cohen--Macaulay noetherian local ring of dimension $d'$
  and let $M$ and $M'$ be two finitely generated Cohen--Macaulay
  modules. Let $(t_1,\dots,t_m)$ be a regular sequence of elements of
  the maximal ideal of $R$ which is also $M$ and $M'$-regular. Assume
  that $[M]=[M']$ in $Z_{d'}(\Spec R)$. Then
  \[[M/(t_1,\dots,t_m)M]=[M'/(t_1,\dots,t_m)M']\in Z_{d'-m}(R).\]
\end{lemma}

\begin{proof}
  By induction it is sufficient to prove the result when $m=1$. Set
  $t=t_1$. Let $\frakp$ be a prime ideal of $R$ which is a generic
  point of $\Supp(M)$ or $\Supp(M')$. It is sufficient to prove that
  $[M_{\frakp}/tM_{\frakp}]=[M'_\frakp/t M'_{\frakp}]$ in
  $Z_{d'-1}(\Spec R_{\frakp}/(t))$, i.e.~that
  $M_{\frakp}/tM_{\frakp}$ and $M'_\frakp /t M'_{\frakp}$ are two
  $R_{\frakp}/(t)$-modules of the same length. This is a consequence
  of \cite[\href{https://stacks.math.columbia.edu/tag/02QG}{Lemma
    02QG}]{stacks-project}.
\end{proof}

Let $\mathcal N\subset\frakg$ be the nilpotent cone and let
$\widetilde{\mathcal N}\rightarrow\mathcal N$ be the Springer
resolution. Similarly to the definition of the closed subschemes $X_w\subset X$ in \ref{sec:local-models} we define
\[Z_w\subset \widetilde{\mathcal N}\times_{\mathcal N}\widetilde{\mathcal
  N}\subset X\] to be the Zariski closure of preimage under $\widetilde{\mathcal N}\times_{\mathcal N}\widetilde{\mathcal
  N}\rightarrow \underline G_L/\underline B\times\underline
G_L/\underline B$ of the orbit
$\underline G_L(1,w)\subset\underline G_L/\underline B\times\underline
G_L/\underline B$. 
Set
\[ \mathcal Z_w=g(f^{-1}(Z_w\cap \widehat X_{I,w,x_{\dR}}))\times
  \widehat{\mathcal X}^p_{\rho^p}\times\widehat{\mathbb U}^g
  \subset\overline{\mathcal X}_{\infty,x,\mathcal R}^{\qtri},\]
  where $f$ and $g$ are the maps from Theorem \ref{thm:formsmoothdiag}. 

In the following we will make use of the following abusive notation for (local) formal schemes: Let $\Spf R$ be a (local) affine formal scheme. Then we will say that $\Spf R$ is reduced, if $R$ is reduced. Moreover, we will say that $\Spf R$ is irreducible if ${\rm Spec}\,R$ is irreducible. More generally, for a given irreducible component ${\rm Spec}\, R/\mathfrak{a}\subset {\rm Spec}\, R$, we will refer to the formal subscheme ${\rm Spf}\,R/\mathfrak{a}\subset {\rm Spf}\,R$ as an irreducible component of ${\rm Spf}\, R$. Similarly, we will write $Z^0({\Spf}\,R)=Z^0({\Spec}\, R)$ for the free abelian group on the irreducible components of $\Spf\,R$ to which we also refer as the irreducible components of $\Spf\, R$, etc.

\begin{prop}\label{prop:supports_faisceaux} 
  Let $w \in W$. Then the following properties hold:
  \begin{enumerate}[1)]
  \item \label{prop:supports_faisceaux0} For all 
    $I \subset \Delta$ and all $\overline{w} \in W_I\backslash W$ satisfying $w^{\min}w_0 \geq w_{x,\mathcal R}$ , the formal subscheme
    $\mathcal X^{I-\qtri,ww_0}_{\infty,x,\mathcal
      R}$ is reduced and irreducible and
    coincides with an irreducible component of
    $\mathcal X^{I-\qtri}_{\infty,x,\mathcal R}$.  \item\label{prop:supports_faisceaux1}  The schematic support of
    $\mathcal M_{\infty,x,{\mathcal R}}(M(w  \cdot \lambda))$, for 
    $w \in W$, is contained in
    $\overline{\mathcal X}_{\infty,x,\mathcal R}^{\qtri,ww_0}$ 
    if $ww_0 \geq w_{x,\mathcal R}$, and this sheaf is zero
    otherwise. Moreover,  \[[\mathcal M_{\infty,x,{\mathcal
        R}}(M(w\cdot \lambda))] = m_x[\overline{\mathcal
      X}_{\infty,x,\mathcal R}^{\qtri,ww_0}]\in Z^0(\overline{\mathcal X}^{I-\qtri}_{\infty,x,\mathcal R})\] for $ww_0 \geq
    w_{x,\mathcal R}$, where $m_x$ is the integer defined by (\ref{eqn: def mx}).
  \item\label{prop:supports_faisceaux1bis} There is an equality
    \[ [\mathcal M_{\infty,x,\mathcal
        R}(L(ww_0\cdot\lambda))]=m_x\sum_{w'\leq w}a_{w,w'}[\mathcal
      Z_w] \in Z^0(\overline{\mathcal X}^{I-\qtri}_{\infty,x,\mathcal R})\]
    where the $a_{w,w'}\in\NN$ are the integers defined in
    \cite[Thm.~2.4.7]{BHS3}. In particular $a_{w,w}=1$.
   \item\label{prop:supports_faisceaux2} For all $I \subset \Delta$, the sheaves \[\mathcal M_{\infty,x,{\mathcal R}}(M_I(w^{\rm min} \cdot \lambda))\ \text{and}\ \mathcal M_{\infty,x,{\mathcal R}}(M_I(w^{\rm min} \cdot \lambda)^\vee)\] are non zero if and only if $w^{\rm min}w_0 \geq w_{x,\mathcal R}$.
\item\label{prop:supports_faisceaux3} For all $I\subset\Delta$, the support of \[\mathcal M_{\infty,x,{\mathcal
    R}}(\widetilde{M}_I(w^{\rm min}  \cdot \lambda))\ \text{and}\ \mathcal M_{\infty,x,{\mathcal
    R}}(\widetilde{M}_I(w^{\rm min}  \cdot \lambda)^\vee),\] for $\overline w \in W_I\backslash W$, is ${\mathcal
X}_{\infty,x,\mathcal R}^{I-\qtri,w^{\rm min}w_0}$ if $w^{\rm min}w_0 \geq w_{x,\mathcal R}$ and these sheaves are zero otherwise.
\item\label{prop:supports_faisceaux4} The module $\mathcal
  M_{\infty,x,\mathcal R}(L(\lambda))$ is free of rank $m_x$
  over 
  $\mathcal X_{\infty,x,\mathcal R}^{\cris,h} \subset
  \mathcal X_{\infty,x,\mathcal R}^{\qtri,w_0}$.
\item\label{prop:supports_faisceaux5} For any $I\subset\Delta$ and any
  $w\in W$, the sheaves
  \[\mathcal M_{\infty,x,\mathcal R}(\widetilde{M}_I(w^{\rm min} \cdot
  \lambda))\ \text{and}\ \mathcal M_{\infty,x,\mathcal R}(\widetilde{M}_I(w^{\rm min} \cdot
  \lambda)^\vee)\] are generically free of rank $m_x$ over their
  support.
\end{enumerate}
\end{prop}

\begin{proof}
We first prove point \ref{prop:supports_faisceaux0}).  As $\mathcal X^p$ is smooth at $\rho^p$ (as recalled in the begining of this section), the formal completion $\widehat{\mathcal X^p}_{\rho^p}$ is formally smooth. As $\widehat{\mathbb U^g}$ is also formally smooth, the claim follows from the fact that \[\mathcal X_{r_v,\mathcal R_v}^{I_v-\qtri,\Box} \fleche \mathcal X_{r_v,\mathcal R_v}^{I_v-\qtri}\ \text{and}\ \mathcal X_{r_v,\mathcal R_v}^{I_v-\qtri,\Box}\fleche \widehat{X_{I,x_\pdR}}\] are formally smooth and that $X_{I,w,x_\pdR}$ is an irreducible component of $\widehat{X_{I,x_\pdR}}$.

By Theorem
\ref{thm:component_Verma}, the schematic support of the Cohen-Macaulay
sheaves
\[\mathcal M_{\infty,x,\mathcal R}(\widetilde{M}_I(w\cdot \lambda))\ \text{and}\ \mathcal M_{\infty,x,\mathcal R}(\widetilde{M}_I(w\cdot \lambda)^\vee)\]
is contained in $\mathcal X_{\infty,x,\mathcal R}^{I-\qtri,w}$ which
is irreducible. By Proposition \ref{prop:deformed_functor}, as the
sheaves are Cohen--Macaulay of dimension
$t + \dim \mathfrak z_I = \dim \mathcal X^{I-\qtri}_{\infty,x,\mathcal
  R}$ (e.g.~\cite[equation (5.8)]{BHS3} and Proposition
\ref{prop:comp_Jacquet_parabolic}), we
deduce that, if non empty, their schematic support is all
$\mathcal X_{\infty,x,\mathcal R}^{I-\qtri,w}$.

By Remark
  \ref{rema:2_U(t)actions} we deduce also that   \[{\rm supp}\big(\mathcal M_{\infty,x,\mathcal R}(M_I(w\cdot\lambda))\big)\subset \overline{\mathcal X}_{\infty,x,\mathcal R}^{I-\qtri,w^{\rm min}w_0}\] for $w\in {}^IW$.
  Note that the Jordan--Hölder factors of
  $M_I(w\cdot\lambda)$ are among the the $L(w'\cdot\lambda)$ with
  $w'\geq w$ and that $L(w\cdot\lambda)$ is the cosocle of
  $M_I(w\cdot\lambda)$. Therefore
  $\mathcal M_{\infty,x,\mathcal R}(\widetilde
  M_I(w\cdot\lambda))\neq0$ if and only if
  $\mathcal M_{\infty,x,\mathcal R}(M_I(w\cdot\lambda))\neq0$ if and
  only if
  $\mathcal M_{\infty,x,\mathcal R}(
  L(w\cdot\lambda))\neq0$. Therefore the non nullity assertions in
  \ref{prop:supports_faisceaux2} and \ref{prop:supports_faisceaux3}
  follow from the exactness of $\mathcal M_{\infty,x,\mathcal R}$
  (Proposition \ref{prop:deformed_functor}) and from
  \cite[Thm.~5.3.3]{BHS3}. This proves \ref{prop:supports_faisceaux2}
  and \ref{prop:supports_faisceaux3}

We prove point \ref{prop:supports_faisceaux4}. By \cite[Remark 4.3.1
and Proof of Theorem 5.3.3, Step 7]{BHS3}, the schematic support of
$\mathcal M_{\infty,x,\mathcal R}(L(\lambda))$ is contained in the crystalline locus $\mathcal X_{\infty,x,\mathcal R}^{\cris,h} \subset \mathcal X_{\infty,x,\mathcal R}^{\qtri}$, which is smooth and irreducible of the same dimension as the support of $\mathcal M_{\infty,x,\mathcal R}(L(\lambda))$. Thus these coincide and $\mathcal M_{\infty,x,\mathcal R}(L(\lambda))$ is free of rank $m_x$ over the crystalline locus.

No we prove point \ref{prop:supports_faisceaux1}. The first assertion
has already been proved with \ref{prop:supports_faisceaux2} and
\ref{prop:supports_faisceaux3}, therefore it remains to prove the
assertion on the cycle. It follows from the
proof \cite[Thm.~5.3.3]{BHS3} that $\mathcal M_{\infty,x,\mathcal
  R}(M(w\cdot\lambda))$ is generically free of rank $m_x$ for
$ww_0\geq w_{x,\mathcal R}$. As ${\mathcal
  X_{\infty,x,\mathcal R}^{\qtri,ww_0}}$ is Cohen--Macaulay, the
result is a consequence of point \ref{prop:supports_faisceaux3} and of
Lemma \ref{lemm:reduction_cycles} applied with \[M=\mathcal O_{\mathcal
  X_{\infty,x,\mathcal R}^{\qtri,ww_0}}^{m_x}\ \text{and}\ M'=\mathcal M_{\infty,x,\mathcal
  R}(M(w\cdot\lambda))\] and to a regular sequence generating the maximal
ideal of $U(\mathfrak t)_{\mathfrak m}$. This sequence is $M'$-regular
by Proposition \ref{prop:deformed_functor}.

We deduce \ref{prop:supports_faisceaux1bis} from
\ref{prop:supports_faisceaux1} together with formulas (5.23) and
(5.24) of \cite{BHS3} and the fact that the Verma modules form a basis
of the Grothendieck group of the category $\cO_{\chi_\lambda}$.
      
We prove point \ref{prop:supports_faisceaux5}. As
$\mathcal X^{I-\qtri,w'}_{\infty,x,\mathcal R}$ is generically smooth
for any $w'$, the module $\mathcal M_{\infty,x,\mathcal R}(M)$ is generically
free, say of rank $r$, over its support where \[M\in \{\mathcal M_{\infty,x,\mathcal R}(\widetilde{M}_I(w^{\rm min}\cdot \lambda)), \mathcal M_{\infty,x,\mathcal R}(\widetilde{M}_I(w^{\rm min}\cdot \lambda)^\vee)\}.\] 
Now we claim that there exists an open an subset $U$
in the regular locus of
$\Spec(R_{\infty,x,\mathcal R}^{I-\qtri,w^{\rm min}w_0})$ such that $U$ intersects the support of
$\mathcal M_{\infty,x,\mathcal R}(L(w^{\rm min}\cdot\lambda))$. The claim then implies
$r=m_x$. Indeed, the restriction of $\mathcal M_{\infty,x,\mathcal
  R}(\widetilde M_I(w^{\rm min}\cdot\lambda))$ to $U$ is locally free since
$U$ is regular. Therefore $\mathcal M_{\infty,x,\mathcal
  R}(M_I(w^{\rm min}\cdot\lambda))$ is locally free of rank $r$ over its
support intersected with $U$. It follows from the point
\ref{prop:supports_faisceaux1bis} that $\mathcal M_{\infty,x,\mathcal
  R}(L(w'\cdot\lambda))$ is not supported at the generic point of
$\mathcal Z_{w^{\rm min}w_0}$ for $w'>w^{\rm min}$ and that $\mathcal M_{\infty,x,\mathcal
  R}(L(w^{\rm min}\cdot\lambda))$ has length $m_x$ at the generic point of
$\mathcal Z_{w^{\rm min}w_0}$. As $L(w^{\rm min}\cdot\lambda)$ appears with multiplicity one in
$M_I(w^{\rm min}\cdot\lambda)$ and all other subquotient are of the form
$L(w'\cdot\lambda)$ with $w'>w^{\rm min}$, we have $r=m_x$. We now construct an open subset $U$ with the claimed properties. We set
 \[ U=g(f^{-1}(V_{w^{\rm min}w_0}\cap \widehat X_{I,w^{\rm min}w_0,x_{\dR}}))\times
   \widehat{\mathcal X}^p_{\rho^p}\times\widehat{\mathbb U}^g, \]
 where $f$ and $g$ are the maps of Theorem \ref{thm:formsmoothdiag}
 and $V_{w^{\rm min}w_0}$ is the preimage of the Schubert cell $\underline
 G_L(1,w^{\rm min}w_0) \subset \underline{G}_L/\underline{B} \times
 \underline{G}_L/\underline{B}$ in $X_{I,w^{\rm min}w_0}$. This is an open and
 smooth subset
 of $X_{I,w^{\rm min}w_0}$: indeed, the maps $f$ and $g$ are formally smooth, the formal scheme $\mathcal X_{\infty,x,\mathcal R}^{\qtri} \fleche \mathcal X_{\rho_p,\mathcal R}^{\qtri}$ is formally smooth and the point as $\rho^p$ lies in the smooth locus of $\mathcal X^p$.
\end{proof}

\begin{prop}\label{prop:freeness_criterion}
  Assume that $x_{\pdR}$ is a smooth point of $X_{w^{\rm min}w_0}$. Then
  \[\mathcal{M}_{\infty,x,\mathcal R}(M(w^{\rm min}\cdot\lambda))\ \text{and}\ \mathcal{M}_{\infty,x,\mathcal R}(M(w^{\rm min}\cdot\lambda)^\vee)\] are finite free
  $\mathcal O_{\overline{\mathcal X}_{\infty,x,\mathcal
      R}^{\qtri,w^{min}w_0}}$-modules.
\end{prop}

\begin{proof}
We write $w^{\rm min}=w$ to simplify the notations. By Remark \ref{rema:2_U(t)actions}, the two $U(\mathfrak t)$-module structures on $\mathcal{M}_{\infty,x,\mathcal R}(\widetilde{M}(w\cdot\lambda))$ coming from the $U(\mathfrak{t})$-action on $\widetilde{M}(w\cdot \lambda)$ and the one coming from the derivative of the locally analytic action, coincide. Thus we have the equality between  $\mathcal{M}_{\infty,x,\mathcal R}(M(w\cdot\lambda))$ and the localisation \[ \mathcal{M}_{\infty,x,\mathcal R}(M(w\cdot\lambda)) \simeq i_*i^*\mathcal{M}_{\infty,x,\mathcal R}(\widetilde{M}(w\cdot\lambda)),\] where $i : \widehat T^{\rm sm} \fleche \widehat T$ denotes the inclusion of the closed subspace of smooth characters. 
A similar remark applies to the dual Verma module.  In particular, it is enough to show that the 
$\mathcal O_{{\mathcal X}_{\infty,x,\mathcal R}^{\qtri,ww_0}}$-modules \[\mathcal{M}_{\infty,x,\mathcal R}(\widetilde{M}(w\cdot\lambda))\ \text{and}\ \mathcal{M}_{\infty,x,\mathcal R}(\widetilde{M}^\vee(w\cdot\lambda))\] are finite free. But these modules are Cohen-Macaulay with support the localization at $x$ of $\mathcal X_{\infty,x,\mathcal R}^{{\rm qtri},ww_0}$, which is smooth.
\end{proof}

\subsection{Recollection on Bezrukavnikov's functor}\label{sec: Bez}

The aim if this section (or even of the paper) is to identify the patching functor that takes objects in $\mathcal{O}_{\rm alg}$ (or more generally in $\mathcal{O}_{\rm alg}^\infty$) to Cohen-Macaulay modules on certain Galois deformation rings with a functor constructed by Bezrukavnikov in geometric representation theory (more precisely: with the pullback from our local models to the Galois deformation rings). Before doing so, we will need to recall the result of Bezrukavnikov.

 Recall that $X = \widetilde{\mathfrak g}\times_{\mathfrak g} \widetilde{\mathfrak g}$ where $\mathfrak g$ is the Lie algebra of $\underline
  G_L=\prod_{v\in\Sigma}(L\times_{\QQ_p}\Res_{F_v/\QQ_p}\GL_n)$ as in
  section \ref{sec:local-models} and denote by $X^\wedge$ the completion
  of $X$ along the preimage of $\set{(0,0)}\in\mathfrak
  t\times_{\mathfrak t/W}\mathfrak t$ in $X$. Moreover, we write $\overline X = X
  \times_{\mathfrak t} \{0\}$, where the fiber product is taken with respect to the map $\kappa_1 : X \fleche \mathfrak t$ of \ref{sec:local-models} that maps $
  (g\underline B,h\underline B,N)$ to $\ad(g^{-1})(N)
  \pmod{\mathfrak n} \in \mathfrak t$.
  As in the preceding sections we fix the shift \[\delta'_{\underline{G}}={\rm det}^{\tfrac{1-n}{2}} \delta_{\underline{G}}\in X^\ast(\underline{T})\] of the half sum of the positive roots $\delta_{\underline{G}}$.

\begin{theor}[Bezrukavnikov]\label{thm:Bez}
Let  $\lambda\in X^*(\underline T)$ be a dominant character. There exists an exact functor
\[ \mathcal{B} : {\mathcal O_{\chi_\lambda}} \fleche \Coh^{\underline G_L}(X^\wedge),\]
such that 
\begin{enumerate}[1)]
\item for all $M \in \mathcal O_{\chi_\lambda}$ the sheaf $\mathcal{B}(M)$ is a Cohen-Macaulay sheaf,
\item for all $w \in W$ there is an isomorphism $\mathcal{B}(M(ww_0\cdot \lambda)^\vee) \simeq \mathcal O_{\overline{X_w}}$,
\item for all $w \in W$ there is an isomorphism $\mathcal{B}(M(ww_0\cdot \lambda)) \simeq \omega_{\overline{X_w}}$,
\item the image $\mathcal{B}(P(w_0 \cdot \lambda))$ of the anti-dominant projective $P(w_0\cdot\lambda)$ is the structure sheaf $\mathcal O_{\overline X}$,
\item the image 
  $\mathcal B(L(\lambda))$ of the algebraic representation $L(\lambda)$ is the line bundle $\mathcal O(-\delta'_{\underline{G}})\boxtimes \mathcal{O}(-\delta'_{\underline G})$ on $\underline G_L/\underline B \times \underline
  G_L/\underline B$ which is viewed as a closed subscheme of $X^\wedge$ via \[(g\underline B,h\underline B)\mapsto(g\underline B,h\underline B,0).\]
\end{enumerate}
\end{theor}

This result is (a small part of a result) due to Bezrukavnikov and his collaborators whose proof is spread out through the papers \cite{BezTwo,BR,BezL,BRTop}). For the convenience of the reader, we explain how to get the result in the previous form.

\begin{proof} By the main result of \cite{BezTwo}, there are reverse equivalence of categories
\[\Psi : D_{I^0,I^0} \leftrightarrow D^b(\Coh(\widetilde{\mathfrak g} \times_{\mathfrak g} \widetilde{\mathfrak g})) : \Phi_{I^0,I^0},\]
which we can then localize on $X^\wedge \subset X$. Up to use translation functors, we can focus on the case $\lambda = 0$. By \cite[Corollary 42 ]{BezTwo} the functor $\Psi$ in fact takes values in  ($\underline G$-equivariant) coherent sheaves on $X$, when restricted to perverse sheaves $F \in \Perv_{\underline N}(\underline{G}/\underline{B})$. 
Moreover, the Beillinson--Bernstein localization theorem, more precisely by \cite{BGWall} Localization Theorem 2.2, provides an exact fully faithfull embedding of categories
\[ \mathcal O_{\chi_0} \fleche \Perv_{\underline{N}}(\underline{G}/\underline{N}).\]
Composing the  Beillinson--Bernstein equivalence with Bezrukavnikov's functor (noting that the blocks $\mathcal O_{\chi_0}$ and $\mathcal O_{\chi_\lambda}$ are equivalent) we get the exact functor $\mathcal B$. 

Denote $\mu = w_0\cdot \lambda$ denote the antidominant weight in the dot-orbit of $\lambda$. Now the proof of \cite[Proposition 5.8]{BezL} implies that 
$\mathcal B({M}(s\cdot \mu)^\vee) = \mathcal O_{\overline{X_s}}$ for all simple reflection $s$ and $\mathcal{B}(P(\mu) )= \mathcal O_{\overline{X}}$. Bezrukavnikov's main result \cite[Theorem 1]{BezTwo} implies that $\Psi$ (hence $\mathcal{B}$) intertwines the convolutions on both sides.
Here the convolution on the category $\mathcal O_{\chi_\lambda}\simeq \mathcal O_{\chi_0}$ is inherited from the convolution in $ \Perv_{\underline N}(\underline{G}/\underline{B})$ defined as in \cite[7.]{BRTop}. We write $w=s_1\dots s_r$ and compute convolutions on both sides. By \cite[Theorem 2.2.1]{BR} we have 
\[ \mathcal O_{\overline{X_w}} = \mathcal O_{\overline{X_{s_1}}}\star \dots \star \mathcal O_{\overline{X_{s_r}}}.\]
By  \cite[Lemma 7.7]{BRTop} we have ${M}(w\cdot \mu)^\vee={M}(s_1\cdot \mu)^\vee \star \dots \star {M}(s_r \cdot \mu)^\vee$  and hence $\mathcal{B}({M}(w\cdot \mu)^\vee)=\mathcal O_{\overline{X_w}}$. Moreover, by \cite[Theorem 2.2.1]{BR} again, the dualizing sheaf of $\overline{X_w}$ is given by the convolution
\[ \omega_{\overline{X_w}}=\omega_{\overline{X_{s_1}}}\star \dots \star \omega_{\overline{X_{s_r}}}.\]
But \cite[Proposition 1.10.3]{BR} implies that the inverse of $\mathcal O_{\overline{X_s}}$ for the convolution is $\omega_{\overline{X_s}}$, and as $\mathcal{B}$ is compatible with convolution, and as the inverse of ${M}(s\cdot \mu)^\vee$ is ${M}(s \cdot \mu)$ (again using \cite[Lemma 7.7]{BRTop}for example), we deduce $\omega_{\overline{X_s}} = \mathcal{B}(M(s\cdot\mu))$. 
Finally  5. is a consequence of \cite[Lemma 6.7]{BezL} (with $P = \underline
G$).
\end{proof}
Recall that we have fixed a point $x\in\mathcal X_\infty$ associated which we have defined the positive integer $m_x$ in (\ref{eqn: def mx}).
\begin{cor}\label{cor:bezandM}
The functor $\mathcal{B}$ induces an exact functor \[\mathcal{B}_x : \mathcal O_{\chi_\lambda} \fleche \Coh(\mathcal X_{\infty,x,\mathcal R}^{\rm qtri})\] such that, for all $M \in \mathcal O_{\chi_\lambda}$ the sheaf $\mathcal{B}_x(M)$ is a Cohen-Macaulay sheaf and such that 
\[ [\mathcal M_{\infty,x,\mathcal R}(M)] = m_x [\mathcal{B}_x(M)]\in Z^0(\overline{\mathcal X}^{I-\qtri}_{\infty,x,\mathcal R}).\]
\end{cor}

\begin{proof}
Let $\underline G_1$ be the completion of $\underline G$ at the unit
element. As the representations $(\rho_v)_{v|p}$ defined by the point $x$ are crystalline and hence  de Rham we may choose a basis $\alpha$ of  $W(x) = \prod_{v \in \Sigma} W_{\rm dR}(D_{\rm rig}(\rho_{x,v})[1/t])$ and define  a point $x_{\rm pdR}$ associated to $x$ (or rather to the representations $(\rho_v)_{v|p}$) as in (\ref{eqn: def xpdr}). For all $M\in \mathcal O_{\chi_\lambda}$, the sheaf $\mathcal{B}(M)$ is a $\underline G_L$-equivariant sheaf
on $X^\wedge$ and hence gives rise to a $\underline G_1$-equivariant sheaf on
$\widehat X_{x_{\pdR}}$. Now by \cite[Theorem 3.4.4. and Corollary 3.5.8]{BHS3}, see also Theorem \ref{thm:formsmoothdiag} above, we have a diagram
\[
  \begin{tikzcd}
  & \mathcal X_{\infty,x,\mathcal R}^{\qtri,\Box}\ar[ld,"\pi"']   \ar[rd,"W"] &
 \\
  \mathcal X_{\infty,x,\mathcal R}^{\qtri} & &X^\wedge_{x_{\pdR}}.
\end{tikzcd} \]
More precisely, the map $\pi$ forgets the deformation of the fixed basis $\alpha$, and hence it is a $\underline G_1$-torsor. Moreover, $W$ formally smooth and $\underline G_1$-equivariant for the natural left actions $g \cdot \widetilde{\alpha} := \widetilde{\alpha} \circ g^{-1}$ on the source (acting only on the deformation of the isomorphisms $\alpha_v : L \otimes_{\QQ_p} F_v \overset{\sim}{\fleche} W_v$) and $g \cdot (kB,hB,N) = (gkB,ghB,g^{-1}Ng)$ on the target of $W$.

It follows that the pullback of $\mathcal{B}(M)^\wedge_{x_{\pdR}}$ at $\widehat{X}_{x_{\pdR}}$ along $W$ is a $\underline G_1$-equivariant sheaf and hence descends to a coherent sheaf  \[\mathcal{B}_x(M)\in {\rm Coh}(\mathcal X_{\infty,x,\mathcal R}^{\qtri}).\] 
It follows from the construction that $M\mapsto \mathcal{B}_x(M)$ and that $\mathcal{B}_x(M)$ is Cohen-Macaulay, as $\mathcal{B}(M)$ is. Moreover, $\mathcal{B}_x$ is exact, as $W$ is formally smooth and hence flat. 

It  remains to check the assertion on cycles. But as taking cycles is additive and $\mathcal B_x$ is exact, we only need to check this equality on a generating set of the Grothendieck group of $\mathcal O_{\chi_\lambda}$, such as the Verma modules $M(w \cdot \mu)$. Hence the desired equality follows from the previous result on Bezrukavnikov's functor together with Proposition \ref{prop:supports_faisceaux}.
\end{proof}

\subsection{A detail study of local models when $n=3$}
\label{subsect:detailn=3}
From now on we assume $n=3$, so that the group $\underline G_L$ is \[\underline G_L \simeq (\Res_{F\otimes_\QQ \QQ_p/\QQ_p}\GL_3)\times_{\QQ_p} L \simeq \prod_{v \in S_p}(L \times_{\QQ_p} \Res_{F_v/\QQ_p}\GL_{3,F_v}) \simeq \prod_{\tau \in \Sigma_F} \GL_{3,L}.\] We identify the previous local Weyl group $W$ with $\prod_\tau W_\tau$ and each $W_\tau$ with $W_{\GL_3} \simeq \mathfrak S_3$ and denote $s_{1,\tau},s_{2,\tau}$ the two simple reflection corresponding to the choice of the upper Borel, and $w_{0,\tau} = s_{1,\tau}s_{2,\tau}s_{1,\tau}$ the longuest element in $W_\tau$. If $\tau$ is understood, we often omit it from the notation.

As in section \ref{sec:local-models} we denote by $X$ the Steinberg variety for the group 
\[ \underline G = \Res_{F\otimes_\QQ \QQ_p/\QQ_p}\GL_3,\]
over $L$. As $L$ is assumed to contain all Galois conjugates of $F$ we have $X \simeq \prod_{\tau \in \Sigma_F} X_3$ (see Remark \ref{defi:Xn} for the notation $X_3$).
The Steinberg variety $X$ (resp.~$X_3$) has dimension $9^{\vabs{\Sigma_F}}$ (resp.~$9$) and $6^{\vabs{\Sigma_F}}$ (resp.~$6$) irreducible components $X_w, w \in W$ (resp. $X_{3,w}, w \in \mathfrak S_3$), see e.g.~\cite[Proposition 2.2.5]{BHS3}.

\begin{prop}\label{prop:geomXn=3}  For $w = (w_\tau)_{\tau \in \Sigma_F}$, let $ s = \vabs{\set{ \tau\in\Sigma_F \mid w_\tau = w_0}}$.
Then the component $X_w$ is smooth if and only if $s = 0$. Moreover, if $s \neq 0$, then the component $X_{w}$ is Cohen--Macaulay but not Gorenstein. More precisel, let 
\[ x_{\pdR}=(g\underline B,h \underline B,N) =(g_\tau \underline
  B_\tau,N_\tau,h_\tau\underline B_\tau)\in X_w(L)=\prod_{\tau\in\Sigma_F}X_{3, w_\tau}(L),\] 
and assume that $N_\tau=0$ when $w_\tau = w_0$. Then 
\[ \dim_L \omega_{X_{w}} \otimes k(x_{\pdR}) = 2^r,\]
where $r  \coloneqq\vabs{ \set{ \tau \mid w_\tau = w_0, \text{ and } g_\tau
      \underline B_\tau = h_\tau \underline B_\tau}}$.
\end{prop}

\begin{proof}   
The smoothness is a consequence of Proposition \ref{prop:smoothpartialsteinbergcomp}. 
As $X = \prod_{\tau \in \Sigma_F} X_3$, it is enough to prove the analogous result for $X_3$ only. Indeed, by base change and composition of upper shriek functors, the dualizing sheaf of $X$ is a derived tensor product $\bigotimes^{\mathbb L}_\tau p_\tau^*\omega_{X_3}$,
where $p_\tau :  X \fleche X_3$ is projection to the $\tau$-component. 
But as the product $X=\prod_\tau X_3$ is a product over a field, we find
\[ \omega_X = \bigotimes_\tau p_\tau^*\omega_{X_3}.\] Thus from now on we denote $X_3$ simply by $X$. 

It is thus enough to prove that the fiber of $\omega_{X_{w_0}}$, is
2-dimensional at a point of the form $(g\underline B,0,g\underline
B)$. Let $q : \widetilde{\mathfrak g} \fleche \mathfrak g$ denote the Grothendieck resolution, then $X \simeq
\underline G_L \times^{\underline B} q^{-1}(\mathfrak b)$. Moreover, $Y \coloneqq q^{-1}(\mathfrak b)$ decomposes into irreducible components $Y=\bigcup_{w\in W} Y_w$ such that $X_w\simeq \underline G_L \times^{\underline B} Y_w$. 
Hence it is enough to prove that $\omega_{Y_{w_0}}$ has fiber dimension $2$ at the point $y_{\rm pdR}=(\underline{B},0)$.
As $X_{w_0}$ is Cohen-Macaulay and flat over $\mathfrak t$ (cf \cite[Proposition 2.2.3]{BHS3}), we have the base change formula $\omega_{X_{w_0}} \otimes_{X} \overline X \simeq \omega_{\overline{X_{w_0}}}$. We are thus reduced to compute the dualizing sheaf $\omega_{\overline{Y_{w_0}}}$ of the irreducible component \[\overline{Y_{w_0}}=Y_{w_0}\times_{\mathfrak t}\{0\}\] of $\overline Y = q^{-1}(\mathfrak n)$. 
This scheme now has dimension $3$ and we can use explicit computations. 

A point of $\overline{Y}(L)$ is of the form 
$(g\underline B,N) \in (\underline G/\underline B \times \mathfrak
g)(L)$. We use the embedding $\underline G/\underline B \hookrightarrow
\mathbb P_L^2 \times (\mathbb P_L^2)^\vee$ that sents a full flag $(0
\subset \mathcal L \subset \mathcal P \subset k^3)$ to $(\mathcal L
\subset k^3, \mathcal P \subset k^3)$. In homogeneous coordinates $([x_0:x_1:x_2],[y_0:y_1:y_2])$ the condition $\mathcal L
\subset \mathcal P$ is given by $x_0y_0 + x_1y_1+x_2y_2 = 0$. 
Let $\overline{Y}^0 \subset \overline{Y}$ denote the open subset defined by
the condition $x_0 = y_2 = 1$. It is enough to compute on this open subset, as this is a neighborhood of the point $y_{\rm pdR}=(\underline{B},0)=([1:0:0],[0:0:1])$. 
On $\overline{Y}^0$ we can thus remove $y_0$ from our equations. Let us write
\[
N = \left(
\begin{array}{ccc}
 0 &  u_{12} &  u_{13} \\
  & 0  &  u_{23} \\
  &   & 0  
\end{array}
\right)
\]
for the universal matrix over $\overline{Y}^0$. The ideal defining \[\overline{Y}^0_{w_0}\subset Z \coloneqq
\Spec(k[x_1,x_2,y_1,u_{12},u_{23},u_{13}])\] is then given by
\[ I_{w_0} = (u_{23}x_2,u_{12}(x_2+x_1y_1),u_{12}x_1+u_{13}x_2,u_{23}y_1 - u_{13}(x_2+x_1y_1)).\]
We remark that we can replace $u_{12}(x_2+x_1y_1)$ by $u_{12}x_2 - x_{13}x_2y_1$ using the third equation, and that automatically $y_1u_{12}u_{23} = 0$ using our new equation and $u_{23}y_1 - u_{13}(x_2+x_1y_1)=0$. We then check (e.g.~using Macaulay2) that
\[0 \fleche \mathcal O_Z^2 \overset{A'}{\fleche} \mathcal O_Z^6 \overset{A}{\fleche} \mathcal O_Z^5 \overset{A''}{\fleche} \mathcal O_Z\]
 is a resolution of $\mathcal O_Z/I_{w_0}$, where 
\[A'=\left(
\begin{array}{cc}
 y_1 & y_1u_{13}-u_{12}   \\
  -x_2&0   \\
x_1   & u_{23} \\
0 & -u_{12}u_{23} \\
0 & - x_2u_{23} \\
0 & x_1u_{12}+x_2u_{12}  
\end{array}
\right)\ ,\ A''=\left(
\begin{array}{c}
x_1u_{12} + x_2u_{13}  \\
  x_2u_{23} \\
   y_1u_{12}u_{23} \\
   x_1y_1u_{13} - y_1u_{23} + x_2u_{13} \\
   x_2y_1u_{13} - x_2u_{12}
\end{array}
\right)^t\]

\[A = \begin{pmatrix}
      -x_2u_{23}& -y_1u_{23}&0&x_2&-y_1u_{13}&0\\
      x_1u_{12}+x_2u_{13}&y_1u_{13}&-y_1u_{12}&-y_1&0&-y_1u_{13}+u_{12}\\
      0&x_1&x_2&0&1&0\\
      0&0&0&{-x_2}&u_{12}&0\\
      0&0&0&x_1&u_{13}&u_{23}\end{pmatrix}.\]
     Let  $i : \overline{Y}^0_{w_0} \hookrightarrow Z$ denote the canonical closed embedding. Then the dualizing sheaf can be computed as  $\omega_{\overline{Y}^0_{w_0}} = i^*{\rm Ext}^3_{\mathcal O_Z}(\mathcal O_{\overline{Y^0_{w_0}}},\mathcal O_Z)$ which is given by
      \[ \omega_{\overline{Y}^0_{w_0}} \simeq \mathcal O_Z^2 /<(y_1,y_1u_{13}-u_{12}),(x_2,0),(x_1,u_{12}),(0,u_{12}u_{23})>,\]
      as $x_2u_{23} = x_2u_{12}+x_2u_{12} = 0$ on
      $\overline{Y}^0_{w_0}$. It follows that the fiber of $\omega_{\overline{Y}^0_{w_0}}$ at $y_\pdR$ is $2$-dimensional.
\end{proof}

\begin{lemma}\label{lemm:X_3_J} Let $J\subset\Delta_{\GL_3}$.
\begin{enumerate}
\item\label{lemm:X_3_J1} For $w \in W(\GL_3) \simeq \mathfrak S_3$ the component $X_{3,w}$ is smooth if $w \neq w_0$.
\item\label{lemm:X_3_J2} If $x_{\pdR} = (g\underline B_3,h\underline B_3,0) \in X_{3,w_0}(L)$, with $g\underline B_3 \neq h\underline B_3$, then $x_{\pdR}$ is a smooth point of $X_{3,w_0}$.
\item\label{lemm:X_3_J3} For $\emptyset \neq J \subset \Delta_{\GL_3}=\{s_1,s_2\}$ the component $X_{3,J,\overline w}$ is smooth for any $\overline w \in W_J\backslash W_{\GL_3}$.
\end{enumerate}
\end{lemma}

\begin{proof}
  Point \ref{lemm:X_3_J1} is Proposition \ref{prop:geomXn=3}. For the point \ref{lemm:X_3_J2}, denote $w'$ the index of the Schubert stratum in which $x_{pdR}$ lies.
  By \cite[Proposition 2.5.3(ii)]{BHS3} it is thus enough 
  (as $\overline{U_{w_0}} = \GL_3/\underline B_3\times \GL_3/\underline B_3$ is smooth) to prove that 
  $\codim_{\mathfrak t}(\mathfrak t^{w_0w{'-1}}) = \lg(w_0)-\lg(w')$. But this codimension is what we have denoted 
  $\ell(w_0w^{'-1})$ in the proof of Proposition \ref{prop:smoothpartialsteinbergcomp}. As $w' \neq 1$ and $n=3$, 
  $w_0w^{'-1}$ is a product of distinct simple reflections thus $\ell(w_0w^{'-1}) = \lg(w_0w^{'-1}) = \lg(w_0) - \lg(w')$.
  For point \ref{lemm:X_3_J3}, as $n=3$ we have that $J = \{s_1\},\{s_2\}$ or $J = \{s_1,s_2\}$. 
  Denote $\underline P = \underline P_J$. In the case $J = \{s_1,s_2\}$, then $\underline P_J = \GL_3$ and $X_{3,J} = \widetilde{\mathfrak g}$ is smooth.  It is sufficient to prove the case of $J = \{s_1\}$ (the other case is exactly the same), where an explicite computation gives the smoothness (alternatively, when $w^{min}$ has length $\leq 1$, \cite[Corollary 5.3.4]{BreuilDing} also implies smoothness).  
\end{proof}

\begin{cor} \label{cor:freeness}Let $w = (w_\tau)_{\tau} \in W$ and let $I = \coprod_\tau I_\tau\subset \Delta$. Let
$x_{\pdR} = (x_{\pdR,\tau})_\tau = (g_\tau \underline B_\tau,h_\tau \underline B_\tau,N_\tau)$ be a point such that  $N_\tau = 0$ whenever  $I_\tau = \emptyset, w_\tau = 1$.
If \[\mathcal M_{\infty,x,\mathcal R}(M_I(w^{\rm min}\cdot \lambda)) \quad
\text{(resp. }\mathcal M_{\infty,x,\mathcal R}(M_I(w^{\rm min}\cdot \lambda)^\vee)),\]
is not a finite free over
${\overline{\mathcal X}_{\infty,x,\mathcal R}^{I-\qtri,w^{\rm min}w_0}}$-module, then there exists an embedding $\tau$ such 
that $I_\tau = \emptyset$, $w_\tau = 1$ and $w_{x,\mathcal R,\tau}=1$.
  \end{cor}

\begin{proof}
Assume that there is no $\tau$ such that  $I_\tau = \emptyset$ and $w_\tau = w_{x,\mathcal R,\tau} = 1$. Lemma \ref{lemm:X_3_J}, then shows that the local model $X_{I}$ is smooth at $x_\pdR$. By \ref{prop:supports_faisceaux} the support 
\[\mathcal X_{\infty,x,\mathcal R}^{I-\qtri,ww_0}={\rm supp}\ \mathcal M_{\infty,x,\mathcal R}(\widetilde{M}_I(w \cdot \lambda))\]
is smooth. Thus  $\mathcal M_{\infty,x,\mathcal R}(\widetilde{M}_I(w \cdot \lambda))$ is a free of rank $m_x$ over ${\mathcal X_{\infty,x,\mathcal R}^{I-\qtri,ww_0}}$. By Remark \ref{rema:2_U(t)actions} its follows that $\mathcal M_{\infty,x,\mathcal R}({M}_I(w \cdot \lambda))$ is a free of rank $m_x$ over ${\overline{\mathcal X}_{\infty,x,\mathcal R}^{I-\qtri,ww_0}}$.
 
 The same argument also applies to $\mathcal M_{\infty,x,\mathcal R}(\widetilde{M}_I(w \cdot \lambda))$.
\end{proof}

\begin{prop}\label{prop:Lwfree}
For all $w \in W$ the sheaf $\mathcal B_x(L(w\cdot \lambda))$ is cyclic. Moreover, for all $w \in W$ such that $ww_0 \geq w_{x,\mathcal R}$ the sheaf $\mathcal M_\infty(L(w \cdot \lambda))$ is free of rank $m_x$ over its support.
\end{prop}

\begin{proof}
  Recall that, for $w\in W$, $Z_w$ is the closure in
  $\widetilde{\mathcal N}\times_{\mathcal N}\widetilde{\mathcal N}$ of
  the preimage $V_w$ of the Bruhat Cell
  $U_w=\underline G_L(1,w)\subset\underline G_L/\underline
  B\times\underline G_L/\underline B$. By
  \cite[Prop.~3.3.4]{chriss_ginzburg}, $V_w$ can be identified with
  the conormal bundle of $U_w$ in
  $\widetilde{\mathcal N}\times\widetilde{\mathcal N}\simeq
  T^*(\underline G_L/\underline B\times\underline G_L/\underline
  B)$. As $\frakg$ is isomorphic to direct sum of copies of
  $\mathfrak{gl}_3$, the closure $\overline{U_w}$ of $U_w$ in
  $\underline G_L/\underline B\times\underline G_L/\underline B$ is
  smooth, hence a local complete intersection. This proves that the
  conormal bundle of $\overline{U_w}$ is a closed smooth subscheme of
  $\widetilde{\mathcal N}\times\widetilde{\mathcal N}$ containing
  $V_w$ as an open dense subset so that it coincides with $Z_w$ and
  $Z_w$ is smooth. This implies that $\mathfrak Z_w$ is a smooth. As
  $\mathcal M_{\infty,x,\mathcal R}(L(ww_0\cdot\lambda))$ is
  Cohen--Macaulay, it follows from Proposition
  \ref{prop:supports_faisceaux}~\ref{prop:supports_faisceaux1bis} and
  from the fact that $a_{w,w'}=0$ for $w\neq w'$ (see
  \cite[Rk.~2.4.5]{BHS3}) that the sheaf
  $\mathcal M_{\infty,x,\mathcal R}(L(ww_0\cdot\lambda))$ is locally
  free over its support.
\end{proof}

\subsection{The case of dual Vermas}
\label{sec:dualVermas}

For later use, let us recall the following Lemma.
\begin{lemma}\label{lemma:diagmapsurj}
  Let $R$ be a commutative local ring and let $I \subset J$ two ideals
  of $R$. Let $m\geq1$ and $\pi : (R/I)^m \fleche (R/J)^m$ a
  surjective $R$-linear map. Then there exist isomorphisms
  \[ \varphi : (R/J)^m \fleche (R/J)^m, \quad \psi : (R/I)^m \fleche (R/I)^m\]
  such that $\varphi \circ \pi=\pi \circ \psi=\can^{\oplus m}$ where
  $\can : R/I \fleche R/J$ is the quotient map.
\end{lemma}
  
\begin{proof}
  Let $(e_1,\dots,e_m)$ be the standard basis of $(R/I)^m$ as an
  $(R/I)$-module and $(f_1,\dots,f_m)$ the standard basis of
  $(R/J)^m$. Then $(\pi(e_1),\dots,\pi(e_m))$ is a generating family
  of $(R/J)^m$. As a surjective endomorphism of a module
  is bijective, we see that $(\pi(e_1),\dots,\pi(e_m))$ is also a
  basis of $(R/I)^m$. Therefore we can define $\varphi$ by the formula
  $\varphi(\pi(e_i))=f_i$. Now, for any $1\leq i\leq m$, let
  $f'_i\in (R/I)^m$ such that $\pi(f_i') =f_i$. By Nakayama Lemma the
  family $(f'_1,\dots,f'_m)$ generates $(R/I)^m$ and so is a basis of
  $(R/I)^m$. We can therefore define $\psi$ by the formula
  $\psi(e_i) = f_i'$.
\end{proof}
  
We will use the previous Corollary \ref{cor:freeness} to start a devissage which will be assured by the following two Lemmas.

\begin{lemma}\label{lemm:quotient}
  Let $M$ be an object of $\cO_{\chi_\lambda}$ and let $Q_1,\dots,Q_r$
  be quotients of $M$. Let $Q$ be the smallest quotient of $M$
  dominating all the $Q_i$, i.e.~$Q=M/(M_1\cap \cdots\cap M_r)$ where
  $M_i=\ker(M\rightarrow Q_i)$ for $1\leq i\leq r$. We assume that
  \begin{enumerate}[(i)]
  \item for any $1\leq i\leq r$, the sheaf
    $\mathcal M_{\infty,x,\mathcal R}(Q_i)$ is free of rank $m_x$ over
    it support;
  \item for any $1\leq i\leq r$, the sheaf $\mathcal B_x(Q_i)$ is cyclic
    (generated by one element);
  \item for any $1\leq i\leq r$,
    $\Supp \mathcal M_{\infty,x,\mathcal R}(Q_i)=\Supp \mathcal
    B_x(Q_i)$ ;
  \item the sheaf $\mathcal B_x(Q)$ is cyclic.
  \end{enumerate}
  Then the sheaf $\mathcal M_{\infty,x,\mathcal R}(Q)$ is free of rank
  $m_x$ over its support and
  \[\Supp(\mathcal M_{x,\infty,\mathcal R}(Q) = \Supp(\mathcal
    B_x(Q)).\qedhere\]
\end{lemma}

\begin{proof}
  To ease notation we note $m=m_x$. Let's prove the result when $r=2$.
  Let $A=\overline R_{\infty,x,\mathcal R}^{\qtri}$ be the ring of
  global sections of
  $\overline{\mathcal X}^{\qtri}_{\infty,x,\mathcal R}$ and let
  $I_i=\Ann(\mathcal B_x(Q_i))$ for $i\in\set{1,2}$. Define $Q_0$ the
  largest common quotient of $Q_1$ and $Q_2$,
  i.e.~$Q_0=M/(M_1+M_2)$. Then we have a short exact sequence
  \[ 0\longrightarrow Q\longrightarrow Q_1\oplus Q_2\longrightarrow
    Q_0\longrightarrow0. \] By exactness of
  $\mathcal M_{\infty,x,\mathcal R}$, we have a short exact sequence
  \[ 0\longrightarrow\mathcal M_{\infty,x,\mathcal
      R}(Q)\longrightarrow \mathcal M_{\infty,x,\mathcal
      R}(Q_1)\oplus\mathcal M_{\infty,x,\mathcal
      R}(Q_2)\longrightarrow\mathcal M_{\infty,x,\mathcal
      R}(Q_0)\longrightarrow0 \] where the map
  $Q_1\oplus Q_2\rightarrow Q_0$ is given by $(x,y)\mapsto x-y$.

  We fix isomorphisms
  $(A/I_i)^m\xrightarrow{\sim}\mathcal M_{\infty,x,\mathcal R}(Q_i)$ for
  $i\in\set{1,2}$. As $Q_0$ is a quotient of both $Q_1$ and $Q_2$, we
  have surjective maps
  \[(A/I_i)^{m} \fleche \mathcal M_{\infty,x,\mathcal R}(Q_i) \fleche
  \mathcal M_{\infty,x,\mathcal R}(Q_0),\] which factor through
  $(A/(I_1+I_2))^{m}$. Using Lemma \ref{lemma:diagmapsurj} we can
  choose the previous isomorphisms such that the following diagram
  commutes
  \begin{equation}
    \label{eq:diagram_M}
    \begin{tikzcd}
      (A/I_1)^{m}\oplus (A/I_2)^{m} \ar[r,"(x{,}y)\mapsto x-y"] \ar[d,"\simeq"] & A/(I_1+I_2)^{m} \ar[r] \ar[d,twoheadrightarrow] & 0 \\
      \mathcal
      M_{\infty,x,\mathcal R}(Q_1)\oplus\mathcal M_{\infty,x,\mathcal
        R}(Q_2)\ar[r] & \mathcal M_{\infty,x,\mathcal R}(Q_0)\ar[r]&0.
    \end{tikzcd}
  \end{equation}
   As the kernel of the upper horizontal map is isomorphic to
  $(A/(I_1\cap I_2))^m$, we obtain a commutative diagram
  \begin{equation}
    \label{eq:diagram_M_big}
    \begin{tikzcd}
      0 \ar[r] & (A/(I_1\cap I_2))^{m} \ar[r] \ar[d,hookrightarrow] & (A/I_1)^{m}\oplus (A/I_2)^{m} \ar[r] \ar[d,"\simeq"] & A/(I_1+I_2)^{m} \ar[r] \ar[d,twoheadrightarrow] & 0 \\
      0\ar[r] &\mathcal M_{\infty,x,\mathcal R}(Q) \ar[r] & \mathcal
      M_{\infty,x,\mathcal R}(Q_1)\oplus\mathcal M_{\infty,x,\mathcal
        R}(Q_2)\ar[r] & \mathcal M_{\infty,x,\mathcal R}(Q_0)\ar[r]&0.
    \end{tikzcd}
  \end{equation}
  
  As $\Ann(\mathcal B_x(Q))=I_1\cap I_2$ and $\mathcal B_x(Q)$ is
  cyclic, there exists an isomorphism
  $\mathcal B_x(Q)\simeq A/(I_1\cap I_2)$. Moreover, by hypothesis, we
  have $\Supp(\mathcal B_x(Q_i))=\Spec(A/I_i)$ so that the maps
  $A/(I_1\cap I_2)\simeq\mathcal B_x(Q)\twoheadrightarrow\mathcal
  B_x(Q_i)$ factors through isomorphisms $A/I_i\simeq\mathcal
  B_x(Q_i)$. Therefore, by exactness of $\mathcal B_x$, we also have a
  commutatif diagram
  \[
    \begin{tikzcd}
      0 \ar[r] & (A/(I_1\cap I_2)) \ar[r,"x\mapsto (x{,}x)"] \ar[d,"\simeq"] & (A/I_1)\oplus (A/I_2) \ar[d,"\simeq"] &  &  \\
      0\ar[r] &\mathcal B_x(Q) \ar[r] & \mathcal
      B_x(Q_1)\oplus\mathcal B_x(Q_2)\ar[r] & \mathcal
      B_x(Q_0)\ar[r]&0.
    \end{tikzcd} \] This implies that we have an isomorphism
  $A/(I_1+I_2)\simeq\mathcal B_x(Q_0)$. As $\mathcal B_x(Q_0)$ is
  Cohen--Macaulay, so is $A/(I_1+I_2)$. As the ring $A/(I_1+I_2)$ is
  Cohen--Macaulay, the vertical right arrow of diagram
  (\ref{eq:diagram_M}) is a surjective map
  $(A/(I_1+I_2))^{m}\twoheadrightarrow\mathcal M_{\infty,x,\mathcal
    R}(Q_0)$ between two Cohen--Macaulay modules with the same cycle
  by Corollary \ref{cor:bezandM}. It is therefore an isomorphism and
  the Snake Lemma allows us to conclude that the left vertical arrow
  in (\ref{eq:diagram_M_big}) is an isomorphism.

  Assume that the result is proved for some integer $r\geq2$. Let
  $Q_1,\dots,Q_{r+1}$ be quotients of $M$ satisfying the hypotheses of
  the Lemma. Let $Q'$ be the smallest quotient of $M$ dominating all
  the $Q_i$ for $1\leq i\leq r$. Note that $\mathcal B_x(Q')$ is a
  quotient of $\mathcal B_x(Q)$ and is therefore cyclic. By induction,
  $\mathcal M_{\infty,x,\mathcal R}(Q')$ is free of rank $m$ over its
  support and
  $\Supp\mathcal M_{\infty,x,\mathcal R}(Q')=\Supp\mathcal
  B_x(Q')$. The quotient $Q$ is now the smallest quotient of $M$
  dominating $Q'$ and $Q_{r+1}$. Therefore the case $r=2$ implies that
  $\mathcal M_{\infty,x,\mathcal R}(Q)$ is free of rank $m$ over its
  support and
  $\Supp\mathcal M_{\infty,x,\mathcal R}(Q)=\Supp\mathcal B_x(Q)$,
  which concludes the induction.
\end{proof}

\begin{lemma}\label{lemm:rk_mx_implies_locfree}
  Let $M$ be an object of the category $\cO_{\chi_\lambda}$. Assume
  that $\mathcal M_{\infty,x,\mathcal R}(M)$ is generated by $m_x$
  elements and $\mathcal B_x(M)$ is cyclic. Then
  $\mathcal M_{\infty,x,\mathcal R}(M)$ is locally free of rank $m_x$
  over its support, its support is Cohen--Macaulay and
  $\Supp\mathcal M_{\infty,x,\mathcal R}(M)=\Supp\mathcal B_x(M)$.
\end{lemma}

\begin{proof}
  We prove the result by induction on the length of $M$. If $M$ is
  simple this is done in Proposition \ref{prop:Lwfree}. Thus we can assume that we have a
  short exact sequence
  \[ 0\longrightarrow L\longrightarrow M\longrightarrow
    Q\longrightarrow 0 \]
  with $L$ simple such that $\mathcal M_{\infty,x,\mathcal R}(L)\neq0$
  and that the result is true for $Q$. Let $I=\Ann(\mathcal
  M_{\infty,x,\mathcal R}(M))$, $I_B=\Ann(\mathcal B_x(M))$,
  $J=\Ann(\mathcal B_x(Q))$ and $K=\Ann(\mathcal B_x(L))$. Then we
  have two short exact sequences
  \begin{gather*}
    0\longrightarrow\mathcal B_x(L) \longrightarrow\mathcal
    B_x(M)\longrightarrow \mathcal B_x(Q)\longrightarrow0 \\
    0\longrightarrow\mathcal M_{\infty,x,\mathcal R}(L)
    \longrightarrow\mathcal M_{\infty,x,\mathcal
      R}(M)\longrightarrow\mathcal M_{\infty,x,\mathcal
      R}(Q)\longrightarrow0.
  \end{gather*}
  The first exact sequence shows that $\widehat R_{\infty,x}/K\simeq
  J/I_B$ so that $I_B=JK$. The second exact sequence shows that
  $I_B\subset I$. Therefore, as $\mathcal M_{\infty,x,\mathcal R}(M)$
  is generated by $m_x$ elements, we have a surjective map
  \[ \mathcal B_x(M)^{m_x}\simeq(\widehat
    R_{\infty,x}/I_B)^{m_x}\twoheadrightarrow\mathcal
    M_{\infty,x,\mathcal R}(M). \] These modules are both
  Cohen--Macaulay of the same dimension with identical associated
  maximal cycle by Corollary \ref{cor:bezandM}, therefore this map is
  an isomorphism and $I_B=I$. Moreover as
  $\mathcal M_{\infty,x,\mathcal R}(M)$ is Cohen--Macaulay, so is its
  support.
\end{proof}

\begin{theor}
\label{thm:Vermadualfree}
  For any $w\in W$ such that $ww_0\geq w_{x,\mathcal R}$, the coherent
  sheaf
  $\mathcal M_{\infty,x,\mathcal R}(M(w\cdot\lambda)^\vee)$ is
  locally free of rank $m_x$ over its support.
\end{theor}

\begin{proof}
  As $M(w\cdot\lambda)^\vee$ is a quotient of $M(\lambda)^\vee$ for
  any $w\in W$, Lemma \ref{lemm:rk_mx_implies_locfree} implies that it
  is sufficient to prove the result for $w=1$.

  Recall that $W=\prod_{\tau : F \hookrightarrow L}W_\tau$ and write
  $w_{x,\mathcal R}=(w_{x,\tau})$. Let $J\subset\Hom(F,L)$ be the set
  embeddings such that $w_{x,\tau}=1$. Let $E$ be the set of elements
  $w=(w_v)\in W$ such that $w_\tau\in\set{s_1,s_2}$ if $\tau\in J$ and
  $w_\tau=1$ if $\tau\notin J$. By Corollary \ref{cor:freeness} and
  Lemma \ref{lemm:rk_mx_implies_locfree}, for $w\in E$, the module
  $\mathcal M_{\infty,x,\mathcal R}(M(w\cdot\lambda)^\vee)$ is free of
  rank $m_x$ over its support and
  $\mathcal M_{\infty,x,\mathcal R}(M(w\cdot\lambda)^\vee)=\mathcal
  B_x(M(w\cdot\lambda)^\vee)^{m_x}$. Let $Q$ be the smallest quotient of
  $M(\lambda)^\vee$ dominating all the $M(w\cdot\lambda)^\vee$ for
  $w\in E$. Lemma \ref{lemm:quotient} implies that $\mathcal M_{\infty,x,\mathcal R}(Q)$ is free of
  rank $m_x$ over its support and
  $\mathcal M_{\infty,x,\mathcal R}(Q)=\mathcal
  B_x(Q)^{m_x}$. Let $N$ be the kernel of the map $M\twoheadrightarrow
  Q$. 

  Let $I$ of the form $\coprod_{\tau\in J}\set{s_{i_\tau}}$ where
  $i_\tau\in\set{1,2}$. Then the image of the map
  $M_I(\lambda)^\vee\hookrightarrow M(\lambda)^\vee\twoheadrightarrow
  Q$ is
  $Q_I\coloneqq\Boxtimes_{\tau\in J}L(s_{3-i_\tau}\cdot
  \lambda_\tau)\Boxtimes_{\tau\notin J}M(\lambda_\tau)^\vee$. By
  Corollary \ref{cor:freeness}, the module
  $\mathcal M_{\infty,x,\mathcal R}(M_I(\lambda)^\vee)$ is free of
  rank $m_x$ over its support. Thus
  $\mathcal M_{\infty,x,\mathcal R}(Q_I)$ is generated by $m_x$
  elements, and its quotient
  \[ L_I := \Boxtimes{\tau \in J}
    L(s_{3-i_\tau}\cdot\lambda_\tau)\boxtimes \Boxtimes_{\tau\notin J}
    M(w_{x,\tau}w_0 \cdot \lambda_\tau)^\vee,\] satisfies
  \[\mathcal M_{\infty,x,\mathcal R}(L_I) = \mathcal
    M_{\infty,x,\mathcal R}\big(\Boxtimes_{\tau \in J}
    L(s_{3-i_\tau}\cdot\lambda_\tau)\boxtimes \Boxtimes_{\tau\notin J}
    L(w_{x,\tau}w_0 \cdot \lambda_\tau)\big).\] by Proposition
  \ref{prop:supports_faisceaux}. Moreover, by Proposition
  \ref{prop:Lwfree}, this module is free of rank $m_x$ over its
  support so that its fiber at $x$ has dimension $m_x$. This implies that
  the following surjective maps are all isomorphisms
  \begin{align*}
  k(x)^{m_x} \simeq \mathcal M_{\infty,x,\mathcal
      R}(M_I(\lambda)^\vee)\otimes k(x)&\xrightarrow{\sim}\mathcal
    M_{\infty,x,\mathcal R}(Q_I)\otimes k(x)\\
    &\xrightarrow{\sim}\mathcal M_{\infty,x,\mathcal R}(L_I)\otimes
    k(x) \simeq k(x)^{m_x} . 
    \end{align*}
    As moreover
  $\ker(M_I(\lambda)^\vee\rightarrow Q_I)=N\cap M_I(\lambda)^\vee$, we
  see that the map
  \[ \mathcal M_{\infty,x,\mathcal R}(N\cap M_I(\lambda)^\vee)\otimes
    k(x)\longrightarrow\mathcal M_{\infty,x,\mathcal
      R}(M(\lambda)^\vee)\otimes k(x) \] is zero. As $M(\lambda)^\vee$
  is multiplicity-free, we have $N=\sum_I(N\cap
  M_I(\lambda)^\vee)$ and we conclude that the map
  \[ \mathcal M_{\infty,x,\mathcal R}(N)\otimes
    k(x)\longrightarrow\mathcal M_{\infty,x,\mathcal
      R}(M(\lambda)^\vee)\otimes k(x) \] is zero. Therefore
  $\mathcal M_{\infty,x,\mathcal R}(M(\lambda)^\vee)\otimes
  k(x)\simeq\mathcal M_{\infty,x,\mathcal R}(Q)\otimes k(x)\simeq
  k(x)^{m_x}$ and we conclude with Lemma
  \ref{lemm:rk_mx_implies_locfree} since
  $\mathcal B_x(M(\lambda)^\vee)$ is cyclic.
\end{proof}

\subsection{The case of the antidominant projective}
\label{sec:antidominant_proj}

\begin{theor}\label{theo:antidominant}
  The coherent sheaf
  $\mathcal M_{\infty,x,\mathcal R}(P(w_0\cdot\lambda))$ is free of
  rank $m_x$ over its support.
\end{theor}

\begin{proof}
  Recall that $A=U(\mathfrak t)_{\mathfrak m}$ and set
  $D\coloneqq L\otimes_{A^W}A$. By Proposition
  \ref{prop:endomorphismensatz}, the action of $Z(\mathfrak g)$ on
  $P(w_0\cdot\lambda)$ induces a structure of $D$-module on
  $P(w_0\cdot\lambda)$. As $M(\lambda)^\vee$ is an injective object,
  it follows from \cite[Prop.~6]{Soerg_KatO}, that
  $M(\lambda)^\vee\simeq P(w_0\cdot\lambda)\otimes_D(D/\mathfrak
  m_D)$, where $\mathfrak m_D$ is the maximal ideal of $D$. We have
  also a local map of local algebras
  $\alpha : D\rightarrow\cO_{\overline{\mathcal X}_{\infty,x,\mathcal
      R}^{\qtri}}$ defined in section \ref{sec:two-acti-umathfr}. It
  follows from Corollary \ref{cor:compatibility_center} that these
  define the same action of $D$ on
  $\mathcal M_{\infty,x,\mathcal R}(P(w_0\cdot\lambda))$. As moreover
  the functor $\mathcal M_{\infty,x,\mathcal R}$ is exact, we have an
  isomorphism
  $\mathcal M_{\infty,x,\mathcal R}(M(\lambda)^\vee)\simeq\mathcal
  M_{\infty,x,\mathcal R}(P(w_0\cdot\lambda))\otimes_D(D/\mathfrak
  m_D)$. As moreover the map
  $A\otimes_{A^W}A\rightarrow \cO_{\overline{\mathcal
      X}_{\infty,x,\mathcal R}^{\qtri}}$ is a local map of local
  rings, we have an isomorphism
  $\mathcal M_{\infty,x,\mathcal R}(P(w_0\cdot\lambda))\otimes
  k(x)\xrightarrow{\sim}\mathcal M_{\infty,x,\mathcal
    R}(M(\lambda)^\vee)\otimes k(x)$ and thus
  $\dim_L\mathcal M_{\infty,x,\mathcal R}(P(w_0\cdot\lambda))\otimes
  k(x)=m_x$ by Theorem \ref{thm:Vermadualfree}. We conclude by Lemma
  \ref{lemm:rk_mx_implies_locfree}.
\end{proof}

\begin{cor}\label{cor:quotient_antidominant_proj}
  Let $Q$ be a quotient of the anti-dominant projective $P(w_0\cdot\lambda)$ in the category
  $\cO_{\chi_\lambda}$. If
  $\mathcal M_{\infty,x,\mathcal R}(Q)\neq0$,
  then it is finite free of rank $m_x$ over its support and its
  support is Cohen--Macaulay.
\end{cor}

\begin{proof}
  As $\mathcal M_{\infty,x,\mathcal R}(P(w_0\cdot \lambda))$
  (resp.~$\mathcal B_x(P(w_0\cdot \lambda))$) is free of rank $m_x$
  (resp.~1) over its support by Theorems \ref{theo:antidominant} and
  \ref{thm:Bez}, we have that $\mathcal M_{\infty,x,\mathcal R}(Q)$
  (resp.~$\mathcal B_x(Q)$) is generated by at most $m_x$ elements
  (resp.~cyclic). It follows from Lemma
  \ref{lemm:rk_mx_implies_locfree},
  $\mathcal M_{\infty,x,\mathcal R}(Q)$ is free
  of rank $m_x$ over its support and that its support is Cohen--Macaulay.
\end{proof}

\begin{cor}\label{cor:injectivecoversfree}
For all $w \in W$, the coherent sheaf
\[ \mathcal M_{\infty,x,\mathcal R}(P(w\cdot \lambda)^\vee),\]
is free of rank $m_x$ over its support.
\end{cor}

\begin{proof}
  By Corollary \ref{cor:quotient_antidominant_proj}, it is sufficient
  to prove that $\mathcal M_{\infty,x,\mathcal R}(P(w\cdot\lambda)^\vee)$ is non zero and
  that there exists a surjective map
\[ P(w_0\cdot \lambda) \fleche P(w \cdot \lambda)^\vee.\]

As $P(w_0\cdot\lambda)$ is the projective envelope of
$L(w_0\cdot\lambda)$, this is equivalent to showing that the socle of
$P(w\cdot\lambda)$ is isomorphic to $L(w_0\cdot\lambda)$. By
\cite[Thm.~8.1]{Stroppel}, the socle of $P(w\cdot\lambda)$ is
isomorphic to $L(w_0\cdot\lambda)^m$ with
$m=[P(w\cdot\lambda):M(\lambda)]=[M(\lambda):L(w\cdot\lambda)]$ by
\cite[Thm.~3.9]{HumBGG}. As $\frakg$ is isomorphic to a direct sum of
copies of $\mathfrak{gl}_{3,L}$, we have
$[M(\lambda):L(w\cdot\lambda)]=1$ for any $w\in W$.

Moreover, as $[M(\lambda):L(\lambda)]=1$, we have
\[[P(w\cdot\lambda)^\vee:L(\lambda)]=[P(w\cdot\lambda):L(\lambda)]=1.\] As
$\mathcal M_{\infty,x,\mathcal R}(L(\lambda))\neq0$, we have
$\mathcal M_{\infty,x,\mathcal R}(P(w\cdot\lambda)^\vee)\neq0$.
\end{proof}

\subsection{Duality}
\label{sec:duality}
For a Cohen--Macaulay sheaf $\mathcal F$ on $\overline{\mathcal
  X}_{\infty,x,\mathcal R}^{\qtri}$ of dimension $\dim
\overline{\mathcal X}_{\infty,x,\mathcal R}^{\qtri}$, we
write $\omega^\bullet_{\overline{\mathcal X}_{\infty,x,\mathcal R}^{\qtri}}$ for the dualizing complex and set
\[ \mathcal F^\vee := R\Hom_{\overline{\mathcal X}_{\infty,x,\mathcal R}^{\qtri}}(\mathcal F,\omega^\bullet_{\overline{\mathcal X}_{\infty,x,\mathcal R}^{\qtri}})[-\dim \overline{\mathcal X}_{\infty,x,\mathcal R}^{\qtri}].\]
This complex $\mathcal F^\vee$ is a coherent sheaf concentrated in degree $0$ to which we refer to $\mathcal F^\vee$ as the \emph{shifted} Serre dual of $\mathcal F$.

\begin{lemma}\label{lemm:cycles_dualite}
  Let $\mathcal F$ be a maximal Cohen--Macaulay coherent sheaf over
  $\overline{\mathcal X}_{\infty,x,\mathcal R}^{\qtri}$. Then
  $[\mathcal F^\vee]=[\mathcal F]$. As a consequence if
  $\mathcal Y\subset \overline{\mathcal X}_{\infty,x,\mathcal
    R}^{\qtri}$ is a maximal Cohen--Macaulay closed subscheme, we have
  $[\omega_{\mathcal Y}]=[\mathcal Y]$.
\end{lemma}

\begin{proof}
  Let $R$ be local complete regular ring such that
  $\cO_{\overline{\mathcal X}_{\infty,x,\mathcal R}^{\qtri}}$ is
  isomorphic to a quotient of $R$. Then we can compute
  $\mathcal F^\vee$ by the formula
  $\mathcal F^\vee=\Ext^d_R(\mathcal F,R)$ where $d$ is the
  codimension of $\overline{\mathcal X}_{\infty,x,\mathcal R}^{\qtri}$
  in $\Spec(R)$. By definition, we have $[\mathcal F]=\sum_z a(z) z$
  where the sum is over all maximal points in $\Supp(\mathcal F)$ and
  $a(z)$ is the length of the finite length $R_z$-module
  $\mathcal F_z$. Let $z\in\Spec(R)$ be a maximal point of the support
  of $\mathcal F$. The localization $R_z$ of $R$ at $z$ is a local
  regular ring and we have
  $\mathcal F^\vee_z\simeq\Ext^d_{R_z}(\mathcal F_z,R_z)$. As
  $\Ext^d_{R_z}(-,R_z)$ is a an exact functor on the subcategory of
  finite length $R_z$-modules and
  $\dim_{k(z)}\Ext_{R_z}^d(k(z),R_z)=1$, the
  length of the $R_z$-module $\Ext^d_{R_z}(\mathcal F_z,R_z)$ is
  $a(z)$. So we have the proved the claim.
\end{proof}

\begin{prop}\label{prop:subobject}
  Let $M$ be a subobject of the anti-dominant projective $P(w_0\cdot\lambda)$. Assume that
  $\mathcal M_{\infty,x,\mathcal R}(M)\neq0$ and let $\mathcal Y$ be
  the support of $\mathcal M_{\infty,x,\mathcal R}(M)$. Then
  $\mathcal M_{\infty,x,\mathcal R}(M)$ is isomorphic to
  $\omega_{\mathcal Y}^{m_x}$ and $\mathcal Y$ is Cohen--Macaulay.
\end{prop}

\begin{proof}
  Let $Q$ be the quotient of $P(w_0\cdot\lambda)$ by $M$. If $\mathcal
  M_{\infty,x,\mathcal R}(Q)=0$, then Theorem \ref{theo:antidominant} implies
  the result. So we can assume that $\mathcal M_{\infty,x,\mathcal
    R}(M)\neq0$ and $\mathcal M_{\infty,x,\mathcal
    R}(Q)\neq0$. By Corollary \ref{cor:quotient_antidominant_proj},
  $\mathcal M_{\infty,x,\mathcal R}(Q)$ is isomorphic to $\mathcal
  O_{\mathcal Z}^{m_x}$ for $\mathcal Z\subset\overline{\mathcal
    X}_{\infty,x,\mathcal R}^{\qtri}$ maximal
  Cohen--Macaulay. Using Lemma \ref{lemma:diagmapsurj}, we can
  construct a commutative diagram
  \[
    \begin{tikzcd}
      0 \ar[r] & \mathcal M_{\infty,x,\mathcal R}(M) \ar[r] \ar[d,"\simeq"] & \mathcal M_{\infty,x,\mathcal R}(P(w_0\cdot\lambda)) \ar[r] \ar[d,"\simeq"] & \mathcal M_{\infty,x,\mathcal R}(Q) \ar[r] \ar[d,"\simeq"] & 0 \\
      0\ar[r] & \Ker \ar[r] & \cO_{\overline{\mathcal
          X}_{\infty,x,\mathcal R}^{\qtri}}^{m_x} \ar[r,"\can^{m_x}"] &
      \cO_{\mathcal Z}^{m_x}\ar[r]&0.
    \end{tikzcd} 
  \]
  Let $I$ be the ideal defining $\mathcal Z$. As
  $\overline{\mathcal X}^\qtri_{\infty,x,\mathcal R}$ and $\mathcal Z$ are
  Cohen--Macaulay of the same dimension, the involutivity of the
  duality implies that we have $I\simeq \omega_{\cO_{\mathcal Y}}$
  where $\mathcal Y=\Supp(I)$ so that we have the result.
\end{proof}

\begin{theor}\label{thm:vermadomdualizing}
For all $w \in W$, with $ww_0 \geq w_{x,\mathcal R}$, the sheaf $\mathcal M_{\infty,x,\mathcal R}(M(w\cdot \lambda))$ is isomorphic to
\[ \big( \omega_{\overline{\mathcal{X}}^{\qtri,ww_0}_{\infty,x,\mathcal R}}\big)^{\oplus m_x}.\]
\end{theor}

\begin{proof}
  It follows from Propositions \ref{prop:subobject} and
  \ref{prop:freeness_criterion} that
  $\mathcal M_{\infty,x,\mathcal
    R}(M(w\cdot\lambda))\simeq\omega_{\mathcal Y}^{m_x}$ where
  $\mathcal Y=\Supp\mathcal M_{\infty,x,\mathcal R}(M(w\cdot\lambda))$
  is Cohen--Macaulay. However it follows from Theorem
  \ref{thm:component_Verma} that
  $\mathcal Y\subset\overline{\mathcal X}_{\infty,x,\mathcal
    R}^{\qtri,ww_0}$. By Lemma \ref{lemm:cycles_dualite}, we have an
  equality $[\omega_{\mathcal Y}]=[\mathcal Y]$ and it follows from
  Corollary \ref{cor:bezandM} that
  $[\mathcal M_{\infty,x,\mathcal
    R}(M(w\cdot\lambda))]=m_x[\overline{\mathcal X}_{\infty,x,\mathcal
    R}^{\qtri,ww_0}]$. Therefore we have
  $[\mathcal Y]=[\overline{\mathcal X}_{\infty,x,\mathcal
    R}^{\qtri,ww_0}]$ and thus
  $\mathcal Y=\overline{\mathcal X}_{\infty,x,\mathcal
    R}^{\qtri,ww_0}$.
\end{proof}

We choose for all $\lambda$ dominant weight, and all $w \in W$ a surjective map $\pi_w : P(w_0\cdot\lambda) \fleche P(w\cdot \lambda)^\vee$ (see proof of Corollary \ref{cor:injectivecoversfree}).

\begin{lemma}\label{lemm:lift}
For all map $f_{w,w'} : P(w \cdot \lambda)^\vee \fleche P(w'\cdot \lambda)^\vee$ there exists a map 
$\widetilde{f}_{w,w'} : P(w_0 \cdot \lambda) \fleche P(w_0 \cdot \lambda)$ such that the 
following diagram commutes
  \begin{equation}\label{diag:projfww'}
    \begin{tikzcd}
      P(w_0\cdot\lambda) \ar[r,"\widetilde{f}_{w,w'}"] \ar[d,"\pi_w"] & P(w_0\cdot\lambda) \ar[d,"\pi_{w'}"]\\
      P(w \cdot \lambda)^\vee \ar[r,"f_{w,w'}"] & P(w' \cdot
      \lambda)^\vee
    \end{tikzcd} 
  \end{equation}
\end{lemma}
  
\begin{proof}
  As $\pi_{w'} : P(w_0 \cdot \lambda) \fleche P(w'\cdot \lambda)^\vee$ is surjective and $P(w_0\cdot\lambda)$ is projective, the map 
  $\Hom(P(w_0\cdot\lambda,P(w_0\cdot\lambda)) \fleche \Hom(P(w_0\cdot \lambda),P(w'\cdot\lambda)^\vee)$ is surjective, thus there exists $\widetilde f_{w,w'}$ mapping to $f_{w,w'} \circ \pi_w$. This proves the claim.
\end{proof}

\begin{lemma}
\label{lemma:canisodual}
Let $\mathcal F$ be either $\mathcal B,\mathcal B_x$ or
$\mathcal M_{\infty,x,\mathcal R}$. There exists a family of isomorphisms indexed by $w\in W$
\[ \Psi_w : \mathcal F(P(w\cdot\lambda)^\vee)\overset{\sim}{\fleche} \mathcal F(P(w\cdot\lambda))^\vee.\]
such that for any $w,w'\in W$ and any if $f_{w,w'} : P(w\cdot\lambda)^\vee \fleche P(w'\cdot\lambda)^\vee$, the following diagram commutes
\begin{equation}
  \label{eq:diag_w_w'}
  \begin{tikzcd}
    \mathcal F(P(w\cdot\lambda)^\vee) \ar[r,"\mathcal F(f_{w,w'})"] \ar[d,"\Psi_{w}"] &  \mathcal F(P(w'\cdot\lambda)^\vee) \ar[d,"\Psi_{w'}"]\\
    \mathcal F(P(w\cdot\lambda))^\vee \ar[r,"\mathcal F(f_{w,w'}^\vee)^\vee"] &  \mathcal F(P(w'\cdot\lambda))^\vee
  \end{tikzcd}
\end{equation}
where we denote by the same symbol $(\cdot)^\vee$ the duality in $\mathcal O$ and Serre duality on coherent sheaves.
\end{lemma}

\begin{proof}
  Let $w\in W$. The sheaves $\mathcal F(P(w\cdot\lambda)^\vee)$ and $\mathcal F(P(w\cdot\lambda))^\vee$ are isomorphic to the same quotient of $\mathcal F(P(w_0\cdot\lambda))$ by Theorem \ref{thm:Bez} for $\mathcal B,\mathcal B_x$ and Corollary \ref{cor:injectivecoversfree} and Proposition \ref{prop:subobject} for $\mathcal M_{\infty,x,\mathcal R}$. This implies that there exists an isomorphism $\Psi_w : \mathcal F(P(w\cdot\lambda)^\vee)\xrightarrow{\sim}\mathcal F(P(w\cdot\lambda))^\vee$ such that the following diagram commutes
  
 \begin{equation}\label{diag:proj}
    \begin{tikzcd}
     \mathcal F(P(w_0\cdot\lambda)) \ar[r,"\Psi_{w_0}"] \ar[d,"\mathcal F(\pi_w)"] &  \mathcal F(P(w_0\cdot\lambda))^\vee \ar[d,"\mathcal F(\pi_w^\vee)^\vee"]\\
\mathcal F(P(w\cdot\lambda)^\vee) \ar[r,"\Psi_w"] & \mathcal F(P(w\cdot\lambda))^\vee
    \end{tikzcd} 
  \end{equation}
  
Fix $w,w'$ and let's show that the diagram (\ref{eq:diag_w_w'}) is commutative. Let $f_{w,w'}\in\Hom(P(w\cdot \lambda)^\vee, P(w'\cdot\lambda)^\vee)$. By Lemma \ref{lemm:lift}, there exists a map $\widetilde f_{w,w'}\in\End(P(w_0\cdot\lambda))$ such that the diagram (\ref{diag:projfww'}) is commutative. We first consider the following diagram 
 \begin{equation}\label{diag:proj2}
    \begin{tikzcd}
     \mathcal F(P(w_0\cdot\lambda)) \ar[r,"\Psi_{w_0}"] \ar[d,"\mathcal F(\widetilde{f}_{w,w'})"] &  \mathcal F(P(w_0\cdot\lambda))^\vee \ar[d,"\mathcal F(\widetilde f_{w,w'}^\vee)^\vee"]\\
\mathcal F(P(w_0\cdot\lambda)^\vee) \ar[r,"\Psi_{w_0}"] & \mathcal F(P(w_0\cdot\lambda))^\vee
    \end{tikzcd} 
    \end{equation}
But as $\widetilde{f}_{w,w'} \in \End_{\mathcal O}(P(w_0\cdot
\lambda),P(w_0\cdot \lambda)) \simeq D = L \otimes_{A^W}A$, it follows
from Corollary \ref{cor:compatibility_center} for $\mathcal F=\mathcal
M_{\infty,x,\mathcal R}$ and \cite[Prop.~23]{BezTwo} for $\mathcal
F=\mathcal B_x$, and the fact that $\Psi_{w_0}$ is $\mathcal
O_{\mathcal X^{\qtri}_{\infty,x,\mathcal R}}$-linear, that this
diagram commutes. Now consider the diagram
\[
  \begin{tikzcd}
    & \mathcal F(P(w_0\cdot\lambda)) \ar[rr,"\mathcal F(\pi_{w'}^\vee)^\vee"]    \ar[from=dd,"\mathcal F(\widetilde{f}_{w,w'}^\vee)^\vee" near start]  &  & \mathcal F(P(w'\cdot\lambda))^\vee   \\
    \mathcal F(P(w_0\cdot\lambda)) \ar[rr,crossing over,"\mathcal F(\pi_{w'})" near end]
    \ar[ru,"\Psi_{w_0}"] & & \mathcal F(P(w'\cdot\lambda)^\vee)
    \ar[ru,"\Psi_{w'}"]  & \\
    & \mathcal F(P(w_0\cdot\lambda))
    \ar[rr,"\mathcal F(\pi_{w}^\vee)^\vee" near end]
& & \mathcal F(P(w\cdot\lambda))^\vee
    \ar[uu,"\mathcal F({f}_{w,w'}^\vee)^\vee"'] \\
    \mathcal F(P(w_0\cdot\lambda)) \ar[rr,"\mathcal F(\pi_{w})"] \ar[ru,"\Psi_{w_0}"]
    \ar[uu,"\mathcal F(\widetilde{f}_{w,w'})"] & & \mathcal F(P(w\cdot\lambda)^\vee)
    \ar[ru,"\Psi_{w}"'] \ar[uu,crossing over,"\mathcal F({f}_{w,w'})"
    near start]&
  \end{tikzcd} \]
All faces, except maybe the right hand one (which is the one of the statement), of this cube are commutative diagrams by functoriality and diagrams (\ref{diag:projfww'}),  (\ref{diag:proj}), (\ref{diag:proj2}). Moreover $\mathcal F(\pi_w)$, $\mathcal F(\pi_w^\vee)^\vee$, $\mathcal F(\pi_{w'})$, $\mathcal F(\pi_{w'}^\vee)^\vee$  are surjective, thus the last right hand face also commutes.
\end{proof}

\begin{cor}\label{cor:duality}
For any $M \in \mathcal O_\alg$, there is a compatible choice of isomorphisms
\[ \Psi_M : \mathcal F(M^\vee) \overset{\sim}{\fleche} \mathcal F(M)^\vee,\]
where $\mathcal F$ is either the functor $\mathcal B,\mathcal B_x$ or $\mathcal M_{\infty,x,\mathcal R}$. In particular, $\mathcal F$ is compatible with duality.
\end{cor}

\begin{proof}
Choose a resolution
\begin{equation}\label{eq:resM} \bigoplus_i P(\mu_i) \fleche \bigoplus_j P(\lambda_j) \fleche M \fleche 0.\end{equation}
Then we have two exact sequences
\[ 0 \fleche \mathcal F(M^\vee) \fleche  \bigoplus_j \mathcal F(P(\lambda_j)^\vee) \fleche \bigoplus_i \mathcal F(P(\mu_i)^\vee),\]
and
\[  0 \fleche \mathcal F(M)^\vee \fleche  \bigoplus_j \mathcal F(P(\lambda_j))^\vee \fleche \bigoplus_i \mathcal F(P(\mu_i))^\vee.\]
For the second one, recall that if $K$ denote the Kernel in equation (\ref{eq:resM}) so that
\[ 0 \fleche K \fleche \bigoplus_j P(\lambda_j) \fleche M \fleche 0,\]
then, as $\mathcal F(K)$ is CM of the same dimension as the other modules, we have
\[ 0 \fleche \mathcal F(M)^\vee \fleche \bigoplus_j \mathcal F(P(\lambda_j))^\vee \fleche \mathcal F(K)^\vee \fleche 0,\]
which is exact.
Moreover, by the previous Lemma \ref{lemma:canisodual} we have a commutative diagram with vertical isomorphisms
\begin{center}
\begin{tikzpicture}[description/.style={fill=white,inner sep=2pt}] 
\matrix (m) [matrix of math nodes, row sep=3em, column sep=2.5em, text height=1.5ex, text depth=0.25ex] at (0,0)
{ 
0 & \mathcal F(M^\vee) &\bigoplus_j \mathcal F(P(\lambda_j)^\vee)  &  \bigoplus_i \mathcal F(P(\mu_i)^\vee)& \\
0 &  \mathcal F(M)^\vee&  \bigoplus_j \mathcal F(P(\lambda_j))^\vee &\bigoplus_i \mathcal F(P(\mu_i))^\vee & \\
};
\path[->,font=\scriptsize] 
(m-1-1) edge node[auto] {$$} (m-1-2)
(m-1-2) edge node[auto] {$$} (m-1-3)
(m-1-3) edge node[auto] {$$} (m-1-4)
(m-2-1) edge node[auto] {$$} (m-2-2)
(m-2-2) edge node[auto] {$$} (m-2-3)
(m-2-3) edge node[auto] {$$} (m-2-4)

(m-1-3) edge node[auto] {$\bigoplus_j \Psi_{\lambda_j}$} (m-2-3)
(m-1-4) edge node[auto] {$\bigoplus_i \Psi_{\mu_i}$} (m-2-4);
\end{tikzpicture}
\end{center}
which induces an isomorphism $\Psi_M : \mathcal F(M^\vee)\fleche \mathcal F(M)^\vee$.
\end{proof}

\begin{cor}\label{cor:functors_isom}
  There exists an isomorphism of functors
  $\mathcal B_x^{m_x}\simeq\mathcal M_{\infty,x,\mathcal R}$.
\end{cor}

\begin{proof}
  By a similar argument to the proof of Lemma \ref{lemma:canisodual},
  we can construct a family indexed by $w\in W$ of isomorphisms
  \[\Phi_w : \mathcal B_x(P(w\cdot\lambda)^\vee)^{m_x}\overset{\sim}{\fleche} \mathcal
    M_{\infty,x,\mathcal R}(P(w\cdot\lambda)^\vee)\] such that, for any
  $w,w'\in W$ and any
  $f_{w,w'}\in\Hom(P(w\cdot\lambda)^\vee,P(w'\cdot\lambda)^\vee$, the
  following diagram commutes
  \[
    \begin{tikzcd}
      \mathcal B_x(P(w\cdot)^\vee)^{m_x} \ar[r,"\mathcal B_x(f_{w,w'})"] \ar[d,"\Phi_w"] &  \mathcal B_x(P(w'\cdot\lambda)^\vee) \ar[d,"\Phi_{w'}"]\\
      \mathcal M_{\infty,x,\mathcal R}(P(w\cdot\lambda)^\vee) \ar[r,"\mathcal
      M_{\infty,x,\mathcal R}(f_{w,w'})"] & \mathcal M_{\infty,x,\mathcal R}(P(w'\cdot\lambda)^\vee).
    \end{tikzcd} \] Such a family of isomorphisms provides an
  isomorphism of functors between $\mathcal B_x^{m_x}$ and
  $\mathcal M_{\infty,x,\mathcal R}$ restricted to the full
  subcategory of $\cO_\lambda$ of injective objects. As the category
  $\cO_\lambda$ has enough injectives, this isomorphism extends to all
  of $\cO_\lambda$.
\end{proof}

\subsection{Consequences}

In this section we keep the setting introduced in subsection \ref{subsect:detailn=3}. In particular $n =3$.

\begin{lemma}
\label{lemma:fiberrefinement}
Let $\rho,\lambda,\mathcal R$ be as above and let $x \in \mathcal X_\infty(L)$ the point corresponding to $\rho$.
Then for all $M \in \mathcal O_{\chi_\lambda}$,
\[ \mathcal M_{\infty,x,\mathcal R}(M)\otimes k(x) \simeq \left(\Hom_{U(\mathfrak g)}(M,\Pi^{\rm la}[\mathfrak m_\rho])^{N^0}[\mathfrak m_{\delta_{\mathcal R}}]\right)'.\]
\end{lemma}

\begin{proof}
By construction (see Remark \ref{rema:other_description}), we have
\[ \mathcal M_{\infty,x,\mathcal R}(M)\simeq \left( \Hom_{U(\mathfrak
        g)}(M,\Pi^{\la}_\infty[\mathfrak m_x^\infty])^{N_0}[\mathfrak
      m_{\delta_{\mathcal R}}^\infty]\right)'.\]
By Corollary  \ref{coro:annihilator_I_final}, the $\mathcal X_{\infty} \times \widehat{T}$-structure on the sheaf $\mathcal M_{\infty,x,\mathcal R}(M)$ factors through $\mathcal X_{\infty,x,\mathcal R}^{\qtri} \fleche \mathcal X_{\infty} \times \widehat{T}$. Thus, 
\[ \mathcal M_{\infty,x,\mathcal R}(M) \otimes k(x) \simeq \left( \Hom_{U(\mathfrak
        g)}(M,\Pi^{\la}_\infty[\mathfrak m_x])^{N_0}[\mathfrak
      m_{\delta_{\mathcal R}}]\right)'.\qedhere\]

\end{proof}

\begin{cor}\label{cor:nonclassicalforms}
Let $\delta:T\rightarrow L^\times$ be a continuous character and let $\chi^S:\mathbb{T}^S\rightarrow L$ be a character such that there exists  $f \in S^\dag(K^p)[\chi^S\otimes \delta]$ an overconvergent $p$-adic eigenform on the group $U(3)$.  Assume that the Galois representation $\rho$ associated to $f$ is crystalline strictly dominant and $\varphi$-generic at $p$ satisfying (\ref{hyp:TW}).  Let $r = \vabs{\set{ \tau \in \Sigma_F \mid \omega_{x,\mathcal R,\tau} = 1}}$.
Then
\[ \dim  S^\dag(K^p)[\chi^S\otimes \delta] = 2^r\dim  S^{\rm cl}(K^p)[\chi^S\otimes \delta]\neq 0.\]
\end{cor}

\begin{proof} The assumptions imply that the character $\delta$ is locally algebraic and that it factors as $\delta = \delta_\lambda \delta_\mathcal R$ for some $\lambda \in X^*(\underline T)^+$ and some unramified character $\delta_{\mathcal R}$.
By Breuil's adjunction formula \cite[Théorème 4.3]{BreuilAnalytiqueII} (see also \cite[eq. (5.5)]{BHS3}) and \cite[Lemma 5.2.3]{BHS3} we have
\begin{align*} S^\dag(K^p)[\chi^S\otimes \delta]  &= \Hom_{U(\mathfrak g)}(M(\lambda),\Pi^{\la}[\chi^S])^{N_0}[\mathfrak m_{\delta_\mathcal R}],\\
 S^{\rm cl}(K^p)[\chi^S\otimes \delta] &= \Hom_{U(\mathfrak g)}(L(\lambda),\Pi^{\la}[\chi^S])^{N_0}[\mathfrak m_{\delta_\mathcal R}].\end{align*}
In particular, by Lemma \ref{lemma:fiberrefinement}, these spaces are indentified with the dual vector spaces of the fiber of $\mathcal M_{\infty,x,\mathcal R}(M(\lambda))$ 
resp.~of $\mathcal M_{\infty,x,\mathcal R}(L(\lambda))$ at $k(x)$. Thus, as 
$m_x = \dim \mathcal M_{\infty,x,\mathcal R}(L(\lambda)) \otimes k(x)$, the result is a direct corollary of Theorem \ref{thm:vermadomdualizing} (and Proposition \ref{prop:geomXn=3}).
\end{proof}

We can also deduce the following corollary on the structure of the completed cohomology $\Pi$ (see Definition \ref{def:completed_cohomology}), which is a representation of $G := U(\QQ_p)$. Let $\mathfrak{gl}_3$ be the Lie algebra (over $L$) of the group $\GL_{3}$ and for a dominant $\lambda$ we consider the extension
\[ N(\lambda) = \left[ L(s_1 \cdot \lambda) \oplus L(s_2 \cdot \lambda) - L(\lambda)\right] \in \Ext_{\mathcal O}^1(L(\lambda),L(s_1 \cdot \lambda) \oplus L(s_2 \cdot \lambda) ),\]
which is non trivial when mapped in each of $ \Ext_{\mathcal O}^1(L(\lambda),L(s_i \cdot \lambda))$, for $i = 1,2$. This extension is the quotient of the Verma module $M(\lambda)$ by $M(s_1s_2\cdot\lambda)+M(s_2s_1\cdot\lambda)$.

As before we consider the Lie algebra \[\mathfrak g={\rm Lie}\big(\underline G_L \simeq (\Res_{F\otimes_\QQ \QQ_p/\QQ_p}\GL_3)\times_{\QQ_p} L\big) \simeq \prod_{\tau \in \Sigma_F} \mathfrak{gl}_{3}\]  with Borel $\mathfrak b \simeq \prod_\tau \mathfrak b_\tau$. Associated to a dominant weight $\lambda = (\lambda_\tau)_\tau \in X^*(\underline T)^+$  and $w_{\mathcal R}= (w_{\mathcal R,\tau})_{\tau\in \Sigma_F} \in W$ we define the object  
\[ N(\lambda,w_{\mathcal R}) = 
\left(\Boxtimes_{\tau : w_{\mathcal R,\tau} \neq 1} L(\lambda_\tau)\boxtimes
\Boxtimes_{\tau : w_{\mathcal R,\tau} =1}N(\lambda_\tau)\right)
\] of the category $\mathcal O_{\chi_{\lambda}} =\Boxtimes_\tau \mathcal O_{\chi_{\lambda_\tau}}^{\mathfrak gl_3,\mathfrak b_\tau}$.
We also define \[ S(\lambda,w_{\mathcal R}) = \Boxtimes_\tau S(\lambda_\tau,w_{\mathcal R,\tau}) \in \mathcal O_{\chi_\lambda},\]
where \[ S(\lambda_\tau,w_{\mathcal R,\tau}) = 
\left\{
\begin{array}{ccc}
 \bigoplus_{w \leq w_{\mathcal R,\tau}w_0} L(w\cdot \lambda_\tau) & \text{if } w_{\mathcal R,\tau} \neq 1 \\
\bigoplus_{\ell(w) \neq 1} L(w\cdot \lambda_\tau) \bigoplus N(\lambda_\tau) & \text{if } w_{\mathcal R,\tau} = 1
\end{array}
\right. ,
\]
so that $S(\lambda,w_{\mathcal R}) = \bigoplus_{w \leq w_{\mathcal R}w_0} L(w\cdot \lambda)$ if $w_{\mathcal R,\tau} \neq 1$ for all $\tau$, and
\[ N(\lambda,w_{\mathcal R}) \subset S(\lambda,w_{\mathcal R}),\]
otherwise.

If $M$ is a $U(\mathfrak g)$-module, we denote $\Hom_E(M,E)$ the $U(\mathfrak g)$-module with underlying vector space $\Hom_E(M,E)$ and action of $\mathfrak r \in U(\mathfrak g)$ given by
\[ (\mathfrak r \cdot \phi)(m) :=  \phi(\dot{\mathfrak r}m), \quad \phi \in \Hom_E(M,E), m \in M,\]
where $\mathfrak r \mapsto \dot{\mathfrak r}$ is the anti-involution of $U(\mathfrak g)$ extending $-1$ on $\mathfrak g$.  We denote $\overline{\underline B}$ the Borel opposite to $\underline B$, whose Lie algebra is $\overline{\mathfrak b}$ with $\overline{\mathfrak n}$ its nilpotent radical. We then denote $B = \underline{B}(\QQ_p), \overline{B} = \overline{\underline{B}}(\QQ_p)$ and $\delta_{B}$ the modulus character of $B$. We then denote $M' := \Hom_E(M,E)^{\overline{\mathfrak n}^\infty}$ the vectors which are killed by a finite power of $\overline{\mathfrak n}$.
If $M = \bigoplus_{\lambda \in X^*(\underline T)_L} M_\lambda \in \mathcal O^{\mathfrak g,\mathfrak b}$, then $M' \in \mathcal O^{\mathfrak g,\overline{\mathfrak b}}$. Finally recall that if $M \in \mathcal O^{\mathfrak g,\overline{\mathfrak b}}$ and $\delta$ is a smooth character of $\underline T(\QQ_p)$, then Orlik-Strauch constructed (see \cite{OrlikStrauch} or also \cite{BreuilAnalytiqueI})
\[ \mathcal F_{\overline{B}}^{G}(M,\delta),\]
which is a locally analytic representation of $G$. In particular, locally analytic principal series are of this form : if $M = M(\lambda)^\vee \in \mathcal O^{\mathfrak g,\mathfrak b}$, then
\begin{equation}\label{eq:PSOS} \mathcal F_{\overline{B}}^{G}((M(\lambda)^\vee)',\delta) = \ind_{\overline{B}}^{G}(\delta_{\lambda}\delta)^{\la}.\end{equation}

Let $\rho : \Gal_E \fleche \GL_n(L)$ be a crystalline, Hodge-Tate regular and $\varphi$-generic autodual representation satisfying Hypothesis \ref{hyp:TWtext} such that $\Pi[\mathfrak m_\rho] \neq 0$ where $\mathfrak m_\rho$ is the ideal of $\mathbb T^S \otimes L$ associated to $\rho$.
Let $\mathcal R$ a choice of refinement
and $\delta_{\mathcal R}$ the associated unramified character. 
Denote $\lambda = (\lambda_\tau)_\tau\coloneqq\HT(\rho)-\delta_G \in X^*(T)^+$ the (dominant) algebraic character associated to $\rho$ as before, where 
$\HT(\rho)=(h_{1,\tau}>\cdots >h_{n,\tau})_{\tau \in \Sigma_F}\in X^*(\underline T)$ gives the Hodge-Tate weights of $\rho$. As $\Pi[m_\rho] \neq 0$ and $\rho$ satisfies Hypothesis \ref{hyp:TWtext}, it corresponds to a point $x \in \mathcal X_\infty(L)$. 
Denote $w_{\rho,\mathcal R} = (w_{\rho,\mathcal R,\tau})_{\tau\in \Sigma_F}$ and $m_\rho := m_x \geq 1$ as in Section \ref{subsec:sheavesandsupports}. 

\begin{cor}
For $\rho,\lambda,\mathcal R$ as above and all $w \leq w_{\mathcal R}w_0$, we have
\[ \dim \Hom_{G}(\ind_{\overline{B}}^{G}(\delta_{w \cdot \lambda}\delta_{\mathcal R}\delta_{B}^{-1})^{\rm la},\Pi^{\la}[\mathfrak m_\rho]) = m_\rho.\]
\end{cor}

\begin{proof}
By \cite[Proposition 4.2]{BreuilAnalytiqueII} and \cite[Lemma 5.2.3]{BHS3}, we have, for all $M \in \mathcal O$
\begin{eqnarray*} \Hom_{U(\mathfrak g)}(M,\Pi^{\la}[\mathfrak m_\rho])^{N^0}[\mathfrak m_{\delta_{\mathcal R}}]
\simeq \Hom_{G(\QQ_p)} (\mathcal F_{\overline B}^{G}(M',\delta_{\mathcal R}\delta_{B}^{-1}), 
\Pi^{\la}[\mathfrak m_\rho])\\
\simeq \Hom_{(\mathfrak g,\overline{B}_p)}(M \otimes_L C^\infty_c(N_B(L),\delta_\mathcal R), 
\Pi[\mathfrak m_\rho]).\end{eqnarray*}
Thus, using equation (\ref{eq:PSOS}) and Lemma \ref{lemma:fiberrefinement} we deduce that the statement is equivalent to
\[ \dim \Hom_{U(\mathfrak g)}(M(w\cdot \lambda)^\vee,\Pi^{\la}[\mathfrak m_\rho)[\mathfrak m_\delta] = \dim \mathcal M_{\infty,x,\mathcal R}(M(w\cdot \lambda)^\vee) \otimes k(x) = m_x,\]
which is Theorem \ref{thm:Vermadualfree}.
\end{proof}

\begin{cor} 
\label{cor:JOPila}For $\rho,\lambda,\mathcal R$ as before,
we have an injection of $(\mathfrak g,B(L))$-modules
\[ \left(S(\lambda,w_{\rho,\mathcal R}) \otimes_L C^\infty_c(N_B(L),\delta_\mathcal R)\right)^{\oplus m_\rho} \hookrightarrow \Pi^{\la}[\mathfrak m_\rho],\]
or, equivalently, an injection of $G$-representations
\[ \mathcal F_{\overline B}^{G}(S(\lambda,w_{\rho,\mathcal R})',\delta_{\mathcal R}\delta_{B}^{-1})^{\oplus m_\rho} \subset \Pi[\mathfrak m_\rho].\]
Moreover each map from  $\mathcal F_{\overline B}^{G}(M(w \cdot \lambda)',\delta_{\mathcal R}\delta_{B}^{-1})$ to $\Pi[\mathfrak m_\rho]$ factors through the previous representation $\mathcal F_{\overline B}^{G}(S(\lambda,w_{\rho,\mathcal R})',\delta_{\mathcal R}\delta_{B}^{-1})$.
\end{cor}

\begin{proof}
The two statements about injections are equivalent and each of the $m_\rho$ asserted maps comes from a section of
\[ \Hom_{U(\mathfrak g)}(S(\lambda,w_{\rho,\mathcal R}),\Pi^{\la}[\mathfrak m_\rho])^{N^0}[\mathfrak m_{\delta_{\mathcal R}}],\]
by the adjunction recalled in the proof of the previous corollary.

We already know, by \cite{BHS3}, that for all $w \leq w_{\rho,\mathcal R}w_0$ we have, in previously used notations
\[ \dim \Hom_{U(\mathfrak g)}(L(w \cdot \lambda),\Pi^{\la}[\mathfrak m_\rho])^{N^0}[\mathfrak m_{\delta_{\mathcal R}}] = m_x=m_\rho.\]
 Moreover, for each $w_{\rho,\mathcal R}w_0 \geq w$ with $w_{\rho,\mathcal R,\tau} \neq 1$ if $w_\tau = 1$, we have
\begin{align*}
m_\rho&=\dim \Hom_{U(\mathfrak g)}(M(w \cdot \lambda),\Pi^{\la}[\mathfrak m_\rho])^{N^0}[\mathfrak m_{\delta_{\mathcal R}}] \\ &  = \dim \Hom_{U(\mathfrak g)}(L(w \cdot \lambda),\Pi^{\la}[\mathfrak m_\rho])^{N^0}[\mathfrak m_{\delta_{\mathcal R}}],
\end{align*}
by Corollary \ref{cor:freeness}. Thus, for those $w$, all maps from $M(w\cdot \lambda)$ factors through $L(w \cdot \lambda)$. 

So we really need to take care of the direct factors of $S(\lambda,w_{\rho,\mathcal R})$ where a factor $N(\lambda_\tau)$ appears. Such a factors is of the form
\[ \Boxtimes_{\tau \in I_1} L(w_\tau \cdot \lambda_\tau) \boxtimes \Boxtimes_{\tau \in I_2} N(\lambda_\tau), \quad \Sigma = I_1 \sqcup I_2, \]
and is a quotient of $M(w \cdot \lambda)$ where $w = (w_\tau)$ with $w_{\tau} = 1$ if $\tau \in I_2$, and even of $M_I(w \cdot \lambda)$ where $I = \{ s_{1,\tau} | \tau \in I_1 \text{ such that } w_\tau = 1 = w_{\rho,\mathcal R,\tau}\}$.

We first prove that any map from $M_I(w \cdot \lambda)$ has to factor through $S(\lambda,w_{\rho,\mathcal R})$ and more precisely through the previous factor.

Choose $\tau_0 \in I_2$ so that $w_{\rho,\mathcal R,\tau_0} = w_{\tau_0} = 1$ and for $i = 1,2$ let $s_i^{\tau_0} \in W$ with $(s_i^{\tau_0})_\tau = w_\tau =$ if $\tau \neq \tau_0$, and $w_{\tau_0}=s_i$. Then
\[ M_I(s_i^{\tau_0} \cdot \lambda) \subset M_I(w \cdot \lambda).\]
Moreover, $M_I(s_i^{\tau_0}\cdot\lambda)$ has a quotient \[Q_i^{\tau_0} := \Boxtimes_{\tau \in I_1} L(\lambda_\tau) \boxtimes L(s_i \cdot \lambda_{\tau_0}) \boxtimes \Boxtimes_{\tau \in I_2 : \tau \neq \tau_0} M(w_\tau \cdot\lambda_\tau),\]
and we first prove that maps from $M_I(s_i^{\tau_0}\cdot \lambda)$ into $\Pi^{\la}[\mathfrak m_\rho]$ factors through $Q_i^{\tau_0}$. 
This is equivalent to proving that $\mathcal M_{\infty,x,\mathcal R}(M(s_i^{\tau_0})) \otimes k(x)\fleche 
\mathcal M_{\infty,x,\mathcal R}(M(\lambda))\otimes k(x)$ factors through 
$\mathcal M_{\infty,x,\mathcal R}(Q_i^{\tau_0})\otimes k(x)$. 
By Corollary \ref{cor:functors_isom}, this is equivalent to the same question for $\mathcal B_x$.

\begin{claim}
If $G = G_1 \times G_2$, $\lambda = (\lambda_1,\lambda_2)$ is an algebraic weight and $\mathcal O_{\chi_\lambda} = 
\mathcal O_{\chi_{\lambda_1}} \boxtimes \mathcal O_{\chi_{\lambda_2}}$, then \[ \mathcal B_G(M_1 \boxtimes M_2) = \mathcal B_{G_1}(M_1) \boxtimes \mathcal B_{G_2}(M_2),\]
where $\mathcal B_H$ is Bezrukavnikov's functor of Theorem \ref{thm:Bez} for the group $H$, under the obvious isomorphism of Steinberg varieties
\[ X_G = X_{G_1} \times X_{G_2}.\]
\end{claim}

\begin{proof}
This follows from the very construction of Bezrukvanikov's functor.  The functors $\mathcal B_G$ is even defined on the larger category $D^b(\Perv_{\underline N}(\underline G/\underline N))$ and compatible with its monoidal structure. Then, by Theorem \ref{thm:Bez}, we know that dual Vermas are sent to the structure sheaves of the respective components by $\mathcal B_G$, and similarly for 
$\mathcal B_{G_i}$. In particular we have an isomorphism
\[ \mathcal B_G(M((w_1,w_2) \cdot (\lambda_1,\lambda_2))^\vee) \simeq \mathcal B_{G_1}(M(w_1 \cdot \lambda_1)^\vee)\boxtimes \mathcal B_{G_2}(M(w_2 \cdot \lambda_2)^\vee).\]
As the functors $\mathcal B_G$ and $\mathcal B_{G_1}\boxtimes\mathcal B_{G_2}$ are both monoïdal and triangulated, using 
translation functors we deduce that $\mathcal B_G$ and $\mathcal B_{G_1}\boxtimes\mathcal B_{G_2}$ are isomorphic on projective 
objects. Thus by the same proof of Lemma \ref{lemma:canisodual} and Lemma \ref{cor:functors_isom}, we deduce the isomorphism
 of functors on $\mathcal O_{\chi_\lambda}$.
\end{proof}

So we want to prove that $\mathcal B_x(M_I(s_i^{\tau_0}\cdot \lambda))\otimes k(x) \fleche \mathcal B_x(M_I(w \cdot \lambda))\otimes k(x)$ factors through $Q_i^{\tau_0}$. By the previous claim, it suffices to show one $\tau$ at a time using $M(s_i^{\tau_0}\cdot\lambda) = M(s_i \cdot \lambda_{\tau_0}) \boxtimes M_I(w^{\tau_0} \cdot \lambda^{\tau_0})$. Thus when $\tau \neq \tau_0$ this is obvious for $\tau \in I_2$ and reduces to freeness of $X_{3,I,w_\tau}$ when $\tau \in I_1$ as before. For $\tau = \tau_0$ this amount to show that, for $k \neq \ell \in \{1,2\}$ the map
\[ \mathcal B_x(M(s_ks_\ell \cdot \lambda_{\tau_0}))\otimes k(x) \fleche  \mathcal B_x(M(s_i \cdot \lambda_{\tau_0})) \otimes k(x),\]
vanishes. But as $\mathcal B_x(M(s_i \cdot \lambda_{\tau_0}))$ is free of rank 1, this is obvious. Thus, we have a factorisation through $Q_i^{\tau_0}$ for all $\tau_0, i=1,2$, thus  
\[\mathcal M_{\infty,x,\mathcal R}(M_I(w \cdot \lambda))\otimes k(x)=\mathcal M_{\infty,x,\mathcal R}(M)\otimes k(x),\]
where \[M = \Boxtimes_{\tau : w_{\rho,\mathcal R,\tau} \neq 1} L(\lambda_\tau) \boxtimes \Boxtimes_{\tau : w_{\rho,\mathcal R,\tau} = 1} \underbrace{M(\lambda_\tau)/(M(s_1s_2 \cdot \lambda_\tau)+M(s_2s_1\cdot\lambda_\tau))}_{N(\lambda_\tau)}.\]
Now we prove the last part of the statement, i.e. any map from (the Orlik-Strauch induction of) a Verma $M(w \cdot \lambda)$ to $\Pi^{\la}[\mathfrak m_\rho]$ will factor through $S(\lambda,w_{\mathcal R})$. Assume given a map in $\Hom_{U(\mathfrak g)}(M(w \cdot \lambda),\Pi^{\la}[\mathfrak m_\rho])^{N^0}
[\mathfrak m_{\delta_\mathcal R}]$, and let $I_1 = \{ \tau \in \Sigma | w_{\rho,\mathcal R,\tau} = w_\tau = 1\}$ and 
$I_2$ its complement, then by the previous argument the map factors through  
\[\Boxtimes_{\tau \in I_1} L(w_\tau \cdot \lambda_\tau) \boxtimes \Boxtimes_{\tau \in I_2} N(\lambda_\tau).\]
As any quotient of this module is a sub-representation of $S(\lambda,w_{\rho,\mathcal R})$, we have that any map 
\[\mathcal F_{\overline B}^{G}(M(w \cdot \lambda)^*,\delta_{\mathcal R}\delta_{B}^{-1})\rightarrow \Pi[\mathfrak m_\rho]\] factors through $\mathcal F_{\overline B}^{G}(S(\lambda,w_{\rho,\mathcal R})^*,\delta_{\mathcal R}\delta_{B}^{-1})$.

We now prove the injective part. As $S(\lambda,w_{\rho,\mathcal R})$ is the direct sum of terms of the form, for $w \in W$,
\[ M_{w,I_1,I_2} = \Boxtimes_{\tau \in I_1} L(w_\tau \cdot \lambda_\tau) \boxtimes \Boxtimes_{\tau \in I_2} N(\lambda_\tau), \quad \Sigma = I_1 \sqcup I_2, \]
where $I_2 \subset \{ \tau \in \Sigma | w_{\rho,\mathcal R,\tau} = w_\tau = 1\}$, we first prove that the direct sum of $m_\rho$ copies of (the Orlik-Strauch induction of) each term $M_{w,I_1,I_2}$ injects in $\Pi^{\la}[m_\rho]$. 
First remark
\[ \dim \Hom_{U(\mathfrak g)}(M_{w,I_1,I_2},\Pi^{\la}[\mathfrak m_\rho])[\mathfrak m_{\delta_{\mathcal R}}] = 2^{|I_2|}m_\rho,\]
by the previous factorisation of $M_I(w \cdot \rho)$ for some $I \subset I_1$ and the computation using $\mathcal B_x$ 
(and using Proposition \ref{prop:geomXn=3}). Now each quotient of $M_{w,I_1,I_2}$ is of the form
\[ M_{w,I_1 \cup J,I_2 \backslash J} = \Boxtimes_{\tau \in I_1 \cup J} L(w_\tau \cdot \lambda_\tau) \boxtimes \Boxtimes_{\tau \in I_2 \backslash J} N(\lambda_\tau),\]
for some $J \subset I_2$ (remark that if $\tau \in I_2, w_\tau = 1$), and thus
\[ \dim \Hom_{U(\mathfrak g)}(M_{w,I_1\cup J,I_2 \backslash J},\Pi^{\la}[\mathfrak m_\rho])[\mathfrak m_{\delta_{\mathcal R}}] = 2^{|I_2|- |J|}m_\rho.\]
We deduce that the dimension of homomorphisms modulo those which factors through a strict quotient is
\[ 2^{|I_2|}m_\rho - \sum_{\emptyset \neq J \subset I_2} (-1)^{|J|+1}2^{|I_2|-|J|}m_\rho = m_\rho.\]
In particular there are $m_\rho$ independent injective maps from $\mathcal F_{\overline{B}}^{G}(M_{w,I_1,I_2}',\delta_{\mathcal R})$ to $\Pi^{\la}$.
Now when $w$ and $I_1,I_2$ varies, these objects have distincts irreducible in their socle. Thus the direct sum of all those maps
\[ \bigoplus_{w,I_1,I_2}  \mathcal F_{\overline{B}}^{G}(M_{w,I_1,I_2}',\delta_{\mathcal R})^{\oplus m_\rho} = \mathcal F_{\overline B}^{G}(S(\lambda,w_{\rho,\mathcal R})',\delta_{\mathcal R}\delta_{B}^{-1})^{\oplus m_\rho},\]
injects into $\Pi[\mathfrak m_\rho]$.\qedhere
\end{proof}

\begin{rema}\label{rema:loc_alg_not_in_socle}
In particular, for each $\tau$ such that $w_{\mathcal R,\tau} = 1$ we deduce the injection of the locally analytic representation
\[ \mathcal F_{\overline{B}_{\tau}}^{G_{\tau}}(N(\lambda_\tau)',\delta_{\mathcal R,\tau}\delta_{B_{\tau}}^{-1}) = [{\rm LA}_{s_1} \oplus {\rm LA}_{s_2} - {\rm LALG}],\]
as representation of $\GL_3(F_\tau)$ (acting through $\tau$), where
\[ {\rm LALG} := L(\lambda_\tau) \otimes_{L} \ind_{\overline{B}_{\tau}}^{G_{\tau}}(\delta_{\mathcal R,\tau}\delta_{B_{\tau}}^{-1}),\]
is an irreducible locally algebraic representation which appears in cosocle, where 
\[ {\rm LA}_s = \mathcal F_{\overline{B}_{\tau}}^{G_{\tau}}(L(s\cdot \lambda_\tau)',\delta_{\mathcal R,\tau}\delta_{B_{\tau}}^{-1})\]
is the irreducible, non-locally algebraic, socle of the locally analytic principal series
\[  {\rm LA}_s \subset \ind_{\overline{B}_{\tau}}^{G_{\tau}}(\delta_{s\cdot \lambda_\tau}\delta_{\mathcal R,\tau}\delta_{B_{\tau}}^{-1})^{\la} .\]
In this case, the locally algebraic representation ${\rm LALG}$ appears with multiplicity $m_\rho$ in the socle by the main result of 
$\cite{BHS3}$, but also with multiplicity $m_\rho$ as an higher order Jordan-Hölder factor, namely, in the cosocle of the previous $\mathcal F_{\overline{B}_{\tau}}^{G_{\tau}}(N(\lambda_\tau)',\delta_{\mathcal R,\tau}\delta_{B_{\tau}}^{-1})$.
\end{rema}

\section{Existence of very critical classical modular forms}
\label{sect:critforms}

In this section we show the existence of a classical form $f$ satisfying the hypothesis of Theorem \ref{thm:main}. The main difficulty is to find a form satisfying the Taylor-Wiles hypothesis, which is moreover completely critical at $p$ (i.e. $w_{\rho_f,\mathcal R}=1$).

For a finite extension $F$ of $\QQ_p$, we denote by
$\rec_F : F^\times\rightarrow {\rm Gal}_{F}^{\ab}$ the local reciprocity
map sending a uniformizer of $F$ on a geometric Frobenius. If $K$ is a
number field we denote by $\Art_K$ the Artin reciprocity map
$\mathbb{A}_K^\times/K^\times\rightarrow{\rm Gal}_{K}^{\ab}$ such that, for
any finite place $v$ of $K$ the precomposition of $\Art_K$ with the
inclusion $K_v^\times\hookrightarrow\mathbb{A}_K^\times$ is
$\rec_{K_v}$. If $\Psi$ is a character of
$\mathbb{A}_K^\times/K^\times$ and $v$ is a finite place of $K$ such
that $\Psi_v$ is unramified, we write $\Psi(v)$ for the evaluation of
$\Psi_v$ at an uniformizer of $F_v^\times$. First, we remark the following,
\begin{lemma}
Let $K /\QQ_p$ be a finite extension and let $\rho_p : {\rm Gal}_{K} \fleche \GL_n(\overline\QQ_p)$ be a crystalline representation with regular Hodge--Tate weights such that there exists a refinement $F_\bullet \subset D_{\rm cris}(\rho_p)$ which contains the Hodge filtration. We moreover assume that the eigenvalues of the linearization of the crystalline Frobenius on $D_{\rm cris}(\rho_p)$ are pairwise distinct. Then $\rho_p$ is a split sum of characters.
\end{lemma}

\begin{proof}
This is a simple application of weak admissibility. Up to extending scalars, we can assume that $D = D_{\rm cris}(\rho_p) = \bigoplus_\tau D_\tau$ is split, and is a filtered $\varphi$-module. We consider the linearization $\varphi^f_\tau$ of the Frobenius on $D_\tau$, where $f = [K_0:\QQ_p]$. We write ${\rm Fil}^{\bullet}D_\tau$ for the filtration on $D_\tau$ induced by the Hodge-filtration on $D$. 
 The assumption is that the Hodge filtration on $D$ is $\varphi$-stable i.e. there is a full flag of $K\otimes \overline\QQ_p$-modules $F_\bullet$, stable under $\varphi$, such that, for all $\tau$, if  
$k_1^\tau \leq \dots \leq k^\tau_n$ are the (opposite) $\tau$-Hodge-Tate weights (with multiplicities) then $F_{i,\tau} \subset {\rm Fil}^{k_{n-i+1}}D_\tau$. Denote the eigenvalues of $\varphi^f$ on $F_{i,\tau}$ by $(\varphi_1,\dots,\varphi_i)$. Thus by weak admissibility, 
\[\frac{1}{f}(v(\varphi_1) + \dots + v(\varphi_i)) \geq \sum_\tau \sum_{k = 1}^i k^\tau_{n+1-k}.\]
Now, if $G_i$ is a complementary $\varphi$-stable subspace of $F_i$ in $D$ (which exists due to the assumptions on the eigenvalues of $\varphi^f$), then we see directly that the $\tau$-Hodge-Tate weights of $G_i$ are $k^\tau_1,\dots,k^\tau_{n-i}$. Thus by weak admissibility again,
\[ \frac{1}{f}(v(\varphi_{i+1}) + \dots + v(\varphi_n)) \geq \sum_\tau \sum_{k = 1}^{n-i} k^\tau_{k}.\]
But by weak admissibility of $D$, the endpoints of both polygons gives
\[ \frac{1}{f}(v(\varphi_1)+\dots+v(\varphi_n)) = \sum_\tau \sum_i k^\tau_i.\]
Thus both $G_i$ and $F_i$ are weakly admissible, thus admissible, thus $\rho_p$ splits accordingly. As this is true for all $i$, we get the Lemma.
\end{proof}

It follows that, when $n=3$, an eigenform $f$ as in Theorem \ref{thm:main} has a split representation at $p$. In the case of modular forms, it was asked by Greenberg (see the work of Ghate and Vatsal \cite{Ghate}, \cite{Ghate_Vatsal}) if a cuspform whose representation is split at $p$ is necessarily a CM form. The natural generalization of this question to $\GL_3$ would suggest that we cannot find a form $f$ to apply Theorem \ref{thm:main} with very large image. Fortunately, we can construct an analog of a CM form for ${\rm GL}_3$ (more precisely for $U(3)$) which still has adequate image modulo $p$.

\subsection{Choosing a Hecke character}
\label{subsect:61}

Let $E$ be a CM field with totally real subfield $E^+ = F$ and let $F'$ be a totally real field disjoint
from $E$, such that $F'/\QQ$ is Galois and such that $[F':\QQ]=3$. Set $K = EF'$. This is a CM field. We moreover assume that all the ramified primes of $K/E$ lie above split primes in $E/E^+$.
Choose two distinct primes $p$ and $\ell$ such that $\ell$ is totally split in $K = EF'$ and primes above $p$ in $E^+=F$ are totally split in $K$. Moreover assume $p > 8 (= 2(n+1)$ when $n = 3$) and $\zeta_p \notin E$.

\begin{example} \begin{enumerate}
\item The easiest choice is $F' = \QQ(\zeta_7)^+$ and $E = \QQ(i\sqrt{3})$ so that $7$ is split in $E$.  
For this $F'$, we can also choose $E =\QQ(i,\sqrt{3})$, with maximal totally real subfield $E^+=\QQ(\sqrt{3})$ so that $E/E^+$ is unramified everywhere.
\item The second easiest choice for $F'$ is $F' = \QQ(\zeta_9)^+$. In this case we can choose $E = \QQ(i\sqrt{5})$.
\item If $E = \QQ(i)$, we can choose $F' = \QQ(\alpha)$ with $\alpha$ a root of $X^3-X^2-4X-1$, which has discriminant $13^2$.
\item If $F' = \QQ(\alpha)$ and $E = \QQ(i)$, we can choose any prime $p>8,\ell$ congruent to $1,5,21,25 \pmod{52}$, like $5,53,73,...$. In particular in that case we better should exclude $p=13$ as in the early version \cite{BellNewex} (who knows?). 
\item If $F' = \QQ(\zeta_7)^+$ and $E = \QQ(i\sqrt{3})$, we can choose any prime congruent to $1,13 \pmod{21}$ like $13,43,97...$.
\item If $F' = \QQ(\zeta_7)^+$ and $E = \QQ(i,\sqrt{3})$, we can take any prime $\ell \equiv 1,13 \pmod{84}$ like 
$13,97,169...$ and $p \equiv 1,13 \pmod{21}$ like $13,43,97...$.
\item If we really want to use $p=13$ and that $p=13$ is inert in $F=E^+$, and if we want moreover $E/E^+$ to be unramified 
everywhere, we can choose $E = \QQ(i,\sqrt{7})$ with $F' = \QQ(\beta) \subset \QQ(\zeta_{43})$ as $43$ is split in $\QQ(i,\sqrt{7})/\QQ(\sqrt{7})$, with $\beta$ a root of $X^3-X^2-14X-8$. 
\end{enumerate}
\end{example}

In the following we say that a weight $\underline k\in\ZZ^{\Hom(K,\CC)}$ is \emph{very
  regular} if, for $\tau_1\neq\tau_2$ in $\Hom(K,\CC)$, we have
$\vabs{k_{\tau_1}-k_{\tau_2}}\geq2$.

Let $\Psi$ be an algebraic Hecke character of $\mathbb A_K^\times$ with algebraic very regular weight $\underline k=(k_v)_{v|\infty}$, such that $\Psi^c = \Psi^\vee$ and such that $\Psi$ is
unramified both at $p$ and $\ell$. Choose an isomorphism $\iota : \CC \simeq \overline{\QQ_p}$. We moreover assume that 
\begin{description}
\item[$(\Psi,p)$]  if $\mathfrak p | p$ in $E$, we have $\Psi(v)\Psi(v')^{-1}\notin\set{1,p}$ for $v\neq v'$ places of $K$ dividing $\mathfrak p$.
\end{description}

\begin{description}
\item[$(\Psi,\ell)$] \label{eq:hypell} There exists $\lambda | \ell$ in $E$, and $\lambda' | \lambda$ in $E(\zeta_p)$, such that for all $v_1 \neq v_2$ places of $K$ dividing $\lambda$, if $v_1',v_2'$ are the corresponding places above $\lambda'$ in $K(\zeta_p)$, $\iota(\Psi(v_1') \pmod{\mathfrak m_{\overline \QQ_p}} \neq \iota(\Psi(v_2')) \pmod{\mathfrak m_{\overline \QQ_p}}$.
\end{description}

Consider moreover the following hypothesis on $\Psi$ :
\begin{description}
\item[$(\Psi,Ram)$] If $v$ is a place of $K$ such that $\Psi$ is
  ramified at $v$, then $v$ divides a
  prime which is totally split in $K/\QQ$.
\end{description}

Let $\Psi_p : \mathbb A_K^\times \fleche \overline \QQ_p^\times$ be
the $p$-adic realization of $\Psi$ and $\iota$, and $\psi_p :
{\rm Gal}_{K}\rightarrow\overline\QQ_p^\times$ such that
$\psi_p=\Psi_p\circ\Art_K$. It is a Galois representation satisfying
$\psi_p^\vee = \psi_p^c$.

\subsection{Galois induction} 
\begin{defin}\label{def galois induction} We denote by $\rho$ the induced Galois representation
\[ \rho = \ind_{{\rm Gal}_K}^{{\rm Gal}_E} \psi_p = \{ f : {\rm Gal}_E \fleche \overline{\ZZ_p}^\times | f(gk) = \psi_p^{-1}(k)f(g) \forall g \in {\rm Gal}_E, k \in {\rm Gal}_K\},\]
where the action of $g \in {\rm Gal}_E$ is given by $(g \cdot f)(x) = f(g^{-1}x)$.
\end{defin}

Then $\rho$ is a three dimensional Galois representation since $[K:E]$ is Galois of degree 3. We claim the following
\begin{lemma}\label{lemma:Ind}
\begin{enumerate}
\item \label{lemma:pt1} The representation $\overline\rho := \rho \otimes \overline{\mathbb F_p}$ is absolutely irreducible, in particular $\rho$ is absolutely irreducible. 
\item \label{lemma:pt2} The representation $\overline\rho(\Gal_{E(\zeta_p)})$ is adequate.
\item \label{lemma:pt3} The representation $\rho$ is polarized, i.e. $\rho^c \simeq \rho^\vee$.
\item\label{lemma:pt4} The representation $\rho_{\Gal_{E_v}}$ is split, $\varphi$-generic,
  Hodge--Tate regular for any $v | p$ in $E$,
\item\label{lemma:pt5} If $v$ is a place of $E$ such that $\rho$ is
  ramified at $v$, then $\Hom_{\Gal_{E_v}}(\rho_v,\rho_v(1))=0$.  
\end{enumerate}
\end{lemma}

\begin{proof}
We will actually prove that $\overline\rho({\rm Gal}_{E(\zeta_p)})$ acts absolutely irreducibly, which will imply point \ref{lemma:pt1} and point \ref{lemma:pt2} will follow by \cite{ThorneAut} Lemma 2.4. To prove point 1, remark that if we denote by $\sigma \in {\rm Gal}_E$ a lift of a generator of the Galois group $\Gal(K/E) = <\overline \sigma> = \ZZ/3\ZZ$, then $\rho$ has a basis given by $f, \sigma \cdot f, \sigma^{2} \cdot f$, where $f$ is the function
\[ f : {\rm Gal}_E = {\rm Gal}_K \coprod \sigma {\rm Gal}_K \coprod \sigma^2 {\rm Gal}_K \fleche
  \overline\ZZ_p^\times, k \in {\rm Gal}_K \mapsto \psi_p^{-1}(k), \sigma
  k,\sigma^2 k \mapsto 0.\]
Then $\sigma^3 \cdot f = \psi_p(\sigma^3) f.$ Thus, after restricting to ${\rm Gal}_K$, there is an isomorphism $\rho_{|{\rm Gal}_K} \simeq \psi_p \oplus \psi_p^\sigma \oplus \psi_p^{\sigma^2}$, where $\psi_p^{\sigma}=\psi_p(\sigma^{-1}\cdot\sigma)$. 
We reduce mod $p$, where we have a similar reduction after restricting to ${\rm Gal}_K$. Because of the hypothesis $(\Psi,\ell)$ away from $p$, we have that 
$\overline\rho_{{\rm Gal}_{E(\zeta_p)_{\lambda'}}}$, for $\lambda' | \ell$, is the sum of three
distinct characters. Moreover the group ${\rm Gal}_E$ acts transitively on
these three eigenspaces. Therefore this representation is absolutely irreducible.
To prove point \ref{lemma:pt3}, we compute $\rho^\vee$. By \cite[Prop.~10.28]{Curtis_Reiner_I}, we have an isomorphism
\[ \rho^\vee\simeq\Ind_{{\rm Gal}_K}^{{\rm Gal}_E}\psi_p^{-1}=\Ind_{{\rm Gal}_K}^{{\rm Gal}_E}\psi_p^c\simeq\rho^c. \]
Let us prove \ref{lemma:pt4}. As $p$ is totally split in $K/F$, we have for $v | p$ in $E$, ${\rm Gal}_{E_v}\subset {\rm Gal}_K$ so that $\rho_{|{\rm Gal}_{E_v}}\simeq\psi_{p,v}\oplus\psi_{p,v}^\sigma\oplus\psi_{p,v}^{\sigma^2}$. As the group ${\rm Gal}_E$ acts transitively on the three places of $K$ over $v$, we have $\rho_{|{\rm Gal}_{E_v}}\simeq\bigoplus_{v'|v}\psi_{p,v'}$. Therefore $\rho_{|{\rm Gal}_{E_v}}$ is crystalline and the eigenvalues of the Frobenius endomorphism of $D_{\cris}(\rho_{|{\rm Gal}_{E_v}})$ are the $\Psi(v')$ for $v' | v$ in $K$. It follows from hypothesis $(\Psi,p)$ that $\rho_{|{\rm Gal}_{E_v}}$ is $\varphi$-generic. Moreover the Hodge--Tate weights of $\rho_{|{\rm Gal}_{E_v}}$ corresponds to the algebraic (infinitesimal) weight of $\Psi$, which was assumed regular so that $\rho_{{\rm Gal}_{E_v}}$ is Hodge--Tate regular.

Finally we prove \ref{lemma:pt5}.  Let $v$ be a place of $E$ such that
$\rho_v$ is ramified. Then either $v$ is ramified in $K/E$ or $\Psi_v$
is ramified. Assume in a first time that $\Psi_v$ is ramified. Then
$(\Psi,Ram)$ implies that $v$ divides a prime of $\QQ$ which is
totally split in $K$. In particular, $v$ is split in $K/E$. As above, we have
$\rho_{v}\simeq\bigoplus_{v'|v}\psi_{p,v'}$ with
$\psi_{p,v'}=\Psi_{v'}\circ\rec_{K_{v'}}^{-1}$ as $v'\nmid
p$. Therefore it follows from Lemma \ref{lemm:forcing_Ram} below that
$\Hom_{{\rm Gal}_{E_v}}(\rho_v,\rho_v(1)) = 0$.

Now assume that $v$ is non split in $K$. As $K/E$ is Galois there is a
unique place $w$ of $K$ over $v$ and
$\rho_v\simeq\Ind_{{\rm Gal}_{K_w}}^{{\rm Gal}_{E_v}}\psi_{p,w}$. By Frobenius
reciprocity, we have
\[
  \Hom_{{\rm Gal}_{E_v}}(\rho_v,\rho_v(1))\simeq\Hom_{{\rm Gal}_{K_w}}(\psi_{p,w}\oplus\psi_{p,w}^\sigma\oplus\psi_{p,w}^{\sigma^2},\psi_{p,w}\chi_{\cyc|K_w}). \]
Assume that $\psi_{p,w}=\psi_{p,w}^\sigma\chi_{\cyc|K_w}$. As
$\chi_{\cyc|K_w}=\chi_{\cyc|K_w}^\sigma$, we deduce
$\psi_{p,w}^\sigma=\psi_{p,w}^{\sigma^2}\chi_{\cyc|K_w}$ and
$\psi_{p,w}^{\sigma^2}=\psi_{p,w}^{\sigma^3}\chi_{\cyc|K_w}=\psi_{p,w}\chi_{\cyc|K_w}$
so that $\psi_{p,w}=\psi_{p,w}\chi_{\cyc|K_w}^3$ which is
false. We prove similarly than
$\psi_{p,w}\neq\psi_{p,w}^{\sigma^2}\chi_{\cyc|K_w}$ and deduce
$\Hom_{{\rm Gal}_{E_v}}(\rho_v,\rho_v(1))=0$. If
$\Psi_w\neq\Psi_w\circ\sigma$, then the characters
$\Psi_w,\Psi_w\circ\sigma,\Psi\circ\sigma^2$ are pairwise distinct and
$\rho_v$ is irreducible so that $\Hom_{{\rm Gal}_{E_v}}(\rho_v,\rho_v(1))
=0$. If $\Psi_w=\Psi_w\circ\sigma$, then $\rho_v$ is not irreducible,
but clearly $\Hom_{{\rm Gal}_{E_v}}(\rho_v,\rho_v(1)) = 0$ (as
$\psi_{p|{\rm Gal}_{E_v}} \neq \psi_{p|{\rm Gal}_{E_v}}(1)$).
\end{proof}

\begin{lemma}\label{lemm:forcing_Ram}
  Let $\Psi : \mathbb A_K^\times/K^\times$ be an algebraic Hecke
  character of very regular weight $\underline k$. Then, if $\ell$ is
  a prime number which is totally split in $K$, then
  $\Psi_v\neq\Psi_w\vabs{\cdot}_w$ for all places $v,w$ of $K$
  dividing $\ell$.
\end{lemma}

\begin{proof}
  Let $\Psi$ and $\ell$ be as in the statement. Fix
  $\iota : \CC\simeq\overline\QQ_\ell$ and let $\vabs{\cdot}_\ell$ be the unique
  absolute value on $\overline\QQ_\ell$ extending the one on $\QQ_\ell$. Let $\Psi_\iota$ be the
  continuous character
  $\mathbb A_K^\times/K^\times
  K_\infty^\times\rightarrow\overline\QQ_\ell^\times$ defined by
  \[
  \Psi_{\iota,w}(x_w)=\begin{cases} \Psi_w(x_w)& \text{if}\ w\not| \ell, w\not|\infty \\ 1 & \text{if}\ w|\infty \\  \iota(\Psi_w(w_w))\prod_{{\tau 
        \in\Hom(K_w,\overline\QQ_\ell),
        \tau|w}}\tau(x_w)^{k_{\iota^{-1}\tau}} & \text{if}\ w|\ell,\end{cases}
  \]
  where $\tau | w$
  means that $|.|_{\ell} \circ\tau$ extends the absolute value given
  by $w$ on $K$, and $(k_\sigma)_{\sigma \in \Hom(K,\CC)}$ is the
  weight of $\Psi$.  As the group
  $\mathbb A_K^\times/K^\times K_\infty^\times$ is compact, we have
  $\Im(\Psi_\iota)\subset\overline \ZZ_\ell^\times$. As $\ell$ is
  totally split $\iota$ induces a bijection between $\set{ v | \ell}$
  and $\Hom(K,\CC)$. Let $v$ be a place of $K$ dividing $\ell$
  corresponding to $\tau$ (i.e. $|.|_\ell \circ \iota^{-1}\tau$
  extends $|.|_v$) and denote $k_v\coloneqq k_{\iota^{-1}\tau}$. We
  have
  \[
    \vabs{\iota{\Psi_v(\ell)}\tau(\ell)^{k_v}}_\ell=1 \] so that
  $\vabs{\iota(\Psi_v(\ell))}=l^{k_v}$. As $\ell$ is a uniformizer of
  $K_v$, for any $v|\ell$, the result follows.
\end{proof}

\subsection{Construction of an explicit set of Hecke characters}

In this subsection we explain one way to find a $\Psi$ as before, satisfying hypothesis $(\Psi,p),(\Psi,\ell),(\Psi,Ram)$. Fix $E$ a CM extension, with $E^+=F$ its maximal totally real subfield, so that $[E:E^+] = 2$. Fix also $F'$ disjoint from $E$, a totally real degree 3 Galois extension of $\QQ$. Choose $p,\ell$ two primes with are totally split in $K := EF'$ such that $p>8$.
The following Lemma is a more precise version of
\cite[Lem.~4.1.1]{CHT}.

\begin{lemma}\label{lemm:ext_Dirichlet_char}
  Let $F$ be a number field. Let $S$ be a finite set of places of
  $F$. Let $\chi_S$ be an unramified continuous character $F_S^\times\coloneqq\prod_{v\in S}F_v^\times\rightarrow \CC^\times$ of finite order. Let $T$ be a set of finite places of $F$, disjoint
  from $S$ and of Dirichlet density $1$. Then there exists a
  continuous character
  $\chi : \mathbb A_F^\times/F^\times\rightarrow\CC^\times$ of finite
  order such that $\chi_{|F_S^\times}=\chi_S$ and the ramification
  places of $\chi$ are in $T$.
\end{lemma}

\begin{proof}
  Let $U^S$ be the product of the $\cO_{F_v}^\times$ for $v\notin S$. Then
  $F^\times\cap U^S$ is a finitely generated subgroup of
  $F^\times$. Let us write $m$ for the order of the finite cyclic group
  $\chi_S(F^\times\cap U^S)$. It follows from the proof of Theorem 1
  in \cite{Chevalley_2} that we can find finitely many places
  $w_1,\dots,w_r$ in $T$ such that the subgroup of $F^\times\cap U^S$
  congruent to $1$ modulo $\frakp_{w_1},\dots,\frakp_{w_r}$ is
  contained in $(F^\times\cap U^S)^m$. We conclude as in the proof of
  \cite[Lem.~4.1.1]{CHT} choosing for $U$ the product of the $U_v$ for
  $v$ not in $S$ nor $\set{w_1,\dots,w_r}$ and a small enough subgroup
  at $w_1,\dots,w_r$.
\end{proof}

\begin{lemma}\label{lemm:ext_Dir_char}
  Let $K$ be an (imaginary) CM field with totally real subfield $K^+$ 
  and complex conjugacy $c$. Denote $\psi : \mathbb{A}_K^\times/K^\times\rightarrow\CC^\times$ be a
  continuous character. Assume that there exists a finite set $S$ of
  places of $K$  which are split in $K/K^+$ and such that $\psi_v^{-1}=\psi_{cv}$ for
  $v\in S$. Moreover, assume that $S$ contains the Archimedean
  places. Let $T$ be a finite set of places of $K$ that contains $S$ and is stable under $c$,
  such that $\psi$ is unramified outside of $T$. Then there exists a Hecke
  character
  $\widetilde \psi :
  \mathbb{A}_K^\times/K^\times\rightarrow\CC^\times$ such that
  $\widetilde \psi^{-1}=\widetilde \psi^c$ and  $\widetilde \psi_v=\psi_v$
  for $v\in S$ and such that $\widetilde \psi_v$ is unramified outside of $T$.
\end{lemma}

\begin{proof}
  Let $\theta=\psi\circ N_{K/K^+}$. As $S$ contains the Archimedean
  places, the character $\theta$ is trivial at Archimedean places and is therefore a
  character of finite order. Let
  $U_T\subset\prod_{v\in T\setminus S}K_v^\times$ be a compact open
  subgroup such that $\theta_{|U_T}$ is trivial and such that
  $c(U_T)=U_T$. Let
  \[ U=\big(\prod_{v\notin T}\cO_{K_v}^\times\big)\cdot U_T\cdot\big(\prod_{v\in
      S}K_v^\times\big). \] We have an injection of compact groups
  \[
    N_{K/K^+}(\mathbb{A}_K^\times)/(N_{K/K^+}(\mathbb{A}_K^\times)\cap
    K^\times U)\hookrightarrow\mathbb{A}_K^\times/K^\times U. \] Under
  our hypothesis, the character
  $\psi_{|N_{K/K^+}(\mathbb A_K^\times)}$
 is trivial on
  $(N_{K/K^+}(\mathbb{A}_K^\times)\cap K^\times U)$. Therefore it
  extends to a character $\alpha$ of finite order of
  $\mathbb A_K^\times$ trivial on $K^\times U$. We thus have 
  $\psi\circ N_{K/K^+}=\alpha\circ N_{K/K^+}$. It is easy to check that
  the character $\widetilde \psi=\psi\alpha^{-1}$ satisfies our
  requirements.
\end{proof}

\begin{prop}\label{prop:exist_char}
  For each choice of fields $E$ and $F'$ and places $p$ and $\ell$ and very regular weight
  $\underline k$ as above there exists a Hecke character
  $\Psi : \mathbb A_K^\times/K^\times\rightarrow\CC^\times$ satisfying
  $(\Psi,p)$, $(\Psi,\ell)$ and $(\Psi,Ram)$ and such that
  $\Psi^{-1}=\Psi^c$.
\end{prop}

\begin{proof}
  Let $\underline k$ be a very regular weight. It follows from
  \cite{Schappacher}, Section 0.3, that there exists a Hecke character
  $\Psi_0$ of $\mathbb A_K^\times/K^\times$ with weight $\underline
  k$. Using Lemma \ref{lemm:ext_Dirichlet_char}, we can construct a
  Hecke character $\theta$ of finite order such that, setting
  $\Psi_1=\Psi_0\theta$, we have
  \begin{itemize}
  \item the character $\Psi_1$ satisfies $(\Psi_1,p)$ and $(\Psi_1,\ell)$ ;
  \item there exists finitely many primes $\ell_1,\dots,\ell_r$,
    different from $p$ and $\ell$, which are totally split in $K$ and
    such that $\Psi_1$ is only ramified at places dividing
    $\ell_1,\dots,\ell_r$ ;
  \item we have $\Psi_{1,w}^{-1}=\Psi_{1,cw}$ for any place $w$ of $K$
    dividing $\ell$ or $p$.
  \end{itemize}
  Now it follows from Lemma \ref{lemm:ext_Dir_char} that there exists
  a Hecke character $\Psi$ of $\mathbb A_K^\times/K^\times$ such that
  \begin{itemize}
  \item $\Psi^{-1}=\Psi^c$ ;
  \item $\Psi_v=\Psi_{1,v}$ if $v$ is a place of $K$ dividing $p$ or
    $\ell$ ;
  \item $\Psi$ is ramified only at places dividing $\ell_1,\dots,\ell_r$.\qedhere
  \end{itemize}
\end{proof}

\subsection{Automorphic Induction and base change}

Let $\Psi$ and $\rho$ as in subsection \ref{subsect:61} and let $U$ denote the unitary group in three variables for $E/E^+$ that is compact at infinity and quasi-split at all finite places. We need to find an automorphic form for $U$ whose associated Galois representation is induced representation $\rho$ from \ref{def galois induction}.

\begin{prop}
There exists an automorphic representation $\Pi$ of $\GL_{3,E}$, cuspidal, cohomological at infinity, unramified at $\ell$ and $p$, polarized, whose associated Galois representation is 
given by $\rho$.
\end{prop}

\begin{proof}
This is the content of \cite{HenniartInduction} Théorème 3 (as $K/E$ cyclic of degree 3) for the existence of the automorphic representation, Théorème 5 for the compatibility with the local correspondence at $\ell$ and $p$ and at infinity (cf. the following remark of \cite{HenniartInduction}). Polarization can be checked after base change of the automorphic induction to $K$, where it follows as $\Psi^c = \Psi^\vee$, and as $\Psi \neq \Psi^\sigma$ for $\sigma \in \Gal(K/E)$ such that $\sigma\neq 1$. Moreover, the automorphic induction is also cuspidal (Theorem 2 of \cite{HenniartInduction}).
\end{proof}

\begin{conjecture}
\label{conj;unitaryautodescent}
There exists a cohomological, cuspidal, automorphic representation $\pi$ of $U$ whose base change to $\GL_{3,E}$ is $\Pi$.
\end{conjecture}

\begin{prop}
If $E/E^+$ is everywhere unramified (e.g. for $E = \QQ(i,\sqrt{3})$ or $\QQ(i,\sqrt{7})$), then the previous conjecture is true.
\end{prop}

\begin{proof}
This is \cite{Labesse} Theorem 5.4.
\end{proof}

\begin{prop}
If $E$ is quadratic imaginary, then the previous conjecture is true.
\end{prop}

\begin{proof}
By \cite{Morel} Corollary 8.5.3 (ii), there exists $\pi'$ an automorphic representation for $GU(3)$ associated to $\Pi \times 1$, which is automorphic for $\GL_3 \times \GL_1$. By \cite{HS} Lemma A.7 (based on \cite{HarrisTaylor}), there exists $\pi$, an automorphic representation of $U(3)$ associated to $\pi'$.
\end{proof}

\begin{cor}
\label{cor:exofverycriticalforms}
If $E$ is quadratic imaginary or if $E/E^+$ is everywhere unramified, then there exists a classical form on $U(3)$ satisfying the hypothesis of Theorem \ref{thm:main}.
\end{cor}

\begin{proof}
Let $\pi$ be the automorphic representation of $U$ considered above, and let $f \in \pi$ be an eigenform for the Hecke operators away from a set $S$ of bad places of $\pi$.
Then $\rho_f = \rho_\pi = \rho$ is crystalline at $p$ and $\varphi$-generic. In particular it has $3! = 6$ refinements which are automorphic and split at $p$. Hence there exists an automorphic refinement $\mathcal R$ of $f$ with relative position $w_{\mathcal R} = 1$ with respect to the Hodge filtration. In particular, for this choice of a refinement, there exists a refined classical modular form $f'$ satisfying all hypothesis of Theorem \ref{thm:main}. But, by Lemma \ref{lemma:Ind}(5) we know that $f$ gives, for all $v \in S \backslash S_p$, a point of $\mathcal X_{\overline{\rho_v}}^\square$ which satisfies $\Hom_{{\rm Gal}_{E_v}}(\rho_v,\rho_v(1)) = 0$. When $v$ splits in $E/E^+$, such a $v$ is a smooth point by \cite{All1} Prop 1.2.2. 
\end{proof}

\bibliographystyle{smfalpha}
\bibliography{Biblio}

\printindex
\end{document}